\documentclass[11pt,a4paper,reqno]{amsart}
\usepackage[english]{babel}
\usepackage[T1]{fontenc}
\usepackage{verbatim}
\usepackage{palatino}
\usepackage{amsmath}
\usepackage{mathabx}
\usepackage{amssymb}
\usepackage{amsthm}
\usepackage{amsfonts}
\usepackage{graphicx}
\usepackage{esint}
\usepackage{color}
\usepackage{mathtools}
\usepackage{overpic}
\usepackage{pdfpages}

\usepackage[colorlinks = true, citecolor = black, linkcolor = black, urlcolor = black]{hyperref}
\author[Algom, Rodriguez Hertz, and Wang]{Amir Algom, Federico Rodriguez Hertz, and Zhiren Wang}
	
	\address{Department of Mathematics, University of Haifa at Oranim, Tivon 36006, Israel}

\email{\href{mailto:amir.algom@math.haifa.ac.il}{amir.algom@math.haifa.ac.il}}

\address{Department of Mathematics\\
	The Pennsylvania State University\\
	107 McAllister Building\\
	University Park, PA 16802\\
	USA}

\email{\href{mailto:fjr11@psu.edu}{fjr11@psu.edu}}

\address{Department of Mathematics\\
	Johns Hopkins University\\
	3400 N. Charles Street\\
	Baltimore, MD 21218\\
	USA}
	
	\email{\href{mailto:zhirenw@jhu.edu}{zhirenw@jhu.edu}}

\title[Absolutely continuous convolutions and projections]{Absolutely Continuous Convolutions and Projections of Fractal Measures}
\date{\today}
\thanks{A.A. is supported by  the Israel Science Foundation (Grant No. 392/25),   NSF-BSF Grant No. 2024692, and  Grant No. 2022034 from the United States-Israel Binational Science Foundation (BSF), Jerusalem, Israel. \newline
F. RH. is partially supported by NSF Grant No.~2453688 and by the Anatole Katok Chair in Mathematics. \newline
Z.W is partially supported by NSF grant No. 2453689.}

\newcommand{\R}{\mathbb{R}}

\newcommand{\N}{\mathbb{N}}

\newcommand{\C}{\mathbb{C}}
\newcommand{\Z}{\mathbb{Z}}

\newcommand{\supp}{\operatorname{supp}}

\newcommand{\diam}{\operatorname{diam}}

\newcommand{\dist}{\operatorname{dist}}

\newcommand{\sgn}{\operatorname{sgn}}

\def\Barint_#1{\mathchoice
          {\mathop{\vrule width 6pt height 3 pt depth -2.5pt
                  \kern -8pt \intop}\nolimits_{#1}}%
          {\mathop{\vrule width 5pt height 3 pt depth -2.6pt
                  \kern -6pt \intop}\nolimits_{#1}}%
          {\mathop{\vrule width 5pt height 3 pt depth -2.6pt
                  \kern -6pt \intop}\nolimits_{#1}}%
          {\mathop{\vrule width 5pt height 3 pt depth -2.6pt
                  \kern -6pt \intop}\nolimits_{#1}}}

\numberwithin{equation}{section}

\theoremstyle{plain}
\newtheorem{thm}{Theorem}
\numberwithin{thm}{section}
\newtheorem*{"thm"}{"Theorem"}

\newtheorem{lemma}[thm]{Lemma}

\newtheorem{cor}[thm]{Corollary}
\newtheorem{proposition}[thm]{Proposition}
\newtheorem{"proposition"}[thm]{"Proposition"}
\newtheorem{"lemma"}[thm]{"Lemma"}

\theoremstyle{definition}

\theoremstyle{remark}
\newtheorem{remark}[thm]{Remark}

\newcommand{\nref}[1]{(\hyperref[#1]{#1})}

\DeclareMathSymbol{\intop}  {\mathop}{mathx}{"B3}

\newcommand{\Leb}{\mathcal{L}}
\newcommand{\Pone}{\mathbb{P}^{1}}
\newcommand{\Sone}{\mathbb{S}^{1}}
\newcommand{\dimFr}{\dim_{\mathrm{Fr}}}
\newcommand{\push}[1]{(#1)_{\#}}

\begin{document}

\begin{abstract}
We develop a general criterion for establishing absolute continuity of convolutions of fractal measures on the line, and more generally of prescribed line projections of fractal measures in the plane. The crucial new ingredient is  exact scaling covariance of the \(L^1\)-norm of Littlewood--Paley pieces of the projected measure, under affine renormalization. We apply this criterion in three key settings. First, we show that the convolution of two measures on the line is absolutely continuous whenever their dimensions sum to more than one, provided one is a self-conformal measure whose defining IFS is not \(C^2\)-conjugate to linear, and the other is either self-conformal or Ahlfors--David regular. Second, we show that every line projection of a planar complex-analytic self-conformal measure of dimension greater than one is absolutely continuous, under natural nonlinearity and nondegeneracy assumptions. Finally, we show that every line projection of a planar self-affine measure is absolutely continuous, under natural irreducibility and proximality assumptions, whenever the correlation dimension of the measure and the Frostman dimension of its Furstenberg measure sum to more than two.

\end{abstract}

\maketitle

\section{Introduction}\label{sec:introduction}
\subsection{Background}
Marstrand's projection theorem is one of the fundamental results  of geometric measure theory.  It states that if $\mu \in \mathcal{P}(\mathbb{R}^2)$, i.e.  a  Borel probability measure on $\R^2$, then, writing $\pi_e(x)=\langle x,e\rangle$ and $\dim \mu :=\inf \lbrace \dim_H A:\, \mu(A)>0\rbrace$,
\begin{equation}\label{eq:MM-dimension}
 \dim \push{\pi_e}\mu=\min\{1,\dim \mu\}, \text{ for Lebesgue-almost every direction } e\in \Sone.
\end{equation}
Moreover, 
\begin{equation}\label{eq:MM-ac}
 \push{\pi_e}\mu\ll \Leb^1 \text{ for Lebesgue-almost every direction } e\in \Sone 
 \text{ if }  \dim \mu >1.
\end{equation}
See \cite[Chapter 9]{mattila1999geometry}, which also discusses Mattila's extension to $\mathbb{R}^d,d>2$. 

For dynamically defined measures one expects sharper results. A guiding principle,   influenced by  Furstenberg's conjectures  around the $\times 2, \times 3$ Conjecture \cite{furstenberg1967disjointness, furstenberg1970intersections}, is that if a projection violates \eqref{eq:MM-dimension} or \eqref{eq:MM-ac}, then it should represent some resonance between the dynamics and that specific projection. Thus, in the absence of  algebraic or dynamical obstructions, one  expects \eqref{eq:MM-dimension} and \eqref{eq:MM-ac} to hold in a given direction. Indeed, this heuristic has been demonstrated for the dimension part \eqref{eq:MM-dimension}, in a large class of examples, as we recall below. Much less is known, however, regarding the absolute continuity counterpart \eqref{eq:MM-ac}.

In this paper we  make significant progress on this absolute-continuity problem.  We establish a general criterion, Theorem \ref{thm:stopped-intro} below, that ensures a \emph{given} projection of a fractal measure on $\mathbb{R}^2$ is absolutely continuous. We then apply it to three important classes of examples: convolutions of fractal measures on the line, line projections of self-conformal measures on the plane, and line projections of self-affine measures on the plane. We first state these three applications, together with their background, and then formulate the main technical theorem \ref{thm:stopped-intro} from which they all follow.

\subsection{Absolute continuity of convolutions}\label{subsec:convolutions-intro}
We begin with convolutions of fractal measures $\mu,\nu\in\mathcal P(\mathbb R)$ in the supercritical regime; that is, when
$
\dim \mu+\dim \nu>1.
$
The terminology comes from viewing the convolution as a projection of the product measure $\mu\times\nu\in\mathcal P(\mathbb R^2)$. Indeed, up to a harmless linear rescaling, $\mu*\nu$ is the projection of $\mu\times\nu$ in the direction
$
e=\frac{1}{\sqrt2}(1,1)\in\Sone.
$
Moreover, standard  estimates (see e.g. \cite[Chapter 7]{mattila1999geometry}) give
$
\dim (\mu\times\nu)\geq\dim \mu+\dim \nu>1.
$
Thus, by \eqref{eq:MM-ac}, almost every line projection of $\mu\times\nu$ is absolutely continuous. For dynamically defined measures, one is therefore led to expect that  $\mu*\nu$ is also absolutely continuous, unless there is a resonance between the structures of $\mu$ and $\nu$.

We will show that this is indeed the case for a broad class of dynamically defined measures. Our mechanism for ruling out such resonances is to require one of the two measures to arise from a sufficiently nonlinear dynamical system. We consider self-conformal measures, which serve as an axiomatic model for more general stationary measures. To define them, let $I\subset\R$ be a compact interval. A \emph{$C^2$ iterated function system}, abbreviated as a $C^2$ IFS, is a finite family
$
\Phi=\{f_i:i\in\mathcal A\}
$
of non-singular $C^2$ contractions from a neighbourhood of $I$ into $I$ such that
\begin{equation}\label{eq:1d-uniform-contraction}
0<\rho_-\leq |f_i'(x)|\leq\rho_+<1,
\,i\in\mathcal A,\ x\in I.
\end{equation}
It is well known that there exists a unique non-empty compact set $K=K_\Phi$ satisfying
\begin{equation}\label{eq:attractor}
K=\bigcup_{i\in\mathcal A}f_i(K).
\end{equation}
We call $K$ the \emph{attractor} of $\Phi$, or the corresponding \emph{self-conformal set}. If $\Phi$ contains only affine maps, we call it \emph{self-similar} and $K$ a \emph{self-similar set}. Given a strictly positive probability vector $p=(p_i)_{i\in\mathcal A}$, the associated \emph{self-conformal measure} is the unique probability measure $\nu \in \mathcal{P}(K)$ satisfying the stationarity relation
\begin{equation} \label{eq:self-conformal}
\nu=\sum_{i\in\mathcal A}p_i(f_i)_\#\nu.
\end{equation}
If $\Phi$ is self-similar, we call $\nu$ a \emph{self-similar measure}. We say that $\Phi$ is \emph{linear} if
$
f_i''(x)=0
\text{ for every }i\in\mathcal A\text{ and }x\in K_\Phi.
$
Finally, we say that a $C^2$ IFS $\Phi$ is $C^2$-conjugate to  linear,  if there exists a $C^2$ diffeomorphism $g$ such that
$
\{g\circ f_i\circ g^{-1}:i\in\mathcal A\}
$
is linear.

This notion of linearity originates from our work on the Fourier decay problem \cite{algom2021decay,algom2023polynomial}. In general, linearity and self-similarity are distinct notions. In the real-analytic category, however, they coincide up to analytic conjugacy: a real-analytic IFS on $\mathbb R$ which is $C^2$-conjugate to linear is in fact real-analytically conjugate to a self-similar IFS. In finite smoothness the distinction is genuine. In \cite{algom2024linear}, together with Ben Ovadia and Shannon, we constructed linear IFSs which are not smoothly conjugate to self-similar ones, and in \cite{AlgomBenOvadiaRHWShannon2026} we constructed a $C^\infty$ example of this type admitting a stationary measure which is not even Rajchman. We also recall that \cite{algom2023polynomial} gives explicit sufficient conditions for a $C^2$ IFS not to be $C^2$-conjugate to linear; in particular, in that case every non-atomic self-conformal measure associated to it has positive Fourier dimension.

Finally, recall that a measure $\nu\in\mathcal P(\mathbb R)$ is called \emph{Ahlfors--David regular}, abbreviated as AD regular, if there exists $C\geq1$ and $\alpha>0$ such that
$$
C^{-1}r^\alpha\leq\nu(B(x,r))\leq Cr^\alpha
\qquad
(x\in\supp\nu,\ 0<r\leq\diam(\supp\nu)).
$$

Our first theorem is the following.

\begin{thm}\label{thm:convolution-intro}
Let $\mu,\nu\in\mathcal P(\mathbb R)$. Assume that $\mu$ is a self-conformal measure for a $C^2$ IFS which is not $C^2$-conjugate to linear. Assume in addition that one of the following holds:
\begin{enumerate}
\item[(A)] $\nu$ is a self-conformal measure for a $C^2$ IFS; or
\item[(B)] $\nu$ is AD regular.
\end{enumerate}
If
$
\dim \mu+\dim \nu>1
$
then
$
\mu*\nu\ll\Leb^1.
$
\end{thm}
The hypotheses of Theorem~\ref{thm:convolution-intro} are easy to verify in many concrete examples; this is particularly so in the real-analytic category. We give  explicit examples in Subsection~\ref{subsec:explicit-convolution}. The conclusion of the theorem also holds under every nonzero rescaling of either factor. Indeed, if $S_t(x)=tx$, then for every $t \neq 0$ the measure $(S_t)_{\#}\nu$ satisfies the same  hypotheses as $\nu$, and has the same dimension. Hence 
$
\mu*(S_t)_\#\nu \ll\Leb^1
\, \text{ for every }t\neq0.$
In particular,
$$
\Leb^1\bigl(\supp(\mu)+t\supp(\nu)\bigr)>0
\,\text{ for every }t\neq0.
$$

The corresponding  problem of establishing the dimension of the convolution or arithmetic sum of the supports is by now quite well understood. In the self-similar setting, Peres and Shmerkin \cite{Peres2009Shmerkin} proved the expected dimension formula under arithmetic non-resonance: if $\Phi,\Psi$ are self-similar IFSs and
$
\frac{\log |f_i'|}{\log |g_j'|}\notin\mathbb Q
$
for some $f_i\in\Phi$ and $g_j\in\Psi$, then, writing $K,K'$ for their attractors,
\begin{equation}\label{eq:dim}
\dim_H(K+tK')
=
\min\{1,\dim_H K+\dim_H K'\}
\, \text{ for every }t\neq0.
\end{equation}
Hochman and Shmerkin subsequently developed a broad measure-theoretic framework unifying much of this dimension theory \cite{hochman2009local}; in particular, their work resolved Furstenberg's  sumset conjecture, a version of \eqref{eq:dim} for $\times m$- and $\times n$-invariant sets. In the strongly separated self-conformal setting relevant to Theorem~\ref{thm:convolution-intro}(A), their results also give the measure-theoretic counterpart
\begin{equation}\label{eq:dim-measures}
\dim \bigl(\mu*(S_t)_{\#}\nu\bigr)
=
\min\{1,\dim \mu+\dim \nu\},
\, t\neq0.
\end{equation}
Finer information on the $L^q$ dimensions of such convolutions was later obtained by Shmerkin \cite{shmerkin2016furstenberg} as part of his resolution of Furstenberg's intersection conjecture; see also Wu \cite{wu2016proof}. More recently, B\'ar\'any, K\"aenm\"aki, Py\"or\"al\"a and Wu \cite{barany2023scaling} proved versions of \eqref{eq:dim-measures} under natural arithmetic assumptions without separation conditions, while Py\"or\"al\"a \cite{Pyorala2025Dissonance} obtained a particularly close dimension-theoretic counterpart of Theorem~\ref{thm:convolution-intro}: one factor may arise from a real-analytic self-conformal IFS which is not conjugate to linear, and the other may be either self-conformal or AD regular, again without separation assumptions.

An important predecessor to this  theory is the work of Moreira and Moreira--Yoccoz on arithmetic sums of dynamically defined Cantor sets, originating in the Palis program on homoclinic bifurcations. Moreira announced in \cite{Moreira1998sums}, and later proved in the more general setting of \cite{Moreira2023images}, that if $K,K'\subset\R$ are attractors of $C^2$ IFSs satisfying the strong separation condition, and one of the defining IFSs is not $C^2$-conjugate to linear, then \eqref{eq:dim} holds. The underlying methods grew out of the theory of stable intersections developed by Moreira and Yoccoz \cite{carlos2001stable}. In particular, they proved Palis' conjecture that, generically, when $\dim_H K+\dim_H K'>1$, suitable translations of such Cantor sets have stable intersection, or equivalently their arithmetic difference contains an interval. The conclusion of Theorem~\ref{thm:convolution-intro} is related but of a different nature: for every fixed pair satisfying its hypotheses, it gives positive Lebesgue measure of the corresponding arithmetic sumset for every nonzero scaling.

A particularly relevant work is due to Nazarov, Peres and Shmerkin \cite{nazarov2012peres}. Let $\mu_a$ and $\mu_b$ be the equal-weight Cantor--Lebesgue measures associated to the IFSs
$
\Phi_a=\{ax,1-a+ax\},
\,
\Phi_b=\{bx,1-b+bx\}.
$
Recall that the lower correlation, or $L^2$, dimension of a measure $\eta\in\mathcal P(\mathbb R)$ is defined by
\begin{equation}\label{eq:def-correlation-dim}
\dim_2\eta
:=
\sup\bigl\{s:I_s(\eta)<\infty\bigr\},\,
\text{ where }
I_s(\eta):=
\iint |x-y|^{-s}\,d\eta(x)\,d\eta(y).
\end{equation}
When $\log a/\log b\notin\mathbb Q$, Nazarov, Peres and Shmerkin proved that, for every $t\neq0$,
$
\dim_2\bigl(\mu_a*(S_t)_\#\mu_b\bigr)
=
\min\{1,\dim_2\mu_a+\dim_2\mu_b\}.
$

Crucially, \cite{nazarov2012peres} also shows the analogous statement for absolute continuity is \emph{false}. For $a=1/4$ and $b=1/3$, the convolution is supercritical and $\log a/\log b\notin\mathbb Q$, yet
$
\mu_{1/4}*(S_t)_\#\mu_{1/3}
$
is singular for a dense $G_\delta$ set of $t$. Hence every non-coordinate projection of $\mu_{1/4}\times\mu_{1/3}$ has full dimension, while a dense set of them are singular. This shows that arithmetic non-resonance alone cannot force absolute continuity, and highlights the essential role of the nonlinearity hypothesis in Theorem~\ref{thm:convolution-intro}.

More broadly, there is a substantial theory showing that  parametrized families of convolutions of dynamically defined measures are absolutely continuous in the super-critical regime for almost every parameter, typically after excluding a possibly nontrivial exceptional set. This includes the work of Peres--Solomyak \cite{Peres1998Solomyak} and the parameter-exclusion method of Peres--Schlag \cite{Peres200Schlag}; see also \cite{Shmerkin2016Solomyak,shmerkin2016furstenberg,barany2022typical}. Of particular relevance is the work of Shmerkin and Solomyak \cite{Shmerkin2016Solomyak}, who established absolute continuity, with quantitative regularity, for almost every parameter in broad  families of self-similar measures. A  striking application of this circle of ideas appears in spectral theory: Damanik, Gorodetski and Solomyak \cite{DamanikGorodetskiSolomyak2015} proved absolute continuity of the density of states measure for the square Fibonacci Hamiltonian for almost every pair of sufficiently small coupling constants, and explicitly asked whether the exceptional set can be removed. Relatedly, Gorodetski and Northrup proved positive Lebesgue measure for almost every member of certain families of  Cantor sums "near" the affine setting, and conjectured that this almost every restriction is only technical \cite{Gorodetski2018Northrup}. Theorem~\ref{thm:convolution-intro} addresses this same general issue from a different direction: it gives absolute continuity for a prescribed pair rather than for almost every parameter.

\subsection{Projections of self-conformal measures}
\label{subsec:planar-conformal-intro}
We next turn to projections of self-conformal measures on the plane. We identify
$\mathbb R^2$ with $\mathbb C$, and let $D=B(0,1)$ be the unit disc. Consider a finite IFS
$\Phi=\{f_i:i\in\mathcal A\}$
of injective complex-analytic contractions defined on a neighbourhood of $D$,
satisfying the uniform contraction assumptions analogous to
\eqref{eq:1d-uniform-contraction}. For a strictly positive probability vector
$p$ on $\mathcal A$, let $\mu$ be the corresponding  self-conformal
measure, as in \eqref{eq:self-conformal}, and let $K_\Phi$ be its attractor, as in
\eqref{eq:attractor}. We again call $K_\Phi$ a self-conformal set. If $\Phi$
consists only of similarities (affine conformal maps), we call $\Phi$ self-similar, $K_\Phi$ a
self-similar set, and $\mu$ a self-similar measure.

A fundamental problem is to understand the geometry of the line projections
$(\pi_e)_\#\mu$ in prescribed directions. At the level of  dimension,
this problem is by now quite well understood for broad classes of dynamically
defined measures. Building on the local-entropy methods of Hochman and Shmerkin
\cite{hochman2009local}, several results establishing \eqref{eq:MM-dimension}
in prescribed directions have been obtained in this setting, including work of
Falconer and Jin \cite{Falconer2014Jin,Falconer2015Jin} and Bruce and Jin
\cite{Bruce2019Jin}. In particular, under suitable non-linearity assumptions,
Bruce and Jin proved that every orthogonal projection satisfies
\eqref{eq:MM-dimension}. More recently, 
B\'ar\'any--K\"aenm\"aki--Py\"or\"al\"a--Wu \cite{barany2023scaling} and
Py\"or\"al\"a \cite{Pyorala2025Dissonance} have given
further  dimension preservation results for projections of self-conformal measures. Py\"or\"al\"a's result, in particular, shows that under the assumptions of Theorem \ref{thm:planar-conformal-intro} below, $\dim (\pi_e)_\#\mu=1$ for all $e\in \Sone$.  See also the
recent work of Algom and Shmerkin \cite{algom2024dimension} on prescribed
projections of self-similar measures, and Wu \cite{wu2025projection}, who
showed, in particular, that for a broad class of dynamically defined measures
there are at most countably many exceptional directions for
\eqref{eq:MM-dimension}.

There is some suggestive evidence for absolute continuity of $(\pi_e)_\#\mu$ when $\dim \mu >1$. A principle appearing in Shmerkin's work \cite{Shmkerin2014Abs} about measures on the line, is that convolution of a measure with power
Fourier decay with a measure of dimension one is absolutely continuous. In our
setting, assuming mild non-degeneracy and non-linearity,  the dimension theory above supplies $\dim (\pi_e)_\#\mu=1$,
while our previous work \cite{algom2024plane} gives power Fourier decay for the  self-conformal measure itself (and so uniformly for all of its projections). Taken together,
these facts  suggest an absolute-continuity counterpart could be possible. The following theorem
confirms this expectation, although its proof does not proceed through  a
convolution decomposition and does not use Shmerkin's criterion at all.

\begin{thm}
\label{thm:planar-conformal-intro}
Let $\mu$ be a self-conformal measure with respect to a $C^\omega (\mathbb{C})$ IFS $\Phi$. Suppose that:
\begin{enumerate}
\item[(i)] $\Phi$ is not $C^\omega$ conjugate to a self-similar IFS; and,
\item[(ii)] $K_\Phi$ is not contained in a real-analytic planar curve; and,
\item[(iii)] $\dim \mu>1$.
\end{enumerate}
Then
$(\pi_e)_\#\mu\ll\Leb^1$ for every $e\in\Sone.$
\end{thm}
Although assumption~(ii) follows from~(iii), we state it explicitly
to emphasize that our hypotheses match those of the renewal theorem
from our previous work \cite[Theorem~3.1]{algom2024plane}.
This result is recalled below as Theorem~\ref{thm:planar-angular-renewal}
and provides a key input to the proof of the present theorem.

The nonlinearity assumption is again essential to this Theorem. Rapaport
\cite{Rapaort2017exp} constructed a planar self-similar measure satisfying the
strong separation condition, with dense rotations and dimension greater than
one, whose projections are singular for a dense $G_\delta$ set of directions.
Thus, even though every projection has the expected dimension, absolute
continuity  can fail dramatically in the self-similar
setting.

The hypotheses of Theorem~\ref{thm:planar-conformal-intro} are not hard to verify in many situations. Under the strong separation
condition, the dimension of $\mu$ is given by the classical 
formula
$\dim\mu=\frac{H(p)}{\chi(\Phi,p)},$
where
$H(p):=-\sum_{i\in\mathcal A}p_i\log p_i$ is Shannon entropy, and
$\chi(\Phi,p):=
-\sum_{i\in\mathcal A}p_i\int\log |f_i'(z)|\,d\mu(z)$  is the Lyapunov exponent. 
Thus assumption~(iii) is equivalent in this case to
$H(p)>\chi(\Phi,p)$. This formula belongs to the classical thermodynamic
formalism for conformal repellers and cookie-cutter systems; see, for example,
\cite{Bedford1991cookie,Patzschke1997self}, and also the much more general
work of De-Jun Feng and Hu \cite{feng2009dimension}. We can estimate
$\chi(\Phi,p)$ in our setting by noting that if
$m_i\leq |f_i'(z)|\leq M_i$ on $D$, then
$-\sum_i p_i\log M_i
\leq \chi(\Phi,p)\leq
-\sum_i p_i\log m_i.$
In particular, the uniform bounds in \eqref{eq:1d-uniform-contraction} give
$-\log\rho_+\leq\chi(\Phi,p)\leq-\log\rho_-$.

More recently, Zhou Feng and Rapaport \cite{FengRapaport2026}, extending
Rapaport's one-dimensional result \cite{Rapaport2024analytic}, obtained the
same dimension assuming
only exponential separation, together with natural 
nondegeneracy assumptions. Recall that $\Phi$ is \emph{exponentially separated} if
\begin{equation} \label{eq:exp sep}
\exists c>0\, \text{ s.t. for infinitely many } n,\, 
\sup_{z\in D} |f_u-f_v|\geq c^n
\,
\text{for all distinct }u,v\in\mathcal A^n,
\end{equation}
where $f_u=f_{u_1}\circ\cdots\circ f_{u_n}$. This condition, originally introduced by Hochman \cite{hochman2014self} is substantially weaker than
the strong separation condition. Feng and Rapaport assume, in addition,
that the maps have no common fixed point, preserve no regular real-analytic
curve, and that $\Phi$ is not holomorphically conjugate to a homothetic IFS.
These  assumptions follow from (i)--(ii) above: a common fixed
point would make $K_\Phi$ a singleton, an invariant real-analytic curve
would contain $K_\Phi$, and a holomorphic conjugacy to a homothetic IFS
would in particular give a real-analytic conjugacy to a self-similar IFS.
Consequently, in their setting,
assumption~(iii) again reduces to the concrete inequality
$H(p)>\chi(\Phi,p)$. The geometric assumptions (i)--(ii) can likewise be
checked directly in many examples; see also
\cite{FengRapaport2026} for useful criteria excluding real-analytic curve
obstructions. We give explicit examples satisfying all the hypotheses in
Subsection~\ref{subsec:explicit-planar}.

There is also a substantial parameter-dependent theory of
absolute continuity of projections in the self-similar setting. Results of
Shmerkin--Solomyak \cite{Shmerkin2016Solomyak}, Shmerkin
\cite{shmerkin2016furstenberg}, and K\"aenm\"aki--Orponen
\cite{Antti2018orponen}, among others, give absolute continuity outside small
exceptional sets of directions or parameters. Rapaport
\cite{Rapaort2020proj} obtained all-direction absolute continuity outside an
exceptional set of system parameters. See also
\cite{barany2022typical,solomyak2023absolute} for broader parameter-dependent
absolute-continuity results for dynamically defined and planar self-similar
measures.

Let us now  compare Theorem~\ref{thm:planar-conformal-intro} with our recent
work \cite{ARHWSmoothProjections2026}. There, we studied the smoothness of
projections of self-similar measures in dimensions three and higher. The present
Theorem~\ref{thm:planar-conformal-intro} concerns nonlinear self-conformal
measures in the plane. The underlying mechanisms are also substantially
different. The main dynamical input in \cite{ARHWSmoothProjections2026} is an
$L^2$ spectral gap for the action of the rotational part of the IFS on the
relevant family of projection directions; in particular, the resulting
regularity may depend on the prescribed direction. Here, by contrast, the crucial input is
a $C^{1+\gamma}$ spectral gap for twisted transfer operators associated to the
derivative cocycle, yielding the effective equidistribution estimates used
below. Accordingly, Theorem~\ref{thm:planar-conformal-intro} is an
all-direction statement: its hypotheses imply absolute continuity in every
 direction. Thus, although both results ultimately yield absolute
continuity of projections, they differ substantially both in scope and in
their underlying mechanism. See more on this in Section \ref{subsec:main-technical-intro}.

Finally, we mention the very recent work of Jin and Sahlsten
\cite{JinSahlsten2026}. Although closer in scope to our previous work \cite{ARHWSmoothProjections2026}, it shares
a key proof idea with the present paper,  related to the
$L^1$  covariance  Lemma~\ref{lem:L1-scaling};
see Remark~\ref{rem:jin-sahlsten} for further discussion.

\subsection{Projections of self-affine measures}
\label{subsec:self-affine-intro}
Our third application concerns self-affine measures on the plane. Here the IFS
$\Phi$ consists of non-singular affine contractions
$f_i(x)=A_i x+b_i,
A_i\in\mathrm{GL}(2,\R), \|A_i\|<1,$
where $\|\cdot\|$ denotes the operator norm, and $b_i\in \mathbb{R}^2$. Given a strictly positive
probability vector $\mathbf{p}$, we let $\mu$ be the corresponding stationary
measure, as in \eqref{eq:self-conformal}. In this setting, $\Phi$ is called a
\emph{self-affine} IFS and $\mu$ a \emph{self-affine measure}. If all the
linear parts $A_i$ are similarities, then $\Phi$ is also self-similar; however,
the assumptions below will rule out this case entirely.

Write
$
\Gamma_+^T=\langle A_i^T:i\in\mathcal A\rangle_+
$
for the semigroup generated by the transposed matrices. We say that $\Phi$ is \emph{strongly irreducible} if no finite union of one-dimensional subspaces is invariant under $\Gamma_+^T$, and \emph{proximal} if $\Gamma_+^T$ contains an element with a simple dominant eigenvalue. We abbreviate these two assumptions by \emph{SIP}. Notice that proximality already rules out the self-similar case: if every $A_i$ were conformal, then every element of the generated semigroup would be conformal and hence would have eigenvalues of equal modulus.

Equip $\Pone$ with the angular metric
$d_{\Pone}([u],[v])
=
\min\{\angle(u,v),\pi-\angle(u,v)\}.$
Under SIP, the transpose-projective random walk has a unique stationary
measure $\nu_F$, called the \emph{Furstenberg measure}; see, for example,
\cite{BougerolLacroix1985}. Thus, analogously to \eqref{eq:self-conformal},
\begin{equation}\label{eq:furstenberg-stationary-intro}
\nu_F
=
\sum_{i\in\mathcal A}
p_i\,((A_i^T)^{\Pone})_\#\nu_F,
\end{equation}
where $(A_i^T)^{\Pone}([v])=[A_i^Tv]$. For a probability measure
$\sigma$ on $\Pone$, define its  Frostman (or $L^\infty$) dimension by
\begin{equation}\label{eq:frostman-dim-intro}
\dimFr\sigma
=
\sup\Big\{\kappa\geq0:
\exists C_\kappa<\infty\text{ such that }
\sigma(B(\theta,r))\leq C_\kappa r^\kappa
\text{ for all }\theta\in\Pone,\ 0<r\leq1
\Big\}.
\end{equation}
We use the notation $I_s(\mu)$ from \eqref{eq:def-correlation-dim}, now with
the Euclidean distance on $\R^2$.

There is again suggestive evidence for absolute continuity of the projections
in the setting of self-affine measures $\mu \in \mathcal{P}(\mathbb{R}^2)$ under the SIP condition. Li and Sahlsten \cite{Li2020Sahl} proved that, in dimension
two, irreducibility together with non-compactness of the projective action
implies power Fourier decay for self-affine measures; these hypotheses follow
from SIP in our setting. Consequently, $\mu$ has power Fourier decay, and so
does every projection $(\pi_e)_\#\mu$. The work of
B\'ar\'any--Hochman--Rapaport \cite{BaranyHochmanRapaport2019} about strongly separated IFS, and its
extension to exponentially separated IFS \eqref{eq:exp sep} by Hochman--Rapaport
\cite{HochmanRapaport2022}, provides us with dimension theory for
$\mu$. Very recently,
B\'ar\'any--K\"aenm\"aki--Kolossv\'ary
\cite{BaranyKaenmakiKolossvary2026} proved that, under exponential separation
and SIP,
$\dim (\pi_e)_\#\mu=\min\{1,\dim\mu\}$
for every $e\in\Sone.$
Thus, when $\dim \mu >1$, every projection simultaneously has full
dimension and power Fourier decay. In light of Shmerkin's convolution
principle \cite{Shmkerin2014Abs}, this gives  evidence for absolute
continuity. As in the self-conformal case, however, our proof uses a different mechanism to prove this. Our third main result is the following.

\begin{thm}
\label{thm:self-affine-intro}
Let $\mu\in \mathcal{P}(\mathbb{R}^2)$ be a self-affine measure with respect to an IFS $\Phi$ that is strongly irreducible and proximal. If, for some
$0<s<2$,
\begin{equation}\label{eq:self-affine-threshold-intro}
I_s(\mu)<\infty
\qquad\text{and}\qquad
s+\dimFr\nu_F>2,
\end{equation}
then
$(\pi_e)_\#\mu\ll\Leb^1$ for every $e\in\Sone.$
\end{thm}
By Rapaport's example
\cite{Rapaort2017exp}, discussed after Theorem \ref{thm:planar-conformal-intro}, 
some genuine non-conformality hypothesis is essential to this Theorem.

There is an important difference between Theorem~\ref{thm:self-affine-intro}
and the preceding two applications. In the convolution and self-conformal
settings, the assumption $\dim\mu>1$ can be converted into  a
uniform $L^q$ concentration estimate required by the proof; this can be done by taking $q>1$ sufficiently close to
one, and using the fact that
\begin{equation}\label{eq:self-affine-lq-limit}
\lim_{q\downarrow1}D_q(\mu)=\dim\mu.
\end{equation}
This was originally noted by Shmerkin and Solomyak \cite{Shmerkin2016Solomyak}; 
see Theorem \ref{thm:conformal-uniform-qmass} for a full proof of a more refined statement. One of the ingredients underlying the argument is exact dimensionality of
$\mu$, a notion recalled below. Exact dimensionality of self-affine
measures was an open problem for a long time, but it is now known in the planar setting by
B\'ar\'any and K\"aenm\"aki \cite{BaranyKaenmaki2017}, and in  greater
generality by Feng \cite{Feng2023dim}. Exact dimensionality alone, however,
does not imply \eqref{eq:self-affine-lq-limit}, and we do not know whether
\eqref{eq:self-affine-lq-limit} holds in the present setting. We therefore
impose the stronger finite-energy hypothesis in
\eqref{eq:self-affine-threshold-intro}.
It is natural to ask whether \eqref{eq:self-affine-threshold-intro} can be replaced by the purely
dimensional condition
\begin{equation}\label{eq:self-affine-desired}
\dim\mu+\dimFr\nu_F>2.
\end{equation}
Indeed, if \eqref{eq:self-affine-lq-limit} were known, then our argument would
yield the conclusion of Theorem~\ref{thm:self-affine-intro} under
\eqref{eq:self-affine-desired}. We do not know whether this strengthening
holds in the required generality.

Theorem~\ref{thm:self-affine-intro} should also be viewed in the context of
the rapidly developing dimension theory of self-affine measures and their
projections. Besides
\cite{BaranyHochmanRapaport2019,BaranyKaenmakiKolossvary2026}, see also
\cite{FengXie2025,MorrisSert2025} for recent results on the dimensions and
geometry of such projections. A particularly relevant absolute-continuity
result is due to B\'ar\'any \cite{Barany2025Hausdorff}, who considers
dominated planar self-affine systems, that is, systems with a uniform
exponential separation between the two singular-value directions. Under an additional separation assumption, he shows that positivity of the Hausdorff measure of the attractor in its affinity dimension is equivalent to the projections of the K\"aenm\"aki measure having uniformly bounded densities in every direction in the support of the corresponding Furstenberg measure.

Finally, the hypotheses of Theorem~\ref{thm:self-affine-intro} can be
verified concretely. SIP can often be checked directly from the matrices,
while the Frostman exponent of $\nu_F$ may be estimated from the associated
projective IFS; see also the recent work of Rush
\cite{Rush2026Frostman} for a general formula for the uniform Frostman
dimension in the $\mathrm{SL}(2,\R)$ SIP setting. Finite energy can likewise
be obtained from direct Frostman estimates, for example under strong
separation together with quantitative lower bounds on the least singular
values of cylinder maps. In Subsection~\ref{subsec:explicit-self-affine} we
give a completely explicit example in which all these quantities are
estimated directly and the planar measure is singular, while every one of
its line projections is absolutely continuous.
\subsection{The main technical theorem}\label{subsec:main-technical-intro}

We now formulate our main technical Theorem, that implies the three applications as special cases. We begin
by fixing the Littlewood--Paley decomposition that will be used throughout the
paper. Choose an even function
$\varphi\in C_c^\infty(\R),
\,
0\leq\varphi\leq1,$
such that
$\varphi(\xi)=1$ for $|\xi|\leq1,$ and
$\varphi(\xi)=0$ for $|\xi|\geq2,$
and set
$\psi(\xi):=\varphi(\xi)-\varphi(2\xi).$
Then $\psi$ is even and
$\supp\psi\subset\{\xi:1/2\leq|\xi|\leq2\}.$
Moreover,
$\sum_{j\in\mathbb Z}\psi(2^{-j}\xi)=1$ for
$\xi \neq 0$. We will make use of the following inhomogeneous form of this identity, which we call a smooth dyadic partition of unity in frequency
space. Namely,
\begin{equation}\label{eq:LP-partition-inhom-intro}
\varphi(\xi)+\sum_{j=1}^{\infty}\psi(2^{-j}\xi)=1,\, \xi\in\R.
\end{equation}
We refer to \cite{Grafakos2014Classical,Grafakos2009Modern} for
standard background on Littlewood--Paley decompositions and the associated
function spaces.

For $\rho \in \mathcal{P}(\mathbb{R})$  and $R>0$, define its Littlewood--Paley
piece $\Delta_R\rho$ by
\begin{equation}\label{eq:LP-def-intro}
\widehat{\Delta_R\rho}(\xi)
=
\psi(\xi/R)\widehat\rho(\xi),
\qquad
\widehat\rho(\xi)
=
\int e^{-2\pi i\xi x}\,d\rho(x).
\end{equation}
We also write $P_{\leq1}\rho$ for the low-frequency piece defined by
$\widehat{P_{\leq1}\rho}(\xi)=\varphi(\xi)\widehat\rho(\xi).$
Thus \eqref{eq:LP-partition-inhom-intro} gives, in the sense of tempered
distributions,
\begin{equation}\label{eq:LP-decomposition-intro}
\rho
=
P_{\leq1}\rho+\sum_{j=1}^{\infty}\Delta_{2^j}\rho.
\end{equation}
For projections, this decomposition has a particularly simple geometric
interpretation. If $\eta\in\mathcal P(\R^2)$ and $e\in\Sone$, then
$\widehat{(\pi_e)_\#\eta}(\xi)
=
\widehat\eta(\xi e).$
Hence the Littlewood-Paley decomposition of $(\pi_e)_\#\eta$ is obtained by
restricting $\widehat\eta$ to 
$\{\xi e:\xi\in\R\}$ and decomposing this line into dyadic annuli. In
particular, $\Delta_{2^j}((\pi_e)_\#\eta)$ corresponds to frequencies
$2^{j-1}\leq|\xi|\leq2^{j+1}$
on this line.

An elementary observation underlying our approach is the well known fact that if $\sum_{j=1}^{\infty} \|\Delta_{2^j}\rho\|_{L^1(\R)} <\infty$  then $\rho$ is absolutely continuous. More quantitatively, if for some $\gamma>0$ we have
$\|\Delta_R\rho\|_{L^1(\R)}
\lesssim R^{-\gamma}$ for
$R\geq2$, then the density belongs to the inhomogeneous Besov space
$B^t_{1,1}(\R)$ for every $0<t<\gamma$; see, e.g., \cite{Mattila2015new}.

This motivates the main observable used in our argument. For a compactly
supported probability measure $\eta$ on $\R^2$ and a projective direction
$\theta=[e]\in\Pone$, put
\begin{equation}\label{eq:Q-def-intro}
Q_R(\eta,\theta)
:=
\|\Delta_R((\pi_e)_\#\eta)\|_{L^1(\R)}.
\end{equation}
Since $\psi$ is even, this is unchanged when $e$ is replaced by $-e$, and
hence $Q_R$ is well defined on $\Pone$.
The key feature of this observable
is its exact covariance under one-dimensional affine renormalization,
explained below. In particular, affine rescaling of the measure changes only the frequency
parameter $R$, without introducing any multiplicative loss in the $L^1$
norm, a feature specific to $L^1$.

Finally, for $m\in\mathbb N$ let
\begin{equation}\label{eq:dyadic-squares-intro}
\mathcal D_m
:=
\left\{
2^{-m}\bigl(k+[0,1)^2\bigr):k\in\mathbb Z^2
\right\}
\end{equation}
be the collection of half-open dyadic squares of side length $2^{-m}$ in
$\R^2$. For $1<q\leq2$, define the dyadic $q$-moment sum of a probability
measure $\eta$ by
\begin{equation}\label{eq:qmass-intro}
S_{m,q}(\eta)
:=
\sum_{Q\in\mathcal D_m}\eta(Q)^q.
\end{equation}

The following theorem is our main technical result.
\begin{thm}\label{thm:stopped-intro}
Let $\nu \in \mathcal{P}(\mathbb{R}^2)$ be  compactly supported, and let $E\subset\Pone$.  Assume there is a non-degenerate compact interval $J\subset\R$,  constants $1<q\leq2$, $0<\kappa\leq 1$, $S>0$ and
\[
 B,C,\epsilon,c,\delta>0,\qquad N_0,M,r\in\N,
\]
with
\begin{equation}\label{eq:smoothing-renewal-balance-intro}
 (2r-1)\delta<\epsilon,
\end{equation}
such that the following holds.

For every $e\in E$ and every sufficiently large $R$ (uniformly in $e$), one can choose
$k\asymp\log R,
 \, U\asymp e^{\delta k}$ with implicit constants uniform in $e,R$,
a probability space $(\Xi,m)$, an integer $1\leq N\leq N_0$, and measurable probability kernels
\[
 \xi\mapsto\eta_{\xi,j}\in\mathcal P(B(0,M)),
 \qquad
 \xi\mapsto P_{\xi,j},\Lambda_{\xi,j}\in\mathcal P(\Pone\times J),
 \qquad 1\leq j\leq N,
\]
such that:

\smallskip
\noindent\textup{(A) Uniform  $q$-mass.}  For every $n\geq1$, for $m$-almost every $\xi$ and every $j$:
\begin{equation}\label{eq:A-intro}
 S_{n,q}(\eta_{\xi,j})\leq B\,2^{-n(q-1)S}.
\end{equation}

\smallskip
\noindent\textup{(B) Renormalization.}
\begin{equation}\label{eq:B-intro}
 Q_R(\nu,e)
 \leq C\int_\Xi\sum_{j=1}^N
 \int_{\Pone\times J}
 Q_{Ue^z}(\eta_{\xi,j},\theta)\,dP_{\xi,j}(\theta,z)\,dm(\xi)
 +Ce^{-ck}.
\end{equation}

\smallskip
\noindent\textup{(C) Effective  angular equidistribution.}  There is a full $m$-measure set $\Xi_0\subset\Xi$ such that for every $\xi\in\Xi_0$, every $j$, and every nonnegative $G\in C^r(\Pone\times J)$,
\begin{equation}\label{eq:C-intro}
 \int G\,dP_{\xi,j}
 \leq C\int G\,d\Lambda_{\xi,j}
 +Ce^{-\epsilon k}\|G\|_{C^r}.
\end{equation}
Furthermore, uniformly in $\xi,j$,

\begin{equation}\label{eq:C-frostman-intro}
 (\operatorname{proj}_{\Pone})_\#\Lambda_{\xi,j}(B(\theta,\rho))\leq C\rho^\kappa
 \qquad\theta\in\Pone,\ 0<\rho\leq1
\end{equation}
\smallskip

If
$S+\kappa>2,$
then there are  $C_*,\gamma>0$ such that
$Q_R(\nu,e)\leq C_*R^{-\gamma}$ for every
$e\in E,\ R\geq2$. In particular,
 $(\pi_e)_\#\nu\ll\Leb^1$ for every $e\in E$.
\end{thm}

Let us briefly explain the meaning of the three assumptions. The measures
$\eta_{\xi,j}$ should be thought of as normalized pieces, or blow-ups, arising
from a decomposition of the original measure $\nu$. Condition~\textup{(A)}
says that these pieces retain, uniformly, the $q$-mass estimates available
for the original measure. Condition~\textup{(B)} is  renormalization in the following sense:
after passing to the appropriate blow-up, the high-frequency quantity
$Q_R(\nu,\theta_0)$ can be reduced to the adjusted-frequency quantities
$Q_{Ue^z}(\eta_{\xi,j},\theta)$. Crucially, this renormalization costs only
the exponentially small error $\lesssim e^{-ck}$. Finally,
condition~\textup{(C)} asserts that the projective laws governing the
renormalized pieces equidistribute, uniformly and at an exponential rate,
towards limiting angular laws whose marginals satisfy a uniform Frostman
bound of exponent $\kappa$. In the important case $\kappa=1$, these limiting
angular laws have uniformly bounded densities. These hypotheses may at
first appear rather demanding, but all three of our main applications fit
naturally into this framework.

Let us indicate how condition~\textup{(A)} is obtained in our three
applications. For the convolution and planar self-conformal theorems, the
relation \eqref{eq:self-affine-lq-limit} will allow us to choose $q>1$
sufficiently close to one and convert the super-critical Hausdorff-dimension
hypothesis into the uniform $q$-mass estimate in \eqref{eq:A-intro}. In the
convolution setting the resulting exponent $S$ is the sum of the
contributions from the two factors. For the self-affine application we instead use the
finite-energy assumption in \eqref{eq:self-affine-threshold-intro}.
Standard dyadic energy estimates, recorded in
Lemma~\ref{lem:energy-qmass}, then give \eqref{eq:A-intro} with $q=2$ and
$S=s$. This is the precise reason for the finite-energy hypothesis in
Theorem~\ref{thm:self-affine-intro}.

Condition~\textup{(B)} is obtained by combining a decomposition of the underlying measure with respect to a carefully chosen random variable, together with linearization, and renormalization. In the convolution and planar self-conformal applications, this random variable is related to a random walk driven by the derivative cocycle. It is not, strictly speaking, a stopping time, but it plays the role of one, and can be uniformly compared with a genuine stopping time. Variants of this construction have played a central role in our previous work on Fourier decay
\cite{algom2020decay,algom2021decay,algom2023polynomial,algom2024plane}.
The precise choice of this random variable is important. On the one hand, its distribution gives the effective equidistribution results underlying condition~\textup{(C)}. On the other hand, its uniform comparison with a genuine stopping time allows us to decompose the measure at the corresponding geometric scale and to linearize the relevant cylinder maps uniformly. This interplay between the random-walk description and the genuine stopping-time structure is a subtle but essential feature of the argument. After linearization, the resulting normalized pieces are precisely the measures $\eta_{\xi,j}$ appearing in Theorem~\ref{thm:stopped-intro}, and the exact covariance of $Q_R$ under affine rescaling yields the renormalization estimate~\eqref{eq:B-intro}, up to an exponentially small error.
In the self-affine application the situation is considerably simpler. We use the matrix stopping time underlying the renewal theorem of Li--Sahlsten
\cite{Li2020Sahl}. Since the cylinder maps are already affine, no linearization is required; the affine renormalization is exact, and condition~\textup{(B)} follows directly from the covariance of $Q_R$.

As for Condition~\textup{(C)}, under the nonlinearity assumptions in the first two applications, and under the SIP assumption in the self-affine setting, the laws $P_{\xi,j}$ converge at an exponential rate towards limiting laws $\Lambda_{\xi,j}$ which can be identified explicitly. In the nonlinear applications, this effective convergence is one of the central mechanisms behind our Fourier-decay results
\cite{algom2020decay,algom2023polynomial,algom2024plane}; in the form needed here, it follows from a $C^{1+\gamma}$ spectral gap for twisted transfer operators associated to the derivative cocycle. In both of these applications, the angular marginals
$
(\operatorname{proj}_{\Pone})_\#\Lambda_{\xi,j}
$
have uniformly bounded densities, so that Condition~\textup{(C)} holds with $\kappa=1$. In the self-affine setting, the corresponding effective convergence follows from the quantitative renewal theory for products of random matrices developed by Li and Li--Sahlsten
\cite{Li2018decay,li2018fourier,Li2020Sahl}; the angular marginals of $\Lambda_{\xi,j}$ are governed by the Furstenberg measure, yielding Condition~\textup{(C)} for every $\kappa<\dim_{\mathrm{Fr}}\nu_F$.

This should be contrasted with our recent work on smooth
projections of self-similar measures
\cite{ARHWSmoothProjections2026}, where the  input is instead an
$L^2$ spectral gap for the rotational action on the space of projection
directions.

The crucial improvement over our previous Fourier-decay arguments is the
exact covariance of the observable $Q_R$ under affine renormalization. Since
the outer norm in \eqref{eq:Q-def-intro} is  $L^1$, rescaling changes
only the frequency parameter and introduces no additional multiplicative
loss. This leaves only the  linearization and equidistribution errors
to control, and ultimately gives enough decay in $R$ to make the
Littlewood--Paley decomposition \eqref{eq:LP-decomposition-intro} summable.
This feature is specific to $L^1$: replacing the outer norm by $L^p$ would introduce the factor
$|a|^{1/p-1}$ under a rescaling $x\mapsto ax+b$. This additional scaling
loss would have to be absorbed at every renormalization step and makes the
argument substantially more delicate.

We conclude with a brief indication of the proof of
Theorem~\ref{thm:stopped-intro}. Once conditions~\textup{(A)}--\textup{(C)}
have been established, the essential analytic step is an angular averaging
gain. This is a finite-frequency $L^q$ analogue of the classical
Fourier-analytic projection estimates underlying the exceptional-set theory
of Falconer \cite{Falconer1982Projections}; see also
\cite[Chapter~9]{mattila1999geometry}. Combining the uniform $q$-mass
estimate in~\eqref{eq:A-intro} with the $\kappa$-Frostman bound
in~\eqref{eq:C-frostman-intro} gives
\begin{equation}\label{eq:angular-gain-intro}
\int_{\Pone\times J}
Q_{Ue^z}(\eta_{\xi,j},\theta)\,
d\Lambda_{\xi,j}(\theta,z)
\lesssim U^{-\beta},
\qquad
\beta=\frac{q-1}{q}(S+\kappa-2).
\end{equation}
Thus the threshold $S+\kappa>2$ produces a genuine power gain after angular
averaging. Some smoothing and interpolation are needed in order to pass from
the effective equidistribution statement~\eqref{eq:C-intro} to
\eqref{eq:angular-gain-intro}; these are carried out in
Section~\ref{sec:proof-main-technical}. The exponential rate in
condition~\textup{(C)}, together with
\eqref{eq:smoothing-renewal-balance-intro}, absorbs the resulting smoothing
loss. Substituting \eqref{eq:angular-gain-intro} into the renormalization
estimate~\eqref{eq:B-intro} then yields power decay of $Q_R$, and the
Littlewood--Paley summability discussed above completes the proof.

\begin{remark}
\label{rem:jin-sahlsten}
Upon completion of this work, we learned of independent work by
Jin and Sahlsten \cite{JinSahlsten2026}. Although their results
are closer in scope to our previous work on smooth projections
of self-similar measures in higher dimensions
\cite{ARHWSmoothProjections2026}, there is an interesting
methodological connection with the present paper.
They prove that, in dimensions $d\geq 3$, equicontractive
self-similar measures with finite $t$-energy for some $t>k$
have absolutely continuous projections onto every
$k$-dimensional subspace, provided the defining rotations
generate a dense subgroup of $\mathrm{SO}(d)$.
They also obtain quantitative estimates for the projected
densities and establish dimension conservation under
additional assumptions.
Importantly, their result requires no spectral gap for
the underlying rotations.

The methodological connection concerns the use of
$L^1$ quantities whose norms are invariant under
appropriate affine rescaling.
Jin and Sahlsten work with certain smoothened increments
of martingale differences. A key observation in their
argument is that, due to self-similarity, the $L^1$
norms of these increments are preserved under affine
rescaling, after the corresponding adjustment of scale;
see \cite[Section~4.2, equation~(4.7)]{JinSahlsten2026}.
This is essentially the same scaling mechanism underlying
our Lemma~\ref{lem:L1-scaling}, although it is implemented
using different quantities and in a different setting.
The two papers exploit this mechanism in different ways
and to different ends (as we've already explained).  Jin and Sahlsten combine their argument with
Varj\'u's work on random walks on compact groups
\cite{Varju2013random}, whereas we combine
this  $L^1$ covariance with effective
renewal theorems and equidistribution results.

\end{remark}

\subsection{Organization}

In Section~\ref{sec:preliminaries} we establish some preliminary
results concerning our Littlewood--Paley operators, and dyadic
$q$-moment estimates. Section~\ref{sec:proof-main-technical}
is devoted to the proof of Theorem~\ref{thm:stopped-intro},
whose main ingredient is a Falconer-type angular averaging estimate, Proposition \ref{prop:angular-gain}.
In Section~\ref{sec:conformal-qmass} we establish uniform
$q$-mass estimates for cylinders of self-conformal measures. 
The proofs of our three main applications are given in
Sections~\ref{sec:convolution-proof},
\ref{sec:planar-conformal-proof}, and
\ref{sec:self-affine-proof}, respectively.
We conclude in Section~\ref{sec:explicit-examples}
with explicit examples illustrating these results.

The proofs of the three applications of Theorem \ref{thm:stopped-intro} follow a common strategy.
In each case, we first construct a suitable decomposition of
the underlying measure, together with the normalized pieces
and probability kernels appearing in
Theorem~\ref{thm:stopped-intro}.
We then verify its three hypotheses, typically beginning
with the renormalization estimate~\textup{(B)}, followed
by the uniform $q$-mass bound~\textup{(A)} and the effective
angular equidistribution estimate~\textup{(C)}.
Although the overall strategy is the same, the constructions
and the required renewal estimates differ between the
applications. For example,  Theorem \ref{thm:convolution-intro}
requires a more involved linearization argument,
whereas in the self-affine setting, Theorem \ref{thm:self-affine-intro}, the cylinder maps are
already affine and linearization (and renormalization) is simpler.

\section{Preliminaries}\label{sec:preliminaries}
We collect here some basic analytical facts used throughout the paper.  The main point is the exact covariance of the $L^1$ norm of the Littlewood--Paley pieces defined in \eqref{eq:LP-def-intro} under one-dimensional affine rescaling.   We also record the regularity in direction and logarithmic scale needed to apply the effective renewal estimates.

\subsection{Littlewood--Paley kernels and exact $L^1$ covariance}\label{subsec:kernels-scaling}
Recall from Subsection~\ref{subsec:main-technical-intro} the fixed even cutoff
$\psi\in C_c^\infty(\R\setminus{0})$ and the associated Littlewood--Paley operators
$\Delta_R$ defined in \eqref{eq:LP-def-intro}.  Writing
\begin{equation}\label{eq:kernel-def}
K_R(x)=R\check\psi(Rx),
\qquad
\Delta_R\rho=K_R*\rho ,
\end{equation}
gives the corresponding physical-space kernel representation. We  require the following basic estimates.
\begin{lemma}\label{lem:kernel-estimates}
For every $R>0$,
\begin{equation}\label{eq:kernel-basic}
\|K_R\|_{L^1}
=
\|\check\psi\|_{L^1},
\qquad
\|K_R'\|_{L^1}
=
R\|(\check\psi)'\|_{L^1}.
\end{equation}
Moreover, for every $N\geq1$ there is a constant $C_N<\infty$ such that
\begin{equation}\label{eq:kernel-decay}
|K_R(x)|
\leq
C_N R(1+R|x|)^{-N},
\qquad x\in\R,\ R>0.
\end{equation}
Finally, for every fixed $m\geq1$,
\begin{equation}\label{eq:log-kernel-derivative}
\sup_{V>0}
\left\|
(V\partial_V)^m K_V
\right\|_{L^1}
<\infty.\end{equation}
\end{lemma}

\begin{proof}
The identities in \eqref{eq:kernel-basic} follow immediately from
$K_R(x)=R\check\psi(Rx),$ since
$K_R'(x)=R^2(\check\psi)'(Rx),$
after the change of variables $u=Rx$. Since $\check\psi$ is Schwartz, for every
$N\geq1$ there is $C_N<\infty$ such that
$|\check\psi(u)|\leq C_N(1+|u|)^{-N}.$
Substituting $u=Rx$ gives \eqref{eq:kernel-decay}.

It remains to prove \eqref{eq:log-kernel-derivative}. Since
$
K_V(x)=V\check\psi(Vx),
$
we have
$
V\partial_V K_V(x)
=
V\check\psi(Vx)+V^2x(\check\psi)'(Vx).$
More generally, writing $(\check\psi)^{(j)}$ for the $j$th derivative of
$\check\psi$, repeated application of $V\partial_V$ shows that, for every
fixed $m\geq1$, $(V\partial_V)^mK_V(x)$ is a finite linear combination,
with coefficients depending only on $m$, of terms of the form
$
V^{j+1}x^j(\check\psi)^{(j)}(Vx)$ for $0\leq j\leq m.$
For each such term, the change of variables $u=Vx$ gives
$$
\left\|
V^{j+1}x^j(\check\psi)^{(j)}(Vx)
\right\|_{L^1}
=
\left\|
u^j(\check\psi)^{(j)}(u)
\right\|_{L^1}.
$$
The quantity on the right is finite because $\check\psi$ is Schwartz, and
it is independent of $V$. Since only finitely many such terms occur for
each fixed $m$, \eqref{eq:log-kernel-derivative} follows.
\end{proof}

The next identity is elementary; nonetheless, it plays a key role in our analysis.

\begin{lemma}\label{lem:L1-scaling}
Let $T_{a,b}(x)=ax+b$ with $a\neq0$.  Then for every  $\rho \in \mathcal{P}(\mathbb{R})$,
\begin{equation}\label{eq:L1-scaling}
 \|\Delta_R((T_{a,b})_\#\rho)\|_{L^1(\R)}
 =
 \|\Delta_{|a|R}\rho\|_{L^1(\R)}.
\end{equation}
\end{lemma}

\begin{proof}
For $a\neq0$,
$\widehat{(T_{a,b})_\#\rho}(\xi)
=
e^{-2\pi i b\xi}\widehat\rho(a\xi).$
Let
$F=\Delta_{|a|R}\rho.$
Since $\psi$ is even,
\[
\widehat F(a\xi)
=
\psi\left(\frac{a\xi}{|a|R}\right)\widehat\rho(a\xi)
=
\psi\left(\frac{\xi}{R}\right)\widehat\rho(a\xi).
\]
Hence
$\widehat{\Delta_R((T_{a,b})_\#\rho)}(\xi)
=
e^{-2\pi i b\xi}\widehat F(a\xi).$
By Fourier inversion,

$$
\begin{aligned}
\Delta_R((T_{a,b})_\#\rho)(x)
&=
\int_{\R}
e^{2\pi i\xi x}
e^{-2\pi i b\xi}
\widehat F(a\xi)\,d\xi \\
&=
\int_{\R}
e^{2\pi i\xi(x-b)}
\widehat F(a\xi)\,d\xi \\
&=
|a|^{-1}
\int_{\R}
e^{2\pi i u(x-b)/a}
\widehat F(u)\,du \\
&=
|a|^{-1}
F\left(\frac{x-b}{a}\right),
\end{aligned}
$$
where in the third line we changed variables $u=a\xi$.
Therefore, by the change of variables $u=(x-b)/a$,
\[
\|\Delta_R((T_{a,b})_\#\rho)\|_{L^1}
=
\int_{\R}|a|^{-1}
\left|F\left(\frac{x-b}{a}\right)\right|\,dx
=
\|F\|_{L^1}.
\]
Since $F=\Delta_{|a|R}\rho$, this proves the claim.
\end{proof}
\begin{remark}\label{rem:L1-structural}
The identity \eqref{eq:L1-scaling} is specific to $L^1$. Indeed, the same computation in its proof shows that, for $1\leq p<\infty$,
$
\|\Delta_R((T_{a,b})_\#\rho)\|_{L^p}
=
|a|^{1/p-1}
\|\Delta_{|a|R}\rho\|_{L^p}.
$
Thus $L^1$ is the only exponent for which affine rescaling introduces no multiplicative factor.
\end{remark}

\begin{lemma}\label{lem:uniform-displacement}
Let $\eta \in \mathcal{P}(\mathbb{R}^2)$ , and let $F,L:\R^2\to\R$ satisfy
$ \sup_{x\in \R^2}|F(x)-L(x)|\leq d.$
Then
\begin{equation}\label{eq:uniform-displacement}
 \left|
 \|\Delta_R(F_\#\eta)\|_{L^1}
 -
 \|\Delta_R(L_\#\eta)\|_{L^1}
 \right|
 \leq C_\psi Rd.
\end{equation}
\end{lemma}

\begin{proof}
For $u,v\in\R$, the fundamental theorem of calculus gives, for every $x\in\R$,
$K_R(x-u)-K_R(x-v)
=
-\int_v^u K_R'(x-t)\,dt.$
Therefore,
$$|K_R(x-u)-K_R(x-v)|
\leq
\int_{\min\{u,v\}}^{\max\{u,v\}}
|K_R'(x-t)|\,dt.$$
Integrating in $x$ and using Fubini,
\[
\begin{aligned}
\|K_R(\cdot-u)-K_R(\cdot-v)\|_{L^1}
&\leq
\int_{\min\{u,v\}}^{\max\{u,v\}}
\int_{\R}|K_R'(x-t)|\,dx\,dt \\
&=
|u-v|\,\|K_R'\|_{L^1}.
\end{aligned}
\]
By \eqref{eq:kernel-basic},
$\|K_R(\cdot-u)-K_R(\cdot-v)\|_{L^1}
\leq
C_\psi R|u-v|.$

Since
$ \Delta_R(F_\#\eta)=\int K_R(\cdot-F(x))\,d\eta(x),$
and similarly for $L$, Minkowski's inequality yields
\[
 \|\Delta_R(F_\#\eta)-\Delta_R(L_\#\eta)\|_{L^1}
 \leq C_\psi R\int|F-L|\,d\eta
 \leq C_\psi Rd.
\]
The reverse triangle inequality for norms gives the claim.
\end{proof}

\subsection{Smoothing of Littlewood-Paley kernels}\label{subsec:finite-frequency}
Let $J\subset\R$ be compact and non-degenerate, let $U\geq2$, and let
$\eta\in\mathcal P(\R^2)$ be compactly supported. Recall that for a projective
direction $\theta=[e]\in\Pone$, we defined in \eqref{eq:Q-def-intro}
$
Q_R(\eta,\theta)
:=
\|\Delta_R((\pi_e)_\#\eta)\|_{L^1(\R)},
$
and observed that this is well defined on $\Pone$. We now incorporate a bounded
logarithmic correction to the frequency by setting
\begin{equation}\label{eq:H-def}
H_{U,\eta}(\theta,z)
:=
Q_{Ue^z}(\eta,\theta),
\qquad
(\theta,z)\in\Pone\times J.
\end{equation}

\begin{lemma}\label{lem:H-regularity}
Fix $M<\infty$ and a compact non-degenerate interval $J$.  There is $C=C(M,J,\psi)$ such that, for every $\eta\in \mathcal P(B(0,M))$ and every $U\geq2$, on $\Pone\times J$ we have
\begin{equation}\label{eq:H-Lip}
 \|H_{U,\eta}\|_\infty\leq C,
 \qquad
 \operatorname{Lip}(H_{U,\eta})\leq CU.
\end{equation}
\end{lemma}

\begin{proof}
Young's inequality gives the $L^\infty$ bound since
$$H_{U,\eta}(\theta,z)
 \leq \|K_{Ue^z}\|_{L^1}\,\|(\pi_\theta)_\#\eta\|_{\mathrm{TV}}
 =\|\check\psi\|_{L^1}.$$
For the direction variable, choose  representatives $e_1,e_2\in \Sone$ of $\theta_1,\theta_2$ with
$|e_1-e_2|\lesssim d_{\Pone}(\theta_1,\theta_2)$.  If $x\in B(0,M)$, then
$|\pi_{e_1}(x)-\pi_{e_2}(x)|
 \lesssim M d_{\Pone}(\theta_1,\theta_2).$
Applying Lemma~\ref{lem:uniform-displacement} at frequency $Ue^z\asymp_J U$ gives
\[
 |H_{U,\eta}(\theta_1,z)-H_{U,\eta}(\theta_2,z)|
 \lesssim U d_{\Pone}(\theta_1,\theta_2).
\]
For the scale variable, let $z_1,z_2\in J$ and put
$
V_i=Ue^{z_i},\, i=1,2.
$
By \eqref{eq:log-kernel-derivative},
$$
\|K_{V_1}-K_{V_2}\|_{L^1}
\lesssim_J
|z_1-z_2|.
$$
Indeed, writing $V(t)=Ue^t$,
$
K_{V_1}-K_{V_2}
=
\int_{z_2}^{z_1}
(V(t)\partial_{V(t)})K_{V(t)}\,dt,
$
and the claim follows from the case $m=1$ of
\eqref{eq:log-kernel-derivative}. Therefore,
$$
|H_{U,\eta}(\theta,z_1)-H_{U,\eta}(\theta,z_2)|
\leq
\|K_{V_1}-K_{V_2}\|_{L^1}
\lesssim_J
|z_1-z_2|.
$$
Combining this with the directional estimate, and using $U\geq2$, proves
\eqref{eq:H-Lip}.
\end{proof}

We next require the following smoothing estimate:

\begin{lemma}\label{lem:H-smoothing}
Under the hypotheses of Lemma~\ref{lem:H-regularity}, let $r\geq1$ be an integer.
There exists $C_r=C_r(M,J,\psi)<\infty$ such that, for every $U\geq2$ and every
$\eta\in\mathcal P(B(0,M))$, there is a nonnegative function
$
H_{U,\eta}^{\sharp}\in C^r(\Pone\times J)
$
satisfying
\begin{equation}\label{eq:H-smoothing}
\|H_{U,\eta}-H_{U,\eta}^{\sharp}\|_{\infty}
\leq C_rU^{-1},
\qquad
\|H_{U,\eta}^{\sharp}\|_{C^r}
\leq C_r U^{2r-1}.
\end{equation}
Moreover, $H_{U,\eta}^{\sharp}$ may be chosen measurably as a function of $\eta$.
\end{lemma}

\begin{proof}
Write $J=[a,b]$ and identify $\Pone$ with $\R/\pi\Z$. Define the clamping map
$$
c_J(z):=\min\{b,\max\{a,z\}\},
\, z\in\R,
$$
and extend $H_{U,\eta}$ to $\Pone\times\R$ by
$
\widetilde H_{U,\eta}(\theta,z)
:=
H_{U,\eta}(\theta,c_J(z)).
$
Since $c_J$ is $1$-Lipschitz, Lemma~\ref{lem:H-regularity} gives
$
\|\widetilde H_{U,\eta}\|_\infty\lesssim 1,$
$\operatorname{Lip}(\widetilde H_{U,\eta})\lesssim U.
$
Moreover, $\widetilde H_{U,\eta}\geq0$.

Fix a nonnegative function $\zeta\in C_c^\infty(\R^2)$ with
$
\int_{\R^2}\zeta=1,
$
and for $\tau>0$ set
$$
\zeta_\tau(x):=\tau^{-2}\zeta(x/\tau).
$$
Viewing $\Pone$ as $\R/\pi\Z$, define $\widetilde\zeta_\tau(\theta,z)$ on $\Pone\times\R$ by
$\widetilde\zeta_\tau(\theta,z)
:=
\sum_{k\in\Z}\zeta_\tau(\theta+k\pi,z),$
which is $\pi$-periodic in the first variable. We then set
\[
H_{U,\eta}^{\sharp}
:=
\widetilde H_{U,\eta}*\widetilde\zeta_\tau
\,\text{ on }\Pone\times\R \text{ with } \tau:=U^{-2}.
\]
Since both $\widetilde H_{U,\eta}$ and $\zeta_\tau$ are nonnegative,
$H_{U,\eta}^{\sharp}\geq0$.

As
$\int_{\R^2}\zeta_\tau(u,v)\,du\,dv=1,$
we may write
$H_{U,\eta}(\theta,z)
=
\widetilde H_{U,\eta}(\theta,z)
=
\int_{\R^2}
\widetilde H_{U,\eta}(\theta,z)\,
\zeta_\tau(u,v)\,du\,dv,$
for $(\theta,z)\in\Pone\times J$. So, for such $(\theta,z)$,
$$
\begin{aligned}
\left|
H_{U,\eta}^{\sharp}(\theta,z)
-
H_{U,\eta}(\theta,z)
\right|
&\leq
\int_{\R^2}
\left|
\widetilde H_{U,\eta}((\theta,z)-y)
-
\widetilde H_{U,\eta}(\theta,z)
\right|
\zeta_\tau(y)\,dy \\
&\leq
\operatorname{Lip}(\widetilde H_{U,\eta})
\int_{\R^2}|y|\zeta_\tau(y)\,dy.
\end{aligned}
$$
By scaling,
$
\int_{\R^2}|y|\zeta_\tau(y)\,dy
=
\tau\int_{\R^2}|y|\zeta(y)\,dy,
$
and hence
$
\|H_{U,\eta}^{\sharp}-H_{U,\eta}\|_\infty
\lesssim
U\tau
=
U^{-1}.
$

It remains to estimate the derivatives. Let $\alpha$ be a multi-index with
$1\leq|\alpha|=j\leq r$. Since
$
\int_{\R^2}D^\alpha\zeta_\tau(y)\,dy=0,
$
we may write
$$
\begin{aligned}
D^\alpha H_{U,\eta}^{\sharp}(x)
&=
\int_{\R^2}
\widetilde H_{U,\eta}(x-y)
D^\alpha\zeta_\tau(y)\,dy \\
&=
\int_{\R^2}
\big(
\widetilde H_{U,\eta}(x-y)
-
\widetilde H_{U,\eta}(x)
\big)
D^\alpha\zeta_\tau(y)\,dy.
\end{aligned}
$$
Therefore,
$
|D^\alpha H_{U,\eta}^{\sharp}(x)|
\leq
\operatorname{Lip}(\widetilde H_{U,\eta})
\int_{\R^2}
|y|\,|D^\alpha\zeta_\tau(y)|\,dy.
$
Since
$
D^\alpha\zeta_\tau(y)
=
\tau^{-2-j}(D^\alpha\zeta)(y/\tau),
$
a change of variables gives
$
\int_{\R^2}
|y|\,|D^\alpha\zeta_\tau(y)|\,dy
\lesssim_\alpha
\tau^{1-j}.
$
Hence
$$
\|D^\alpha H_{U,\eta}^{\sharp}\|_\infty
\lesssim_r
U\tau^{1-j}
=
U^{2j-1}
\leq
U^{2r-1}.
$$
Together with the uniform bound on $H_{U,\eta}^{\sharp}$, this yields
$$
\|H_{U,\eta}^{\sharp}\|_{C^r}
\lesssim_r
U^{2r-1}.
$$

Finally, the extension and mollification above are fixed operations, independent
of $\eta$. Since $H_{U,\eta}(\theta,z)$ is measurable in $\eta$ for every
$(\theta,z)$, the same is true of $H_{U,\eta}^{\sharp}(\theta,z)$.
Thus $H_{U,\eta}^{\sharp}$ may be chosen measurably in $\eta$.
\end{proof}

\subsection{Moment sums and annular regularization}\label{subsec:qmass-prelim}
Recall from \eqref{eq:dyadic-squares-intro} that for $m\geq1$,  $\mathcal D_m$ denotes the collection of half-open dyadic cubes of side length $2^{-m}$ in the ambient Euclidean space $\mathbb{R}$ or $\mathbb{R}^2$; which space will always be clear from context. If $\eta$ is a probability measure and $1<q\leq2$, recall from \eqref{eq:qmass-intro} that
$S_{m,q}(\eta)
:=
\sum_{Q\in\mathcal D_m}\eta(Q)^q.$
The lower  $L^q$ dimension of $\eta$ is defined by
\begin{equation}\label{eq:Dq-prelim}
D_q(\eta)
:=
\liminf_{m\to\infty}
\frac{-\log S_{m,q}(\eta)}
{(q-1)m\log 2}.
\end{equation}
Consequently, if
$0<S<D_q(\eta),$
then there exists $C<\infty$ such that
\begin{equation}\label{eq:qmass-from-Dq}
S_{m,q}(\eta)
\leq
C\,2^{-m(q-1)S}
\, \text{ for } m\geq 1.
\end{equation}

We shall compare this discrete information with an $L^q$ norm at physical scale $U^{-1}$. Recall that $\psi$ and $\Delta_R$ were fixed in Subsection~\ref{subsec:main-technical-intro}. For the compact interval $J$ fixed above, choose a radial function
$
\vartheta\in C_c^\infty(\R^2\setminus\{0\})
$
such that
$$
\vartheta(\xi)=1
\qquad\text{whenever}\qquad
\frac12 e^{\min J}\leq|\xi|\leq 2e^{\max J}.
$$
Define the Schwartz kernel $L_U$ by
\begin{equation}\label{eq:LU-def}
\widehat{L_U}(\xi)=\vartheta(\xi/U).
\end{equation}
Then
$
L_U(x)=U^2L_1(Ux),
$
where $L_1=\check\vartheta$ is a Schwartz function.

\begin{lemma}\label{lem:qmass-annular}
Fix $1<q\leq2$. There exists $C=C(q,\vartheta)<\infty$ such that, for every
$\eta\in\mathcal P(\R^2)$ and every $m\geq1$, if
$
2^m\leq U<2^{m+1},
$
then
\begin{equation}\label{eq:qmass-annular}
\|L_U*\eta\|_{L^q(\R^2)}^q
\leq
C\,U^{2(q-1)}S_{m,q}(\eta).
\end{equation}
\end{lemma}
\begin{proof}
Write the dyadic squares in $\mathcal D_m$ as
$
Q_k:=2^{-m}\bigl(k+[0,1)^2\bigr),$  $k\in\Z^2.$
Since $2^m\leq U<2^{m+1}$, we have
$
U\asymp 2^m.$
In particular, every $Q_k$ has diameter $\asymp U^{-1}$ and Lebesgue measure
$
|Q_k|=2^{-2m}\asymp U^{-2}.
$
Recall that
$
L_U(x)=U^2L_1(Ux),
$
where $L_1$ is Schwartz. Hence, for every $N\geq1$,
$
|L_U(x)|
\leq
C_NU^2(1+U|x|)^{-N}.
$

If $x\in Q_k$ and $y\in Q_\ell$, then $
U|x-y|
\gtrsim
|k-\ell|-1,
$
with an absolute implied constant, since $U\asymp2^m$. It follows that, after increasing the constant,
$
|L_U(x-y)|
\leq
C_NU^2(1+|k-\ell|)^{-N}.
$
Choose $N>2$ and set
$
a(j):=(1+|j|)^{-N},$ $j\in\Z^2.$
Then $a\in\ell^1(\Z^2)$. 

Therefore, for $x\in Q_k$,
$$
\begin{aligned}
|L_U*\eta(x)|
&\leq
\sum_{\ell\in\Z^2}
\int_{Q_\ell}|L_U(x-y)|\,d\eta(y) \\
&\lesssim
U^2
\sum_{\ell\in\Z^2}
a(k-\ell)\eta(Q_\ell).
\end{aligned}
$$
Define
$
b(\ell):=\eta(Q_\ell),
\, \ell\in\Z^2.
$
Then the preceding estimate becomes
$$
|L_U*\eta(x)|
\lesssim
U^2(a*b)(k),
\qquad x\in Q_k,
$$
where $a*b$ denotes discrete convolution on $\Z^2$.

Raising to the $q$-th power and integrating over $Q_k$ gives
$$
\begin{aligned}
\int_{Q_k}|L_U*\eta(x)|^q\,dx
&\lesssim
|Q_k|\,U^{2q}|(a*b)(k)|^q \\
&\lesssim
U^{2q-2}|(a*b)(k)|^q.
\end{aligned}
$$
Summing over $k\in\Z^2$, we obtain
$
\|L_U*\eta\|_{L^q(\R^2)}^q
\lesssim
U^{2q-2}\|a*b\|_{\ell^q(\Z^2)}^q.
$
By Young's convolution inequality on $\ell^q(\Z^2)$,
$
\|a*b\|_{\ell^q}
\leq
\|a\|_{\ell^1}\|b\|_{\ell^q}.
$
Since $a\in\ell^1(\Z^2)$ is fixed,
$$
\|L_U*\eta\|_{L^q(\R^2)}^q
\lesssim
U^{2q-2}
\sum_{\ell\in\Z^2}b(\ell)^q.
$$
Finally,
$
\sum_{\ell\in\Z^2}b(\ell)^q
=
\sum_{Q\in\mathcal D_m}\eta(Q)^q
=
S_{m,q}(\eta),
$
and hence
$$
\|L_U*\eta\|_{L^q(\R^2)}^q
\lesssim
U^{2(q-1)}S_{m,q}(\eta),
$$
as required.
\end{proof}

For later use in the self-affine application, we also record the elementary $q=2$ consequence of finite energy.

\begin{lemma}\label{lem:energy-qmass}
Let $\eta\in\mathcal P(\R^2)$ and suppose that
$I_s(\eta)<\infty$ for some $0<s<2$. Then
\begin{equation}\label{eq:energy-qmass}
S_{m,2}(\eta)
\leq
C_s\,I_s(\eta)\,2^{-ms}
\, \quad m\geq1.
\end{equation}
\end{lemma}

\begin{proof}
Fix $m\geq1$. If $x,y$ belong to the same dyadic square
$Q\in\mathcal D_m$, then
$
|x-y|
\leq
\operatorname{diam}(Q)
=
\sqrt2\,2^{-m}.
$
Hence
$
|x-y|^{-s}
\geq
(\sqrt2\,2^{-m})^{-s}
=
2^{-s/2}2^{ms}.
$

Using only those pairs $(x,y)$ which lie in a common dyadic square, we obtain
$$
\begin{aligned}
I_s(\eta)
&=
\iint_{\R^2\times\R^2}
|x-y|^{-s}\,d\eta(x)\,d\eta(y) \\
&\geq
\sum_{Q\in\mathcal D_m}
\iint_{Q\times Q}
|x-y|^{-s}\,d\eta(x)\,d\eta(y) \\
&\geq
2^{-s/2}2^{ms}
\sum_{Q\in\mathcal D_m}
\eta(Q)^2.
\end{aligned}
$$
Since
$
\sum_{Q\in\mathcal D_m}\eta(Q)^2
=
S_{m,2}(\eta),
$
it follows that
$
S_{m,2}(\eta)
\leq
2^{s/2}I_s(\eta)\,2^{-ms},
$
which proves \eqref{eq:energy-qmass}.
\end{proof}

\section{Proof of the main technical theorem}\label{sec:proof-main-technical}

\subsection{A Falconer-type angular average estimate}\label{subsec:angular-gain}

The following proposition is the key  estimate underlying the proof of Theorem~\ref{thm:stopped-intro}. It is a Falconer-type projection estimate, in the spirit of the Fourier-analytic exceptional-set argument of \cite{Falconer1982Projections}; we use the term \emph{angular} rather than spherical average, since the directions are averaged against an arbitrary Frostman measure rather than surface measure. Its assumptions are designed to match conditions~\textup{(A)} and~\textup{(C)} of Theorem~\ref{thm:stopped-intro}.

\begin{proposition}\label{prop:angular-gain}
Fix $M<\infty$, a compact non-degenerate interval $J$, $1<q\leq2$, $S>0$, and $0<\kappa\leq1$. Let $\eta\in\mathcal P(B(0,M))$ and suppose that, for some $B<\infty$,
\begin{equation}\label{eq:angular-qmass-assumption}
S_{m,q}(\eta)\leq B2^{-m(q-1)S}
\quad m\geq1.
\end{equation}
Let $\Lambda\in\mathcal P(\Pone\times J)$, and let $\sigma$ be its marginal on $\Pone$. Assume that
\begin{equation}\label{eq:angular-frostman-assumption}
\sigma(B(\theta,\rho))\lesssim\rho^\kappa,
\qquad \theta\in\Pone,\ 0<\rho\leq1.
\end{equation}
Then, for every $N\geq1$ there exists $C_N<\infty$ such that, for every $U\geq2$,
\begin{equation}\label{eq:angular-gain}
\int_{\Pone\times J}H_{U,\eta}(\theta,z)\,d\Lambda(\theta,z)
\lesssim B^{1/q}U^{-\beta}+C_NU^{-N},
\end{equation}
where
\begin{equation}\label{eq:beta-def}
\beta=\frac{q-1}{q}(S+\kappa-2).
\end{equation}
\end{proposition}
Here and below, the implicit constant in \eqref{eq:angular-gain} depends only on $M,J,q,S,\kappa$ and on the implicit constant in \eqref{eq:angular-frostman-assumption}; the constant $C_N$ may additionally depend on $N$.
\begin{proof}

\medskip
\noindent\textbf{Step 1: A useful annular representation.}
Fix $U\geq2$ and put
$F_U:=L_U*\eta,$
where $L_U$ is defined in \eqref{eq:LU-def}. For an integrable function
$F$ on $\R^2$, let $\pi_\theta F$ denote the density of the push-forward
of $F(x)\,dx$ under orthogonal projection in the direction $\theta$.
Thus, if $e$ is either unit representative of $\theta$, then the
Fourier-slice identity gives
\begin{equation}\label{eq:fourier-slice-proof}
\widehat{\pi_\theta F}(t)=\widehat F(te).
\end{equation}
All norms below are independent of the choice of representative.

Recall that $\vartheta$ was chosen so that
$\vartheta(\xi)=1$ whenever 
$\frac12 e^{\min J}\leq|\xi|\leq2e^{\max J},$
and that
$\widehat{L_U}(\xi)=\vartheta(\xi/U).$
Also recall from Section \ref{subsec:main-technical-intro} that
$\supp\psi\subset\{t\in\R:1/2\leq|t|\leq2\}.$
Consequently, if $z\in J$ and
$\psi\left(\frac{t}{Ue^z}\right)\neq0,$
then
$\frac12
\leq
\frac{|t|}{Ue^z}
\leq
2,$
and hence
$\frac12 e^{\min J}
\leq
\frac{|t|}{U}
\leq
2e^{\max J}.$
Since $|e|=1$, it follows from the choice of $\vartheta$ that
\[
\widehat{L_U}(te)
=
\vartheta(te/U)
=
1, \text{ whenever }  \psi(t/(Ue^z))\neq0.
\]

We Claim that 
\begin{equation}\label{eq:annular-exact-representation}
\Delta_{Ue^z}\bigl((\pi_\theta)_{\#}\eta\bigr)
=
K_{Ue^z}*\pi_\theta F_U.
\end{equation}
By the
definition of $\Delta_R$ and the Fourier-slice identity,
$\widehat{
\Delta_{Ue^z}\bigl((\pi_\theta)_{\#}\eta\bigr)
}(t)
=
\psi\left(\frac{t}{Ue^z}\right)\widehat{\eta}(te).$
On the other hand, since $F_U=L_U*\eta$,
$\widehat{F_U}(te)
=
\widehat{L_U}(te)\widehat{\eta}(te),$
and therefore
\begin{align*}
\widehat{K_{Ue^z}*\pi_\theta F_U}(t)
&=
\psi\left(\frac{t}{Ue^z}\right)
\widehat{\pi_\theta F_U}(t)\\
&=
\psi\left(\frac{t}{Ue^z}\right)
\widehat{F_U}(te)\\
&=
\psi\left(\frac{t}{Ue^z}\right)
\widehat{L_U}(te)\widehat{\eta}(te)\\
&=
\psi\left(\frac{t}{Ue^z}\right)
\widehat{\eta}(te).
\end{align*}
Indeed, we used \eqref{eq:fourier-slice-proof}, and in the last equality that
$\widehat{L_U}(te)=1$ whenever
$\psi(t/(Ue^z))\neq0$, while if
$\psi(t/(Ue^z))=0$ both sides vanish. Thus the two Fourier transforms
agree for every $t\in\R$, and  \eqref{eq:annular-exact-representation} follows.

\medskip
\noindent\textbf{Step 2: $L^1$ and  $L^2$ estimates.}
We first record a simple $L^1$ estimate. By Fubini,
$\|\pi_\theta F\|_{L^1}\leq \|F\|_{L^1},
\, \theta\in\Pone.$
Moreover, $\|K_R\|_{L^1}=\|\check\psi\|_{L^1}$ for every $R>0$.
Hence Young's inequality gives
\begin{equation}\label{eq:L1-angular-endpoint}
\int_{\Pone\times J}
\left\|K_{Ue^z}*\pi_\theta F\right\|_{L^1}\,
d\Lambda(\theta,z)
\lesssim
\|F\|_{L^1}.
\end{equation}

Suppose now that $F\in L^2(\R^2)$ is supported in $B(0,r), r>0$. We claim
that, recalling the definition of $\kappa$ from \eqref{eq:angular-frostman-assumption},
\begin{equation}\label{eq:L2-angular-endpoint}
\int_{\Pone\times J}
\left\|K_{Ue^z}*\pi_\theta F\right\|_{L^2}^2\,
d\Lambda(\theta,z)
\lesssim_r
U^{-\kappa}\|F\|_{L^2}^2.
\end{equation}
To prove this, lift $\Lambda\in\mathcal P(\Pone\times J)$ to an antipodally symmetric  $\widetilde{\Lambda} \in \mathcal{P} \left( \Sone\times J \right)$, obtained by splitting
the mass at each projective direction equally between its two unit
representatives. By Plancherel and \eqref{eq:fourier-slice-proof}
\[
\int_{\Pone\times J}
\left\|K_{Ue^z}*\pi_\theta F\right\|_{L^2}^2\,
d\Lambda(\theta,z)
=
\int_{\Sone\times J}\int_{\R}
|\widehat F(te)|^2
\left|\psi\left(\frac{t}{Ue^z}\right)\right|^2
\,dt\,d\widetilde{\Lambda}(e,z).
\]

Let $\Gamma_U$ be the finite Borel measure on $\R^2$ obtained as the
push-forward of
$\left|\psi\left(\frac{t}{Ue^z}\right)\right|^2
\,dt\,d\widetilde{\Lambda}(e,z)$
under the map $(t,e,z)\mapsto te$. Since
$\supp\psi\subset\{s\in\R:1/2\leq|s|\leq2\},$
the measure $\Gamma_U$ is supported in the annulus
\[
\frac12 e^{\min J}U
\leq |\xi|
\leq
2e^{\max J}U.
\]
Moreover, by the change of variables $s=t/(Ue^z)$,
\[ \Gamma_U(\R^2)
=
U\|\psi\|_{L^2}^2
\int_{\Sone\times J}e^z\,d\widetilde{\Lambda}(e,z) \asymp_J U.
\]
We next estimate the mass of Euclidean balls under $\Gamma_U$. We claim that
\begin{equation}\label{eq:Gamma-ball}
\Gamma_U(B(\xi,r))
\lesssim
U^{-\kappa}r^{1+\kappa},
\,
0<r\lesssim U.
\end{equation}
For $r\asymp U$, this follows immediately from
$\Gamma_U(\R^2)\lesssim U$. Thus it remains to consider $0<r\ll U$.
Suppose that $B(\xi,r)$ meets the support of $\Gamma_U$. Since
$\Gamma_U$ is supported in
$\frac12 e^{\min J}U
\leq |\zeta|
\leq
2e^{\max J}U,$
we then have $|\xi|\asymp_J U$. Let $\theta_\xi\in\Pone$ denote the
projective direction of $\xi$. If
$te\in B(\xi,r)$
for some $(t,e,z)$ contributing to $\Gamma_U$, then
$\bigl||t|-|\xi|\bigr|\leq r,$
so, for fixed $(e,z)$, the admissible values of $t$ lie in a union of
at most two intervals of total length $O(r)$. Moreover, since
$|t|\asymp_J U$ and $|te-\xi|\leq r$, the projective direction $[e]$
must satisfy
$d_{\Pone}([e],\theta_\xi)
\lesssim_J
\frac{r}{U}.$
Therefore, writing $\widetilde{\sigma}$ for the $\Sone$ marginal of
$\widetilde{\Lambda}$ and using that its projection to $\Pone$ is
$\sigma$, we obtain
\begin{align*}
\Gamma_U(B(\xi,r))
&\lesssim
r\,
\widetilde{\sigma}
\left(
\left\{
e\in\Sone:
d_{\Pone}([e],\theta_\xi)
\lesssim_J \frac{r}{U}
\right\}
\right)\\
&\lesssim_J
r\,
\sigma
\left(
B\left(\theta_\xi, \frac{r}{U}\right)
\right)\\
&\lesssim
r\left(\frac{r}{U}\right)^\kappa
=
U^{-\kappa}r^{1+\kappa},
\end{align*}
where in the last line we used
\eqref{eq:angular-frostman-assumption}. This proves
\eqref{eq:Gamma-ball}.

Choose $\chi\in C_c^\infty(\R^2)$ equal to $1$ on the support of $F$.
Since $F=\chi F$,
$\widehat F=\widehat\chi*\widehat F.$
By Cauchy--Schwarz and the rapid decay of $\widehat\chi$, for every
 $A>1+\kappa$,
\[
|\widehat F(\xi)|^2
\lesssim
\int_{\R^2}
|\widehat F(\zeta)|^2
(1+|\xi-\zeta|)^{-A}\,d\zeta.
\]

We will  use \eqref{eq:Gamma-ball} to prove that
\begin{equation}\label{eq:Gamma-potential}
\sup_{\zeta\in\R^2}
\int_{\R^2}
(1+|\xi-\zeta|)^{-A}\,d\Gamma_U(\xi)
\lesssim
U^{-\kappa}.
\end{equation}
Fix $\zeta\in\R^2$ and write
$I_U(\zeta)
:=
\int_{\R^2}
(1+|\xi-\zeta|)^{-A}\,d\Gamma_U(\xi).$
By
\eqref{eq:Gamma-ball}, 
$\Gamma_U(B(\zeta,1))
\lesssim
U^{-\kappa}.$

For $j\geq0$, let
$\mathcal A_j(\zeta)
=
\{\xi\in\R^2:2^j\leq|\xi-\zeta|<2^{j+1}\}.$
As long as $2^{j+1}\lesssim U$, \eqref{eq:Gamma-ball} gives
\[
\Gamma_U(\mathcal A_j(\zeta))
\leq
\Gamma_U(B(\zeta,2^{j+1}))
\lesssim
U^{-\kappa}2^{(j+1)(1+\kappa)}.
\]
Hence
\begin{align*}
\int_{\mathcal A_j(\zeta)}
(1+|\xi-\zeta|)^{-A}\,d\Gamma_U(\xi)
&\lesssim
2^{-jA}\Gamma_U(\mathcal A_j(\zeta))\\
&\lesssim
U^{-\kappa}2^{-j(A-1-\kappa)}.
\end{align*}
Therefore, since $A>1+\kappa$, summing over all such $j$ gives
\[
\sum_{2^{j+1}\lesssim U}
\int_{\mathcal A_j(\zeta)}
(1+|\xi-\zeta|)^{-A}\,d\Gamma_U(\xi)
\lesssim
U^{-\kappa}.
\]
It remains to consider the shells for which $2^j\gtrsim U$. Using
the total mass bound $\Gamma_U(\R^2)\lesssim U$, we obtain
\begin{align*}
\sum_{2^j\gtrsim U}
\int_{\mathcal A_j(\zeta)}
(1+|\xi-\zeta|)^{-A}\,d\Gamma_U(\xi)
&\lesssim
\sum_{2^j\gtrsim U}
2^{-jA}\Gamma_U(\R^2)\\
&\lesssim
U\sum_{2^j\gtrsim U}2^{-jA}\\
&\lesssim
U^{1-A}.
\end{align*}
Since $A>1+\kappa$, 
$U^{1-A}\lesssim U^{-\kappa}.$
Combining the three estimates, \eqref{eq:Gamma-potential} follows.

We can now finish the $L^2$ estimate. Recall that
\[
|\widehat F(\xi)|^2
\lesssim
\int_{\R^2}
|\widehat F(\zeta)|^2
(1+|\xi-\zeta|)^{-A}\,d\zeta,
\]
and hence, by Tonelli's theorem, by our estimates
\begin{align*}
\int_{\R^2}|\widehat F(\xi)|^2\,d\Gamma_U(\xi)
&\lesssim
\int_{\R^2}\int_{\R^2}
|\widehat F(\zeta)|^2
(1+|\xi-\zeta|)^{-A}\,
d\zeta\,d\Gamma_U(\xi)\\
&=
\int_{\R^2}
|\widehat F(\zeta)|^2
\left(
\int_{\R^2}
(1+|\xi-\zeta|)^{-A}\,d\Gamma_U(\xi)
\right)
d\zeta\\
&\lesssim
U^{-\kappa}
\int_{\R^2}|\widehat F(\zeta)|^2\,d\zeta\\
&=
U^{-\kappa}\|F\|_{L^2}^2,
\end{align*}
where the last equality follows from Plancherel. Together with the
definition of $\Gamma_U$, this proves
\eqref{eq:L2-angular-endpoint}.

\medskip
\
\noindent\textbf{Step 3: An $L^q$ estimate for $1<q\leq2$.}
Recall that the kernels $K_R$ were defined in \eqref{eq:kernel-def}.
Let $r>0$ and let $F$ be a function on $\R^2$ supported in
$B(0,r)$. Define
\[
T_UF(\theta,z,x)
:=
K_{Ue^z}*\pi_\theta F(x),
\, \text{ where }
(\theta,z,x)\in\Pone\times J\times\R.
\]
We regard $T_UF$ as a function on
$\Pone\times J\times\R$, equipped with the product measure
$d\Lambda(\theta,z)\,dx$.

The $L^1$ estimate \eqref{eq:L1-angular-endpoint} gives
$\|T_UF\|_{L^1(\Lambda\times\Leb^1)}
\lesssim
\|F\|_{L^1(\R^2)},$
while \eqref{eq:L2-angular-endpoint} gives
$\|T_UF\|_{L^2(\Lambda\times\Leb^1)}
\lesssim
U^{-\kappa/2}\|F\|_{L^2(\R^2)}.$
Therefore, by the Riesz--Thorin interpolation theorem, for every
$1<q\leq2$,
$\|T_UF\|_{L^q(\Lambda\times\Leb^1)}
\lesssim
U^{-\kappa(q-1)/q}\|F\|_{L^q(\R^2)}.$
Raising this estimate to the $q$th power gives
\begin{equation}\label{eq:Lq-angular-endpoint}
\int_{\Pone\times J}
\left\|K_{Ue^z}*\pi_\theta F\right\|_{L^q}^q\,
d\Lambda(\theta,z)
\lesssim
U^{-\kappa(q-1)}\|F\|_{L^q}^q.
\end{equation}

We now apply this estimate to the function
$F_U=L_U*\eta$
introduced in Step~1, and some $1<q\leq 2$. Since $F_U$ need not be compactly supported,
choose $\chi_0\in C_c^\infty(\R^2)$, with
$0\leq\chi_0\leq1$, such that $\chi_0=1$ on a fixed neighbourhood of
$B(0,M)$. Since $\eta$ is supported in $B(0,M)$ and
$L_U(x)=U^2L_1(Ux),$
with $L_1$ Schwartz, for every $N\geq1$ we have
\begin{equation}\label{eq:FU-tail}
\|(1-\chi_0)F_U\|_{L^1}
\leq
C_NU^{-N}.
\end{equation}
By \eqref{eq:annular-exact-representation} and the decomposition
$F_U=\chi_0F_U+(1-\chi_0)F_U,$
we have
\begin{align*}
\int_{\Pone\times J}H_{U,\eta}(\theta,z)\,d\Lambda(\theta,z)
&\leq
\int_{\Pone\times J}
\left\|K_{Ue^z}*\pi_\theta(\chi_0F_U)\right\|_{L^1}\,
d\Lambda(\theta,z)\\
&\quad+
\int_{\Pone\times J}
\left\|K_{Ue^z}*
\pi_\theta\bigl((1-\chi_0)F_U\bigr)\right\|_{L^1}\,
d\Lambda(\theta,z).
\end{align*}
By \eqref{eq:L1-angular-endpoint} and \eqref{eq:FU-tail}, the second
term is bounded by
$\|(1-\chi_0)F_U\|_{L^1}
\lesssim_N
U^{-N}.$

Thus, it remains to estimate
$\int_{\Pone\times J}
\left\|K_{Ue^z}*\pi_\theta(\chi_0F_U)\right\|_{L^1}\,
d\Lambda(\theta,z).$
Choose $m\geq1$ so that
$2^m\leq U<2^{m+1}.$
Since $0\leq\chi_0\leq1$, we have
$\|\chi_0F_U\|_{L^q}
\leq
\|F_U\|_{L^q}.$
Hence \eqref{eq:Lq-angular-endpoint} gives
\begin{align}
&\int_{\Pone\times J}
\left\|K_{Ue^z}*\pi_\theta(\chi_0F_U)\right\|_{L^q}^q\,
d\Lambda(\theta,z)
\notag\\
&\qquad\lesssim
U^{-\kappa(q-1)}
\|F_U\|_{L^q}^q.
\label{eq:Lq-first}
\end{align}
By Lemma~\ref{lem:qmass-annular},
$\|F_U\|_{L^q}^q
=
\|L_U*\eta\|_{L^q}^q
\lesssim
U^{2(q-1)}S_{m,q}(\eta).$
Using the hypothesis \eqref{eq:angular-qmass-assumption},
$S_{m,q}(\eta)
\leq
B2^{-m(q-1)S}.$
Since $2^m\leq U<2^{m+1}$,
$2^{-m(q-1)S}
\lesssim
U^{-(q-1)S},$
and therefore
\[
\|F_U\|_{L^q}^q
\lesssim
B\,U^{(q-1)(2-S)}.
\]
Substituting this into \eqref{eq:Lq-first}, for our $1<q\leq 2$ we have
\begin{equation}\label{eq:Lq-combined}
\int_{\Pone\times J}
\left\|K_{Ue^z}*\pi_\theta(\chi_0F_U)\right\|_{L^q}^q\,
d\Lambda(\theta,z)
\lesssim
B\,U^{-(q-1)(S+\kappa-2)}.
\end{equation}

\medskip
\noindent\textbf{Step 4: The $L^1$ angular average.}
Since $\chi_0F_U$ is supported in the fixed compact set
$\supp\chi_0$, all the projections
$\pi_\theta(\chi_0F_U)$ are supported in a common bounded interval $I$.
Fix a slightly larger interval $I^+$. By H\"older's inequality in the
line variable, followed by H\"older's inequality with respect to
$\Lambda$,
\begin{align*}
&\int_{\Pone\times J}
\left\|K_{Ue^z}*\pi_\theta(\chi_0F_U)\right\|_{L^1(I^+)}\,
d\Lambda(\theta,z)
\\
&\qquad\lesssim
\int_{\Pone\times J}
\left\|K_{Ue^z}*\pi_\theta(\chi_0F_U)\right\|_{L^q}\,
d\Lambda(\theta,z)
\\
&\qquad\lesssim
\left(
\int_{\Pone\times J}
\left\|K_{Ue^z}*\pi_\theta(\chi_0F_U)\right\|_{L^q}^q\,
d\Lambda(\theta,z)
\right)^{1/q}
\\
&\qquad\lesssim
B^{1/q}
U^{-\frac{q-1}{q}(S+\kappa-2)}
=
B^{1/q}U^{-\beta}.
\end{align*}
Since $I$ is contained in the interior of $I^+$,
$d:=\operatorname{dist}(I,\R\setminus I^+)>0.$
For $f_\theta:=\pi_\theta(\chi_0F_U)$ we have
$\supp f_\theta\subset I$ and $\|f_\theta\|_{L^1}\lesssim1$ uniformly
in $\theta$ and $U$. Hence, by \eqref{eq:kernel-decay},
\begin{align*}
\|K_{Ue^z}*f_\theta\|_{L^1(\R\setminus I^+)}
&\leq
\|f_\theta\|_{L^1}
\sup_{y\in I}
\int_{|x-y|\geq d}|K_{Ue^z}(x-y)|\,dx\\
&\lesssim_N
U^{-N},
\end{align*}
uniformly in $(\theta,z)\in\Pone\times J$, since $e^z\asymp_J1$.

Combining the preceding estimates with
\eqref{eq:annular-exact-representation}, we obtain
\[
\int_{\Pone\times J}
H_{U,\eta}(\theta,z)\,d\Lambda(\theta,z)
\lesssim
B^{1/q}U^{-\beta}+C_NU^{-N}.
\]
This proves \eqref{eq:angular-gain}.\end{proof}

\subsection{Proof of Theorem~\ref{thm:stopped-intro}}
\label{subsec:proof-stopped}
Once Proposition~\ref{prop:angular-gain} is established, Theorem~\ref{thm:stopped-intro} is fairly straightforward to derive.

\begin{proof}[Proof of Theorem~\ref{thm:stopped-intro}]
Fix $e\in E$ and a sufficiently large $R$, and fix the data supplied
by the theorem. Recalling the notations of Section \ref{subsec:finite-frequency}, for $1\leq j\leq N$, define
$H_{\xi,j}:=H_{U,\eta_{\xi,j}},$
and let $H_{\xi,j}^{\sharp}$ be the  positive smoothing from
Lemma~\ref{lem:H-smoothing}. Then, by this Lemma, uniformly in $\xi,j$,
\begin{equation}\label{eq:proof-smoothing-bounds}
\|H_{\xi,j}-H_{\xi,j}^{\sharp}\|_\infty\lesssim U^{-1},
\quad
\|H_{\xi,j}^{\sharp}\|_{C^r}\lesssim U^{2r-1}.
\end{equation}
The measurable dependence asserted in Lemma~\ref{lem:H-smoothing}
ensures that the following integrals are well defined.
Condition~\textup{(B)} and the first estimate in
\eqref{eq:proof-smoothing-bounds} give
\begin{equation}\label{eq:proof-B-smoothed}
Q_R(\nu,e)
\lesssim
\int_\Xi\sum_{j=1}^N
\int H_{\xi,j}^{\sharp}\,dP_{\xi,j}\,dm(\xi)
+U^{-1}+e^{-ck}.
\end{equation}
Since $H_{\xi,j}^{\sharp}\geq0$, condition~\textup{(C)} applies.
Using both estimates in \eqref{eq:proof-smoothing-bounds}, we obtain
\begin{align}
\int H_{\xi,j}^{\sharp}\,dP_{\xi,j}
&\lesssim
\int H_{\xi,j}^{\sharp}\,d\Lambda_{\xi,j}
+e^{-\epsilon k}U^{2r-1}
\notag\\
&\lesssim
\int H_{\xi,j}\,d\Lambda_{\xi,j}
+U^{-1}+e^{-\epsilon k}U^{2r-1}.
\label{eq:proof-C-smoothed}
\end{align}
Condition~\textup{(A)} supplies the same exponent $S$ for every
$\eta_{\xi,j}$, while the angular marginal of each
$\Lambda_{\xi,j}$ satisfies the same $\kappa$-Frostman estimate.
Proposition~\ref{prop:angular-gain} therefore gives, uniformly in
$\xi,j$,
\begin{equation}\label{eq:proof-angular-use}
\int H_{\xi,j}\,d\Lambda_{\xi,j}
\lesssim
U^{-\beta},
\qquad
\beta=\frac{q-1}{q}(S+\kappa-2)>0.
\end{equation}
Here the $O_N(U^{-N})$ term in
Proposition~\ref{prop:angular-gain} is absorbed into $U^{-\beta}$ by
choosing $N>\beta$ and increasing the implicit constant.

Substituting \eqref{eq:proof-C-smoothed} and
\eqref{eq:proof-angular-use} into \eqref{eq:proof-B-smoothed}, and
using the uniform bound $N\leq N_0$, yields
\begin{equation}\label{eq:proof-master-bound}
Q_R(\nu,e)
\lesssim
U^{-\beta}+U^{-1}
+e^{-\epsilon k}U^{2r-1}
+e^{-ck}.
\end{equation}
Since $U\asymp e^{\delta k}$, the first two terms decay exponentially
in $k$, while
$e^{-\epsilon k}U^{2r-1}
\lesssim
e^{-[\epsilon-(2r-1)\delta]k}.$
The exponent is positive by
\eqref{eq:smoothing-renewal-balance-intro}. Hence
$Q_R(\nu,e)\lesssim e^{-c_*k}$
for some $c_*>0$. Since $k\asymp\log R$, it follows that
\[
Q_R(\nu,e)\lesssim R^{-\gamma}
\]
for some $\gamma>0$, uniformly in $e\in E$.

For $\rho=(\pi_e)_{\#}\nu$, the definition of $Q_R$ gives
\[
\|\Delta_R\rho\|_{L^1}
=
Q_R(\nu,e)
\lesssim
R^{-\gamma}.
\]
By the standard criterion recalled in
Subsection~\ref{subsec:main-technical-intro} (see also \cite[Lemma~2.4]{ARHWSmoothProjections2026}), this implies
$\rho\ll\Leb^1$. This completes the proof.
\end{proof}

\begin{remark}\label{rem:stopped-besov}
As noted in Subsection~\ref{subsec:main-technical-intro}, the conclusion
of Theorem~\ref{thm:stopped-intro} is slightly stronger than just absolute
continuity. Indeed, the proof gives
$\|\Delta_R \left( (\pi_e)_{\#}\nu \right) \|_{L^1}\lesssim R^{-\gamma},
\, R\geq2,$
uniformly in $e\in E$. Hence, for every $0<t<\gamma$,
$\sum_{j\geq1}
2^{tj}\|\Delta_{2^j} \left( (\pi_e)_{\#}\nu \right) \|_{L^1}
<\infty.$
Thus, by e.g. \cite[Lemma~2.4]{ARHWSmoothProjections2026}, the density of $(\pi_e)_{\#}\nu$ belongs to the inhomogeneous Besov space
$B^t_{1,1}(\R)$, with a bound uniform in $e\in E$.
\end{remark}
\section{Uniform $q$-mass of cylinders of self-conformal measures}
\label{sec:conformal-qmass}

Recall the definition of self-conformal measures on the line from
\eqref{eq:self-conformal}, and the planar version from
Section~\ref{subsec:planar-conformal-intro}. In this section we give a
full proof 
of \eqref{eq:self-affine-lq-limit} in the self-conformal case. Recall that the  right-continuity
of $L^q$ dimensions at $q=1$ was first noted by Shmerkin and Solomyak
\cite[Remark~5.2]{Shmerkin2016Solomyak}. Recall also that our
self-conformal measures are always assumed to have strictly positive
weights.

Recall from \eqref{eq:qmass-intro} that
$S_{m,q}(\eta)
=
\sum_{Q\in\mathcal D_m}\eta(Q)^q,$ and from \eqref{eq:Dq-prelim} that \newline
$D_q(\eta)
=
\liminf_{m\to\infty}
\frac{-\log S_{m,q}(\eta)}
{(q-1)m\log 2}.$
\begin{thm}\label{thm:conformal-uniform-qmass}
Let $\theta$ be a self-conformal measure, either on the line with
respect to a $C^2$ IFS, or on the plane with respect to a $C^\omega$
IFS. Then
\begin{equation}\label{eq:conformal-Dq-continuity}
\lim_{q\downarrow1}D_q(\theta)=\dim \theta.
\end{equation}
In particular, for every $d_0<\dim \theta$ there are
\[
1<q<2,
\quad
d_0<S<D_q(\theta),
\quad
B<\infty,
\]
such that
\begin{equation}\label{eq:conformal-uniform-qmass}
S_{m,q}(\theta)
\leq
B\,2^{-m(q-1)S}
\quad m\geq1.
\end{equation}
\end{thm}
Together with the following lemma, Theorem~\ref{thm:conformal-uniform-qmass}
provides the $q$-mass input needed for condition~\textup{(A)} of
Theorem~\ref{thm:stopped-intro} in the proofs of
Theorem~\ref{thm:convolution-intro}\textup{(A)} and
Theorem~\ref{thm:planar-conformal-intro}. In
Theorem~\ref{thm:convolution-intro}\textup{(B)},  the Ahlfors--David regular
factor can be handled directly.

\begin{lemma}\label{cor:bilip-qmass}
Let $\eta \in \mathcal{P}(\R^d)$,
$d\in\{1,2\}$, and suppose that
\begin{equation}\label{eq:bilip-qmass-hyp}
S_{m,q}(\eta)
\leq
B\,2^{-m(q-1)S}
\qquad(m\geq1)
\end{equation}
for some $q>1$, $S>0$, and $B<\infty$.
Let $\mathcal T$ be a family of maps
$T:\supp\eta\to\R^d$
for which there exists $L\geq1$ such that
\[
L^{-1}|x-y|
\leq
|T(x)-T(y)|
\leq
L|x-y|
\quad
x,y\in\supp\eta,\, T\in\mathcal T.
\]
Then there exists
$B'=B'(B,q,S,L,d)<\infty$
such that
\begin{equation}\label{eq:bilip-qmass-conclusion}
S_{m,q}(T_\#\eta)
\leq
B'\,2^{-m(q-1)S}
\quad m\geq1, \, T\in\mathcal{T}.
\end{equation}
\end{lemma}

The proof is an adaptation of an argument of Shmerkin and Solomyak
\cite[Theorem 5.1]{Shmerkin2016Solomyak} for self-similar measures. They observed that their method extends from
the self-similar to the self-conformal setting; we give the details here
in the form needed for our applications.

\begin{proof}[Proof of Theorem~\ref{thm:conformal-uniform-qmass}]

\medskip
\noindent\textbf{Step 1: uniform geometry of normalized cylinders.}
Fix a base point $x_0$ in the invariant interval in dimension one,
and in $D$ in dimension two. For a finite word
$w=w_1\cdots w_n$, write
$f_w:=f_{w_1}\circ\cdots\circ f_{w_n},$ and $a_w:=|f_w'(x_0)|.$

In dimension one, the standard bounded distortion property
\cite[Theorem~2.1]{algom2020decay} gives a constant $L\geq1$ such that
\begin{equation}\label{eq:1d-distortion-section4}
L^{-1}
\leq
\frac{|f_w'(x)|}{|f_w'(y)|}
\leq
L, \text{ for every word w } \text{ and points } x,y \text{ in the invariant interval}.
\end{equation}
Since
$f_w'$ has constant sign, the mean value theorem gives
\begin{equation}\label{eq:1d-quasisim-section4}
L^{-1}a_w|x-y|
\leq
|f_w(x)-f_w(y)|
\leq
La_w|x-y|, \text{ for every word w } \text{ and points } x,y \text{ in the invariant interval}.
\end{equation}
Thus, after translating $f_w(x_0)$ to the origin and rescaling by
$a_w^{-1}$, the cylinder maps form a uniformly bi-Lipschitz family.

In the planar case, after slightly enlarging $D$ if necessary, the
standard bounded distortion property for holomorphic conformal IFSs
gives, uniformly over all finite words $w$,
\begin{equation}\label{eq:planar-distortion-section4}
L^{-1}
\leq
\frac{|f_w'(x)|}{|f_w'(y)|}
\leq
L
\qquad(x,y\in D_1);
\end{equation}
see \cite[equation (4)]{algom2024plane}. By \cite[(2.1)--(2.2)]{PeresRamsSimonSolomyak2001}, there is a constant $C\geq1$,
independent of $w$, such that
\begin{equation}\label{eq:common-quasisim-section4}
C^{-1}a_w|x-y|
\leq
|f_w(x)-f_w(y)|
\leq
Ca_w|x-y|
\end{equation}
for every finite word $w$ and all $x,y\in D$.
Consequently, the normalized maps
\[
z\longmapsto
\frac{f_w(z)-f_w(x_0)}{f_w'(x_0)}
\]
form a uniformly bi-Lipschitz family on $D$. Since they send $x_0$ to
the origin and $D$ is bounded, their images are contained in one fixed
ball.

\medskip

\noindent\textbf{Step 2: comparing dyadic and symbolic cylinders.}
Let $\mathcal A$ be the finite alphabet indexing the IFS, and let
$p=(p_i)_{i\in\mathcal A}$ be the strictly positive probability vector
defining $\theta$. Put
$\Omega=\mathcal A^{\N}$ and
$\mathbb P=p^{\N} \in \mathcal{P} (\Omega)$,
and let $\Pi:\Omega\to\R^d$ be the coding map, so that
$\theta=\Pi_\#\mathbb P$. For $\omega\in\Omega$ and $k\geq1$, write
$\omega|k=\omega_1\cdots\omega_k.$

For $n\geq1$, define the stopping time
\begin{equation}\label{eq:dyadic-stop-section4}
\tau_n(\omega)
:=
\min\left\{
k\geq1:
a_{\omega|k}\leq2^{-n}
\right\}.
\end{equation}
Let
$\mathcal W_n
:=
\left\{
\omega|\tau_n(\omega):
\omega\in\Omega
\right\}$.
Then the cylinders $\{[w]:w\in\mathcal W_n\}$ form a partition of
$\Omega$, where $[w]:=\{\omega\in\Omega:\omega|_{|w|}=w\}$.
By the definition of $\tau_n$ and
the uniform lower bound on the one-step derivatives, there exists
$c>0$, depending only on the IFS, such that
\begin{equation}\label{eq:stop-scale-section4}
c\,2^{-n}
\leq
a_w
\leq
2^{-n}
\quad
n\geq1,\ w\in\mathcal W_n.
\end{equation}
Combining \eqref{eq:stop-scale-section4} with
\eqref{eq:common-quasisim-section4}, every $f_w(K)$,
$w\in\mathcal W_n$, has diameter $O(2^{-n})$, while $f_w^{-1}$ is
$O(2^n)$-Lipschitz on $f_w(K)$.

We shall repeatedly use the following standard partition comparison:
if $\mathcal P$ and $\mathcal Q$ are measurable partitions of a
probability space and every atom of $\mathcal P$ meets at most $M$
atoms of $\mathcal Q$, then, for every $q\geq1$,
\[
\sum_{P\in\mathcal P}\mu(P)^q
\leq
M^{q-1}
\sum_{Q\in\mathcal Q}\mu(Q)^q;
\]
see \cite[Lemma~5.3]{Shmerkin2016Solomyak}. In particular, if the
incidence is bounded by $M$ in both directions, the two $q$-mass sums
are comparable up to factors $M^{q-1}$.

For each $w\in\mathcal W_n$, choose a dyadic cube
$Q_w\in\mathcal D_n$ meeting $f_w(K)$, and, for
$Q\in\mathcal D_n$, let
\[
A_Q
:=
\bigcup_{\substack{w\in\mathcal W_n\\ Q_w=Q}}[w],
\quad
P_Q:=\mathbb P(A_Q).
\]
The nonempty sets $A_Q$ form a measurable partition of $\Omega$.
By \eqref{eq:stop-scale-section4} and
\eqref{eq:common-quasisim-section4}, there exists a constant
$M_0\geq1$, depending only on the IFS and the ambient dimension, such
that each $A_Q$ meets at most $M_0$ atoms of the partition
$\bigl\{\Pi^{-1}(R):R\in\mathcal D_n\bigr\},$
and conversely. Hence, by
\cite[Lemma~5.3]{Shmerkin2016Solomyak},
\begin{equation}\label{eq:coarse-symbolic-qmass-section4}
\sum_{Q\in\mathcal D_n}P_Q^q
\leq
M_0^{q-1}S_{n,q}(\theta)
\qquad(q\geq1,\ n\geq1).
\end{equation}

\medskip
\medskip
\noindent\textbf{Step 3: submultiplicativity with the correct loss.}
We claim that there exists $M\geq1$, depending only on  
the IFS, such that
\begin{equation}\label{eq:qmass-submultiplicativity-section4}
S_{n+m,q}(\theta)
\leq
M^{q-1}S_{n,q}(\theta)S_{m,q}(\theta)
\quad\, n,m\geq1.
\end{equation}
Fix $n,m\geq1$ and $w\in\mathcal W_n$, and put
$p_w:=\mathbb P([w]).$
For $R\in\mathcal D_{n+m}$, set
$\mu_w(R)
:=
\mathbb P\bigl(\Pi^{-1}(R)\cap[w]\bigr).$
Identifying $[w]$ with $\Omega$ by deleting the prefix
$w$,  independence (by Bernoulli) gives
$\mu_w(R)
=
p_w\,\theta\bigl(f_w^{-1}(R)\bigr).$

By \eqref{eq:common-quasisim-section4} and
\eqref{eq:stop-scale-section4}, the set
$f_w^{-1}(R)\cap K$ has diameter $O(2^{-m})$, uniformly in
$n,m,w,R$. Conversely, if $Q\in\mathcal D_m$ meets $K$, then
$f_w(Q\cap K)$ has diameter $O(2^{-(n+m)})$. Hence the partition of
$\Omega$ induced by
\[
\bigl\{f_w^{-1}(R):R\in\mathcal D_{n+m}\bigr\}
\]
and the pullback under $\Pi$ of $\mathcal D_m$ have uniformly bounded
incidence, in both directions. By the partition comparison used in
Step~2, there exists $M_1\geq1$, depending only on the IFS,
such that
\begin{equation}\label{eq:fine-inside-cylinder-section4}
\sum_{R\in\mathcal D_{n+m}}\mu_w(R)^q
\leq
M_1^{q-1}p_w^q S_{m,q}(\theta).
\end{equation}

Recalling the coarse partition from
Step~2, for $Q\in\mathcal D_n$ let
\[
\mathcal W_n(Q)
:=
\{w\in\mathcal W_n:Q_w=Q\},
\quad
P_Q
=
\sum_{w\in\mathcal W_n(Q)}p_w, \text{ and define }
\mu_Q(R)
:=
\sum_{w\in\mathcal W_n(Q)}\mu_w(R).
\]
Since the cylinders $[w]$, $w\in\mathcal W_n$, partition $\Omega$,
$\theta(R)
=
\sum_{Q\in\mathcal D_n}\mu_Q(R).$

There exists $M_2\geq1$, depending only on the IFS  and the
ambient dimension, such that for every
$R\in\mathcal D_{n+m}$ there are at most $M_2$ cubes
$Q\in\mathcal D_n$ for which $\mu_Q(R)>0$. Indeed, if
$\mu_Q(R)>0$, then for some $w\in\mathcal W_n(Q)$ the set $f_w(K)$
meets $R$, while by construction $Q$ also meets $f_w(K)$; the diameter
bound from Step~2 therefore forces $Q$ to lie within
$O(2^{-n})$ of $R$.
It follows that
\[
\theta(R)^q
\leq
M_2^{q-1}
\sum_{Q\in\mathcal D_n}\mu_Q(R)^q.
\]

Summing over $R\in\mathcal D_{n+m}$ and using Minkowski's inequality,
followed by \eqref{eq:fine-inside-cylinder-section4}, gives
\begin{align*}
S_{n+m,q}(\theta)
&\leq
M_2^{q-1}
\sum_{Q\in\mathcal D_n}
\sum_{R\in\mathcal D_{n+m}}\mu_Q(R)^q
\\
&\leq
M_2^{q-1}
\sum_{Q\in\mathcal D_n}
\left(
\sum_{w\in\mathcal W_n(Q)}
\left(
\sum_{R\in\mathcal D_{n+m}}\mu_w(R)^q
\right)^{1/q}
\right)^q
\\
&\leq
(M_1M_2)^{q-1}
S_{m,q}(\theta)
\sum_{Q\in\mathcal D_n}P_Q^q.
\end{align*}
Finally, \eqref{eq:coarse-symbolic-qmass-section4} yields
\[
S_{n+m,q}(\theta)
\leq
(M_0M_1M_2)^{q-1}
S_{n,q}(\theta)S_{m,q}(\theta),
\]
which proves \eqref{eq:qmass-submultiplicativity-section4} with
$M=M_0M_1M_2$.

\medskip
\medskip
\noindent\textbf{Step 4: right-continuity at $q=1$.}

Let
\[
H_n(\theta)
:=
\sum_{Q\in\mathcal D_n^{(d)}}
\theta(Q)\log\frac{1}{\theta(Q)}
\]
Since  self-conformal measures (in both settings) are exact
dimensional by Feng--Hu \cite[Theorem~2.8]{feng2009dimension}, their
entropy dimension agrees with their Hausdorff dimension:
\begin{equation}\label{eq:entropy-dimension-section4}
\lim_{n\to\infty}
\frac{H_n(\theta)}{n\log 2}
=
\dim\theta.
\end{equation}
Set
$A_n(q)
:=
\log S_{n,q}(\theta)+(q-1)\log M.$
By \eqref{eq:qmass-submultiplicativity-section4},
$A_{n+m}(q)\leq A_n(q)+A_m(q),$
so Fekete's lemma gives
\[
\lim_{n\to\infty}\frac{A_n(q)}{n}
=
\inf_{n\geq1}\frac{A_n(q)}{n}.
\]
Since $1-q<0$, this is equivalent to
\begin{equation}\label{eq:Dq-fixed-scale-section4}
D_q(\theta)
=
\sup_{n\geq1}
\left[
\frac{\log S_{n,q}(\theta)}
     {(1-q)n\log 2}
-
\frac{\log M}{n\log 2}
\right].
\end{equation}

For each fixed $n$, the collection of nonzero masses
\[
\{\theta(Q):Q\in\mathcal D_n^{(d)},\ \theta(Q)>0\}
\]
is a finite probability vector. Hence, by differentiating at $q=1$,
\begin{equation}\label{eq:renyi-to-shannon-section4}
\lim_{q\downarrow1}
\frac{\log S_{n,q}(\theta)}
     {(1-q)n\log 2}
=
\frac{H_n(\theta)}{n\log 2}.
\end{equation}
Fix $\varepsilon>0$. By \eqref{eq:entropy-dimension-section4}, choose
$n$ sufficiently large that
\[
\frac{H_n(\theta)}{n\log 2}
>
\dim\theta-\varepsilon
\qquad\text{and}\qquad
\frac{\log M}{n\log 2}<\varepsilon.
\]
Keeping this $n$ fixed, \eqref{eq:renyi-to-shannon-section4} implies
that, for all $q>1$ sufficiently close to $1$,
\[
\frac{\log S_{n,q}(\theta)}
     {(1-q)n\log 2}
>
\frac{H_n(\theta)}{n\log 2}-\varepsilon.
\]
Using \eqref{eq:Dq-fixed-scale-section4}, we obtain
$D_q(\theta)
>
\dim\theta-3\varepsilon.$
Thus
\[
\liminf_{q\downarrow1}D_q(\theta)
\geq
\dim\theta.
\]

For the reverse inequality, for every finite probability vector
$(p_i)$ and every $q>1$,
$\frac{1}{1-q}\log\sum_i p_i^q
\leq
\sum_i p_i\log\frac{1}{p_i}.$
Applying this to the dyadic masses $\{\theta(Q):Q\in\mathcal D_n^{(d)}\}$
gives
\[
\frac{\log S_{n,q}(\theta)}
     {(1-q)n\log 2}
\leq
\frac{H_n(\theta)}{n\log 2}.
\]
Passing to the limit in $n$ and using
\eqref{eq:entropy-dimension-section4}, we get
$D_q(\theta)\leq\dim\theta.$
Therefore
\[
\lim_{q\downarrow1}D_q(\theta)=\dim\theta,
\]
which proves \eqref{eq:conformal-Dq-continuity}.

Finally, let $d_0<\dim\theta$. By
\eqref{eq:conformal-Dq-continuity}, we may choose $q\in(1,2)$
sufficiently close to $1$, and then choose
\[
d_0<S<D_q(\theta).
\]
The estimate \eqref{eq:qmass-from-Dq}, after enlarging the constant to
cover finitely many initial scales, gives
\eqref{eq:conformal-uniform-qmass}. This completes the proof.
\end{proof}

\begin{proof}[Proof of Lemma~\ref{cor:bilip-qmass}]
Fix $T\in\mathcal T$. By assumption,
\[
L^{-1}|x-y|
\leq
|T(x)-T(y)|
\leq
L|x-y|
\quad x,y\in\supp\eta.
\]
Hence there exists $M=M(L,d)\geq1$ such that, for every $m\geq1$,
each set
$T^{-1}(R)\cap\supp\eta,$
for $R\in\mathcal D_m^{(d)},$
meets at most $M$ cubes of $\mathcal D_m^{(d)}$, and conversely the
image under $T$ of each $Q\cap\supp\eta$,
$Q\in\mathcal D_m^{(d)}$, meets at most $M$ cubes of
$\mathcal D_m^{(d)}$. Thus the dyadic partition at level $m$ and its
pullback under $T$ have incidence bounded by $M$ in both directions.

By the partition comparison used in Step~2 of the proof above,
$S_{m,q}(T_\#\eta)
\leq
M^{q-1}S_{m,q}(\eta).$
Using \eqref{eq:bilip-qmass-hyp}, we obtain
\[
S_{m,q}(T_\#\eta)
\leq
M^{q-1}B\,2^{-m(q-1)S}
\quad m\geq1 .
\]
Therefore \eqref{eq:bilip-qmass-conclusion} holds with
$B'=M^{q-1}B,$
uniformly in $T\in\mathcal T$.
\end{proof}

\section{Absolute continuity of convolutions}\label{sec:convolution-proof}
We now prove Theorem~\ref{thm:convolution-intro}. Recall from
Subsection~\ref{subsec:convolutions-intro} that a $C^2$ IFS on a compact
interval is a finite family
$\Phi=\{f_i:i\in\mathcal A\}$
satisfying \eqref{eq:1d-uniform-contraction}, and that a self-conformal
measure for $\Phi$ is a probability measure satisfying
\eqref{eq:self-conformal} for some strictly positive probability vector.
Recall also that $\Phi$ is linear if $f_i''=0$ on $K_\Phi$ for every
$i$, and is $C^2$-conjugate to linear if a $C^2$ change of coordinates
conjugates it to such a system. Finally, a probability measure is
$\alpha$-Ahlfors--David regular if
\[
C^{-1}r^\alpha\leq \nu(B(x,r))\leq Cr^\alpha
\quad
x\in\supp\nu,\ 0<r\leq\diam(\supp\nu)
\]
for some $C\geq1, \alpha\in (0,1)$.

Throughout this section $\mu$ denotes the nonlinear factor in
Theorem~\ref{thm:convolution-intro}: it is self-conformal for a $C^2$
IFS which is not $C^2$-conjugate to linear. The second factor $\nu$ is
either self-conformal for a $C^2$ IFS or Ahlfors--David regular. Our
goal is to prove
\[
\dim\mu+\dim\nu>1
\qquad\Longrightarrow\qquad
\mu*\nu\ll\Leb^1.
\]
The utility of the nonlinearity assumption will become explicit in
Subsection~\ref{subsec:convolution-angular}. It is this property that will allow us to exhibit the key condition (C) in Theorem \ref{thm:stopped-intro}.

Recall that convolution is the diagonal projection of the product measure. So, set
$e_0=2^{-1/2}(1,1) \in \Sone.$
Writing $\Sigma(x,y)=x+y$, by Lemma~\ref{lem:L1-scaling}
\[
\Sigma=\sqrt2\,\pi_{e_0},
\, \text{ and }
\|\Delta_R(\mu*\nu)\|_1
=
Q_{\sqrt2 R}(\mu\times\nu,[e_0])
\]
We suppress this  fixed factor
$\sqrt2$ below. 
\subsection{The derivative cocycle and the stopping rules}
\label{subsec:convolution-stopping}
Write
$
\Phi=\{f_i:i\in\mathcal A\}
$
for the IFS generating $\mu=\mu_{\mathbf p}$, and let
$\mathbb P=\mathbf p^{\N}$ be the corresponding Bernoulli measure on
$\Omega=\mathcal A^{\N}$.  Fix an arbitrary $x_0\in I$, and write
$
x_\omega
=
\lim_{n\to\infty}f_{\omega|_n}(x_0)
\in K_\Phi
$
for the coding map.  As usual, for
$w=w_1\cdots w_n\in\mathcal A^*$ we use the convention
$
f_w=f_{w_1}\circ\cdots\circ f_{w_n}.
$

For $\omega\in\Omega$, define
$$
S_n(\omega)
=
-\log\left|f'_{\omega|_n}(x_{\sigma^n\omega})\right|,
\qquad
\varepsilon_n(\omega)
=
\sgn f'_{\omega|_n}(x_{\sigma^n\omega})
\in\{\pm1\}.
$$
Thus $
f'_{\omega|_n}(x_{\sigma^n\omega})
=
\varepsilon_n(\omega)e^{-S_n(\omega)}.
$
We refer to $S_n$ as the derivative cocycle and to $\varepsilon_n$ as its
orientation component.  By the chain rule,
$
S_{n+1}(\omega)-S_n(\omega)
=
-\log\left|f'_{\omega_{n+1}}
(x_{\sigma^{n+1}\omega})\right|.
$
Hence the uniform contraction assumption
\eqref{eq:1d-uniform-contraction} gives constants
$0<D\leq D'<\infty$ such that
$
D
\leq
S_{n+1}(\omega)-S_n(\omega)
\leq
D'
$
for every $\omega\in\Omega$ and $n\geq0$.

For $\omega \in \Omega$, put
\begin{equation}\label{eq:convolution-first-stop}
\tau_k(\omega)=\min\{n:S_n(\omega)\geq k\},
\qquad
Y_k(\omega)=S_{\tau_k(\omega)}(\omega)-k\in[0,D'].
\end{equation}
Thus
$
f'_{\omega|\tau_k(\omega)}
\bigl(x_{\sigma^{\tau_k(\omega)}\omega}\bigr)
=
\varepsilon_{\tau_k(\omega)}(\omega)e^{-k-Y_k(\omega)}.
$
Define a measurable partition $\mathcal P_k$ of $\Omega$ by declaring
$$
\omega\sim_k\omega'
\quad\Longleftrightarrow\quad
\sigma^{\tau_k(\omega)}\omega
=
\sigma^{\tau_k(\omega')}\omega'.
$$
Thus the cells of $\mathcal P_k$ consist of those sequences having the same
 tail after the stopping rule $\tau_k$.  For $\mathbb P$-a.e.
$\xi\in\Omega$, let
$$
P_{k,\xi}:=\mathbb P_{\mathcal P_k(\xi)}
$$
denote the conditional measure of $\mathbb P$ on the cell
$\mathcal P_k(\xi)$.  Accordingly, for every integrable $F$,
$
\int F\,d\mathbb P
=
\int_\Omega\left(\int F\,dP_{k,\xi}\right)d\mathbb P(\xi).
$

It is worth emphasizing that the partition $\mathcal P_k$ is not, in general, a partition associated to a bona-fide stopping time.  Although $S_n$ is an additive derivative cocycle, its value need not be determined by the prefix $\omega|_{n}$, since the evaluation point $x_{\sigma^n\omega}$ may well depend on the full tail.  So, $\tau_k$ need not be a stopping time for the usual forward symbolic filtration. In some special cases such as a self-similar IFS, $S_n$ is actually prefix-measurable and this distinction disappears.  In the general self-conformal setting we will therefore need to compare $\tau_k$ below with a genuine prefix stopping time at the same geometric scale. We refer to \cite[Section 4]{algom2020decay} for more discussion of this subtle issue.

Fix a base point $x_0$ in the invariant interval and define for $\omega \in \Omega$, 
$$
\beta_k(\omega)
=
\min\big\{m:|f'_{\omega|m}(x_0)|<e^{-2k}\big\}.
$$
Unlike $\tau_k$,  $\beta_k$ is determined by a finite prefix and is therefore a genuine stopping time for the forward symbolic filtration.  By bounded distortion \eqref{eq:1d-distortion-section4} and \eqref{eq:convolution-first-stop}, for all sufficiently large $k$ we have $\beta_k>\tau_k$.

We will use the following elementary consequence of the distortion estimates from Section~\ref{sec:conformal-qmass}.

\begin{lemma}\label{lem:convolution-relative-Taylor}
There exists $C\geq1$ such that, for every $w\in\mathcal A^*$ and every $x,y\in I$,
\begin{equation}\label{eq:convolution-relative-Taylor}
\left|
f_w(x)-f_w(y)-f_w'(y)(x-y)
\right|
\leq
C|f_w'(y)|\,|x-y|^2.
\end{equation}
\end{lemma}

\begin{proof}
Since the maps $f_i$ are $C^2$ and non-singular on $I$, there is
$M<\infty$ such that
$
\left|\frac{f_i''(x)}{f_i'(x)}\right|\leq M
$
for any $i\in\mathcal A,\ x\in I.
$
Let $w=w_1\cdots w_n$.  By the chain rule,
$$
\frac{f_w''(x)}{f_w'(x)}
=
\sum_{j=1}^n
\frac{f_{w_j}''}{f_{w_j}'}
\left(f_{w_{j+1}\cdots w_n}(x)\right)
f_{w_{j+1}\cdots w_n}'(x),
$$
where the last suffix is understood to be the identity.  Hence, by
\eqref{eq:1d-uniform-contraction},
$$
\left|\frac{f_w''(x)}{f_w'(x)}\right|
\leq
M\sum_{j=1}^n\rho_+^{\,n-j}
\leq
\frac{M}{1-\rho_+},
$$
uniformly in $w$ and $x$.
It follows that
$
|f_w''(u)|
\lesssim
|f_w'(u)|$
uniformly in $w$ and $u$.  By the bounded distortion estimate
\eqref{eq:1d-distortion-section4},
$
|f_w'(u)|\lesssim |f_w'(y)|$
for any $u,y\in I),
$
and therefore
$
\sup_{u\in I}|f_w''(u)|
\lesssim
|f_w'(y)|.
$
Finally, Taylor's formula with integral remainder gives
$$
f_w(x)-f_w(y)-f_w'(y)(x-y)
=
\int_y^x (x-u)f_w''(u)\,du,
$$
so
$
\left|
f_w(x)-f_w(y)-f_w'(y)(x-y)
\right|
\lesssim
|f_w'(y)|\,|x-y|^2.
$
This proves \eqref{eq:convolution-relative-Taylor}.
\end{proof}

By \eqref{eq:convolution-first-stop}, the definition of $\beta_k$, and
the bounded distortion estimate \eqref{eq:1d-distortion-section4},
\[
\left|f'_{\omega|\tau_k(\omega)}(x)\right|
\asymp e^{-k},
\quad
\left|f'_{\omega|\beta_k(\omega)}(x)\right|
\asymp e^{-2k},
\text{ for } x\in I.
\]
Hence, by the chain rule,
\begin{equation}\label{eq:convolution-suffix-scale}
\left|
f'_{\sigma^{\tau_k(\omega)}\omega|
\,\beta_k(\omega)-\tau_k(\omega)}(x)
\right|
\asymp e^{-k},
\qquad x\in I,
\end{equation}
uniformly in $k$ and $\omega$.  This is the analogue of
\cite[Lemma~4.4, equations~(36)--(37)]{algom2020decay}.

The same argument also gives the following finite-remainder
decomposition; compare the discussion following
\cite[Lemma~4.4]{algom2020decay}, and in particular
\cite[equation~(39)]{algom2020decay}.

\begin{proposition}\label{prop:convolution-finite-remainder}
There exists a finite set $\mathcal R\subset\mathcal A^*$ with the
following property. For every $k$ and every $\xi\in\Omega$, writing
$\zeta=\sigma^{\tau_k(\xi)}\xi,$
there is a prefix $c_{\xi,k}$ of $\zeta$ such that
\[
\left|c'_{\xi,k}(x)\right|\asymp e^{-k},
\qquad x\in I,
\]
and, for every $\omega\in\mathcal P_k(\xi)$, there exists
$\rho_{\omega,k}\in\mathcal R$ for which
\begin{equation}\label{eq:convolution-finite-menu}
f_{\omega|\beta_k(\omega)}
=
f_{\omega|\tau_k(\omega)}
\circ c_{\xi,k}
\circ f_{\rho_{\omega,k}}.
\end{equation}
\end{proposition}
This is precisely the finite-remainder argument of
\cite[Lemma~4.4 and equation~(39)]{algom2020decay}: after removing a
uniformly bounded  word, the portion of
$\omega$ between $\tau_k(\omega)$ and $\beta_k(\omega)$ has a common
prefix depending only on the  tail $\zeta$.

For $w\in\mathcal A^*$, write
\begin{equation}\label{eq:convolution-normalized-word}
N_w(x)
=
\frac{f_w(x)-f_w(x_0)}{f_w'(x_0)}.
\end{equation}
By \eqref{eq:1d-quasisim-section4}, the maps $N_w$, $w\in\mathcal A^*$,
are uniformly bi-Lipschitz on $I$ and their images are contained in a
fixed compact interval. When a second self-conformal IFS is considered,
we use the analogous normalization with its fixed base point.
\subsection{Decompositions and renormalizations: Condition \textup{(B)}}
\label{subsec:convolution-renormalization}
\subsubsection{Self-conformal measure as second factor}
Assume first that $\nu$, the other factor, is a self-conformal measure
generated by a $C^2$ IFS
$\Psi=\{g_j:j\in\mathcal B\}.$
Let $\widetilde\Omega=\mathcal B^{\mathbb N}$, let $\mathbb Q$ be the
corresponding Bernoulli measure, and write $y_\eta$ for the coded point
associated to $\eta\in\widetilde\Omega$.

We use the same notation as for the first IFS, with tildes. Thus
\[
\widetilde S_n(\eta)
=
-\log\left|g'_{\eta|n}(y_{\sigma^n\eta})\right|,
\qquad
\widetilde\varepsilon_n(\eta)
=
\sgn g'_{\eta|n}(y_{\sigma^n\eta}),
\]
so that
$g'_{\eta|n}(y_{\sigma^n\eta})
=
\widetilde\varepsilon_n(\eta)e^{-\widetilde S_n(\eta)}.$
The distortion and cocycle estimates from
Subsection~\ref{subsec:convolution-stopping} apply verbatim. Define
$\widetilde\tau_k(\eta)
=
\min\{n:\widetilde S_n(\eta)\geq k\},$
and
\begin{equation}\label{eq:convolution-second-overshoot}
\widetilde Y_k(\eta)
=
\widetilde S_{\widetilde\tau_k(\eta)}(\eta)-k.
\end{equation}
As before, $\widetilde Y_k$ is uniformly bounded.

Fix $y_0 \in  K_\Psi$, and define
the   stopping time
\[
\widetilde\beta_k(\eta)
=
\min\left\{
m:
\left|g'_{\eta|m}(y_0)\right|<e^{-2k}
\right\}.
\]
By bounded distortion \eqref{eq:1d-distortion-section4}, for all  large $k$,
$\widetilde\beta_k(\eta)>\widetilde\tau_k(\eta)$. Put
\[
\widetilde c_{\eta,k}
=
g_{\sigma^{\widetilde\tau_k(\eta)}\eta|
\,\widetilde\beta_k(\eta)-\widetilde\tau_k(\eta)}.
\]
Exactly as in the derivation of
\eqref{eq:convolution-suffix-scale}, the definitions of the two stopping
rules, together with \eqref{eq:1d-uniform-contraction} and
\eqref{eq:1d-distortion-section4}, give
\begin{equation}\label{eq:convolution-second-suffix-scale}
\left|\widetilde c'_{\eta,k}(y)\right|
\asymp e^{-k},\, \text{ uniformly in } \eta, k, y.
\end{equation}

We now define the normalized pieces which will play the role of the
measures $\eta_{\xi,j}$ in Theorem~\ref{thm:stopped-intro}. For
$\xi\in\Omega$, let $c_{\xi,k}$ be supplied by
Proposition~\ref{prop:convolution-finite-remainder}. For
$\rho\in\mathcal R$, set
\begin{equation}\label{eq:convolution-selfconf-remainder}
\Theta_{\xi,\eta,\rho,k}
=
(N_{c_{\xi,k}}\circ f_\rho)_\#\mu
\times
(N_{\widetilde c_{\eta,k}})_\#\nu,
\end{equation}
where in the second factor $N_w$ is defined similarly to \eqref{eq:convolution-normalized-word}. By Proposition~\ref{prop:convolution-finite-remainder},
\eqref{eq:convolution-second-suffix-scale}, and the discussion following
\eqref{eq:convolution-normalized-word}, the maps
$N_{c_{\xi,k}}\circ f_\rho$ and $N_{\widetilde c_{\eta,k}}$
have images contained in fixed compact intervals, uniformly in all
parameters. Hence the measures $\Theta_{\xi,\eta,\rho,k}$ are supported
in one fixed compact subset of $\mathbb R^2$.

We next specify the probability space and kernels from Theorem \ref{thm:stopped-intro}.
Let
\begin{equation}\label{eq:convolution-selfconf-conditioning-space}
(\Xi,m)
=
(\Omega\times\widetilde\Omega,\mathbb P\times\mathbb Q),
\qquad
\chi=(\xi,\eta),
\end{equation}
and let
$\mathcal J=\mathcal R\times\{\pm1\}.$
Thus $\mathcal J$ is a fixed finite set.
For $s\in\{\pm1\}$ and $t\in[0,D']$, put
\begin{equation}\label{eq:convolution-v-selfconf}
v_{\chi,s,k}(t)
=
\left(
s e^k c'_{\xi,k}(x_0)e^{-t},
\,
e^k
\widetilde\varepsilon_{\widetilde\tau_k(\eta)}(\eta)
\widetilde c'_{\eta,k}(y_0)
e^{-\widetilde Y_k(\eta)}
\right) \in\mathbb R^2,
\end{equation}
and define
$$\Phi_{\chi,s,k}(t)
=
\left(
[v_{\chi,s,k}(t)]_{\Pone},
\log|v_{\chi,s,k}(t)|
\right).$$
By Proposition~\ref{prop:convolution-finite-remainder},
\eqref{eq:convolution-second-suffix-scale}, and the boundedness of
$\widetilde Y_k$, both coordinates of $v_{\chi,s,k}(t)$ have absolute
value bounded above and below by positive constants, uniformly in
$\chi$, $s$, $k$, and $t\in[0,D']$. Hence the images
of all the maps $\Phi_{\chi,s,k}$ are contained in
$\Pone\times J$ for one fixed compact interval $J$.

Finally, for $\chi=(\xi,\eta)\in\Xi$ and
$j=(\rho,s)\in\mathcal J$, define
\begin{equation}\label{eq:convolution-selfconf-kernel}
P_{\chi,j}
=
(\Phi_{\chi,s,k})_\#
\big((Y_k)_\#P_{k,\xi}\big),
\qquad
\Theta_{\chi,j}
=
\Theta_{\xi,\eta,\rho,k}.
\end{equation}
$\Theta_{\chi,j}$ and $P_{\chi,j}$ are respectively the measures
$\eta_{\xi,j}$ and the probability kernels appearing in
Theorem~\ref{thm:stopped-intro}. The corresponding limiting laws
$\Lambda_{\chi,j}$ will be defined only in
Subsection~\ref{subsec:convolution-angular}.

To unpack this definition, fix $\chi=(\xi,\eta)\in\Xi$ and
$j=(\rho,s)\in\mathcal J$, and write
$\zeta=\sigma^{\tau_k(\xi)}\xi.$
The measure $P_{k,\xi}$ is supported on the atom
$\mathcal P_k(\xi)$, so for $P_{k,\xi}$-almost every $\omega$,
$\sigma^{\tau_k(\omega)}\omega=\zeta.$
Thus all such $\omega$ have the same  tail after running $\tau_k$ digits, and the prefix map
$f_{\omega|\tau_k(\omega)}$ is linearized at the same point $x_\zeta$.
Its derivative there is
\[
f'_{\omega|\tau_k(\omega)}(x_\zeta)
=
\varepsilon_{\tau_k(\omega)}(\omega)
e^{-k-Y_k(\omega)}.
\]
After fixing the sign $s$, the only remaining dependence on $\omega$ in
this derivative is therefore the centered cocycle value $Y_k(\omega)$.
The map $c_{\xi,k}$ contributes a further derivative of size $e^{-k}$.
Hence, after linearization and after factoring out the common scale
$e^{-2k}$, the first coefficient becomes
$s\,e^k c'_{\xi,k}(x_0)e^{-Y_k(\omega)}.$
This explains both the first coordinate of
$v_{\chi,s,k}(Y_k(\omega))$ and the factor $e^k$ appearing there.
The second coordinate is obtained in exactly the same way from the
second IFS.

Recall from \eqref{eq:Q-def-intro} and \eqref{eq:H-def} that
$Q_R(\theta,\vartheta)
=
\left\|
\Delta_R\big((\pi_e)_\#\theta\big)
\right\|_{L^1},
\,
\vartheta=[e],$
and
$H_{U,\theta}(\vartheta,z)
=
Q_{Ue^z}(\theta,\vartheta).$
We can now state the required renormalization estimate.

\begin{proposition}\label{prop:convolution-renormalization-selfconf}
There exist $C<\infty$ and $k_0\in\mathbb N$ such that, for every
$k\geq k_0$ and every $R\geq1$, setting
\begin{equation}\label{eq:convolution-U}
U=Re^{-2k},
\end{equation}
we have
\begin{equation}\label{eq:convolution-condition-B-selfconf}
Q_R(\mu\times\nu,[e_0])
\leq
C
\int_\Xi
\sum_{j\in\mathcal J}
\int_{\Pone\times J}
H_{U,\Theta_{\chi,j}}(\vartheta,z)
\,dP_{\chi,j}(\vartheta,z)
\,dm(\chi)
+
CR e^{-3k}.
\end{equation}
\end{proposition}
With the choice of $k=k(R)$ made below,
\eqref{eq:convolution-condition-B-selfconf} gives
condition~\textup{(B)} of Theorem~\ref{thm:stopped-intro}. The proof follows from the arguments in \cite[Section 4]{algom2020decay}; for completeness, we explain this in detail.

\begin{proof}
Since $\beta_k$ and $\widetilde\beta_k$ are bona-fide stopping
times, 
\cite[Lemma~4.3]{algom2020decay} applied to the two factors and followed
by Fubini, gives
\begin{equation}\label{eq:convolution-product-stopping-identity}
\mu\times\nu
=
\iint
\left(
f_{\omega|\beta_k(\omega)}
\times
g_{\eta|\widetilde\beta_k(\eta)}
\right)_\#
(\mu\times\nu)
\,d\mathbb P(\omega)\,d\mathbb Q(\eta).
\end{equation}
Since projection and $\Delta_R$ are linear, Minkowski's inequality
therefore gives
\[
Q_R(\mu\times\nu,[e_0])
\leq
\iint
Q_R\left(
\left(
f_{\omega|\beta_k(\omega)}
\times
g_{\eta|\widetilde\beta_k(\eta)}
\right)_\#
(\mu\times\nu),
[e_0]
\right)
\,d\mathbb P(\omega)\,d\mathbb Q(\eta).
\]

We now disintegrate the $\omega$-integral with respect to
$\mathcal P_k$. For $\mathbb P$-almost every $\xi$, put
\[
\zeta=\sigma^{\tau_k(\xi)}\xi.
\]
Then
$\sigma^{\tau_k(\omega)}\omega=\zeta$
for $P_{k,\xi}$-almost every $\omega$. By
Proposition~\ref{prop:convolution-finite-remainder}, for such $\omega$
we have
\begin{equation}\label{eq:convolution-stopped-factorization}
f_{\omega|\beta_k(\omega)}
=
f_{\omega|\tau_k(\omega)}
\circ c_{\xi,k}
\circ f_{\rho_{\omega,k}},
\qquad
\rho_{\omega,k}\in\mathcal R.
\end{equation}
For the second factor, by definition of $\widetilde c_{\eta,k}$,
\begin{equation}\label{eq:convolution-second-stopped-factorization}
g_{\eta|\widetilde\beta_k(\eta)}
=
g_{\eta|\widetilde\tau_k(\eta)}
\circ\widetilde c_{\eta,k}.
\end{equation}
The first of these factorizations is the analogue of
\cite[equation~(39)]{algom2020decay}.

Fix $\xi,\eta,\rho$ and restrict for the moment to those $\omega$ for
which $\rho_{\omega,k}=\rho$. Write
\[
F_\omega=f_{\omega|\tau_k(\omega)},
\qquad
\widetilde F_\eta
=
g_{\eta|\widetilde\tau_k(\eta)}.
\]
We linearize $F_\omega$ at $x_\zeta$. Define
\[
L_{\omega,\xi,\rho,k}(x)
=
F_\omega(x_\zeta)
+
F_\omega'(x_\zeta)
\left(
c_{\xi,k}(f_\rho(x))-x_\zeta
\right).
\]
Since $c_{\xi,k}$ is a prefix of $\zeta$, both
$x_\zeta$ and $c_{\xi,k}(f_\rho(x))$ belong to
$c_{\xi,k}(I)$. Proposition~\ref{prop:convolution-finite-remainder}
and bounded distortion \eqref{eq:1d-distortion-section4} therefore give
$\left|
c_{\xi,k}(f_\rho(x))-x_\zeta
\right|
\lesssim e^{-k}.$
Moreover, by \eqref{eq:convolution-first-stop},
$\left|F_\omega'(x_\zeta)\right|
=
e^{-k-Y_k(\omega)}
\asymp e^{-k}.$
Hence Lemma~\ref{lem:convolution-relative-Taylor} gives
\begin{equation}\label{eq:convolution-first-linearization}
\sup_x
\left|
f_{\omega|\beta_k(\omega)}(x)
-
L_{\omega,\xi,\rho,k}(x)
\right|
\lesssim e^{-3k}.
\end{equation}
This is precisely the $C^0$ linearization used in the proof of
\cite[Claim~4.5, equations~(39)--(42)]{algom2020decay}; equation~(40)
there is the corresponding estimate for the derivative of the inner
map.

The second coordinate is similar, put
$\widetilde y_\eta
=
y_{\sigma^{\widetilde\tau_k(\eta)}\eta}$
and define
\[
\widetilde L_{\eta,k}(y)
=
\widetilde F_\eta(\widetilde y_\eta)
+
\widetilde F_\eta'(\widetilde y_\eta)
\left(
\widetilde c_{\eta,k}(y)-\widetilde y_\eta
\right).
\]
Both $\widetilde c_{\eta,k}(y)$ and $\widetilde y_\eta$ belong to
$\widetilde c_{\eta,k}(I)$, and
\eqref{eq:convolution-second-suffix-scale} gives
$\left|
\widetilde c_{\eta,k}(y)-\widetilde y_\eta
\right|
\lesssim e^{-k}.$
Also,
$\left|
\widetilde F_\eta'(\widetilde y_\eta)
\right|
=
e^{-k-\widetilde Y_k(\eta)}
\asymp e^{-k}.$
Thus the same Taylor estimate gives
\begin{equation}\label{eq:convolution-second-linearization}
\sup_y
\left|
g_{\eta|\widetilde\beta_k(\eta)}(y)
-
\widetilde L_{\eta,k}(y)
\right|
\lesssim e^{-3k}.
\end{equation}

It follows from
\eqref{eq:convolution-first-linearization} and
\eqref{eq:convolution-second-linearization} that
\[
\sup_{x,y}
\left|
\pi_{e_0}\left(
f_{\omega|\beta_k(\omega)}(x),
g_{\eta|\widetilde\beta_k(\eta)}(y)
\right)
-
\pi_{e_0}\left(
L_{\omega,\xi,\rho,k}(x),
\widetilde L_{\eta,k}(y)
\right)
\right|
\lesssim e^{-3k}.
\]
Indeed, each coordinate contributes $O(e^{-3k})$, and $\pi_{e_0}$ is
a fixed linear map. Hence Lemma~\ref{lem:uniform-displacement}, applied
to the measure $\mu\times\nu$, gives
\begin{equation}\label{eq:convolution-linearization-error}
\left|
Q_R\left(
\left(
f_{\omega|\beta_k(\omega)}
\times
g_{\eta|\widetilde\beta_k(\eta)}
\right)_\#
(\mu\times\nu),
[e_0]
\right)
-
Q_R\left(
\left(
L_{\omega,\xi,\rho,k}
\times
\widetilde L_{\eta,k}
\right)_\#
(\mu\times\nu),
[e_0]
\right)
\right|
\lesssim Re^{-3k}.
\end{equation}

We now rewrite the measure
$(\pi_{e_0})_\#
\left(
(L_{\omega,\xi,\rho,k}\times\widetilde L_{\eta,k})_\#
(\mu\times\nu)
\right)$
in terms of the measure
$\Theta_{\chi,j}$ defined in
\eqref{eq:convolution-selfconf-remainder}. By
\eqref{eq:convolution-normalized-word},
\[
c_{\xi,k}(f_\rho(x))
=
c_{\xi,k}(x_0)
+
c'_{\xi,k}(x_0)
(N_{c_{\xi,k}}\circ f_\rho)(x),
\quad \text{ and    }
\widetilde c_{\eta,k}(y)
=
\widetilde c_{\eta,k}(y_0)
+
\widetilde c'_{\eta,k}(y_0)
N_{\widetilde c_{\eta,k}}(y).
\]
Substituting these identities into
$L_{\omega,\xi,\rho,k}$ and $\widetilde L_{\eta,k}$ shows that the
corresponding projected map differs by an additive constant from
\[
\begin{aligned}
&
\varepsilon_{\tau_k(\omega)}(\omega)
e^{-k-Y_k(\omega)}
c'_{\xi,k}(x_0)
(N_{c_{\xi,k}}\circ f_\rho)(x)
\\
&\qquad+
\widetilde\varepsilon_{\widetilde\tau_k(\eta)}(\eta)
e^{-k-\widetilde Y_k(\eta)}
\widetilde c'_{\eta,k}(y_0)
N_{\widetilde c_{\eta,k}}(y).
\end{aligned}
\]
The additive constant may be ignored when computing $Q_R$, since
translation of a measure on $\mathbb R$ leaves its $Q_R$-value
unchanged;  this is the case $a=1$ of
\eqref{eq:L1-scaling}.

Writing
$s=\varepsilon_{\tau_k(\omega)}(\omega),$
and using \eqref{eq:convolution-v-selfconf}, the last display is
\[
e^{-2k}
\left\langle
v_{\chi,s,k}(Y_k(\omega)),
\left(
(N_{c_{\xi,k}}\circ f_\rho)(x),
N_{\widetilde c_{\eta,k}}(y)
\right)
\right\rangle.
\]
Therefore, by the definition
\eqref{eq:convolution-selfconf-remainder} of $\Theta_{\chi,j}$,
the measure
\[
(\pi_{e_0})_\#
\left(
(L_{\omega,\xi,\rho,k}\times\widetilde L_{\eta,k})_\#
(\mu\times\nu)
\right)
\]
is, up to translation, the pushforward of $\Theta_{\chi,j}$ under the
linear map
\[
z\mapsto
e^{-2k}
\left\langle
v_{\chi,s,k}(Y_k(\omega)),z
\right\rangle.
\]
Thus, it is, up to translation, obtained by projecting
$\Theta_{\chi,j}$ in the direction
$[v_{\chi,s,k}(Y_k(\omega))]$
and then dilating the resulting one-dimensional measure by
$e^{-2k}
\left|v_{\chi,s,k}(Y_k(\omega))\right|.$

Therefore, by the  $L^1$ scaling identity
\eqref{eq:L1-scaling}, and recalling that translations do not affect
$Q_R$, we obtain
\[
\begin{aligned}
&
Q_R\left(
\left(
L_{\omega,\xi,\rho,k}
\times
\widetilde L_{\eta,k}
\right)_\#
(\mu\times\nu),
[e_0]
\right)
\\
&\qquad =
Q_{Re^{-2k}|v_{\chi,s,k}(Y_k(\omega))|}
\left(
\Theta_{\chi,j},
[v_{\chi,s,k}(Y_k(\omega))]
\right).
\end{aligned}
\]
Set
$U=Re^{-2k},$ and
recall from \eqref{eq:H-def} that
$H_{U,\theta}(\vartheta,z)
=
Q_{Ue^z}(\theta,\vartheta).$ 
Since
\[
\Phi_{\chi,s,k}(t)
=
\left(
[v_{\chi,s,k}(t)],
\log |v_{\chi,s,k}(t)|
\right),
\]
we have
$Ue^{\log|v_{\chi,s,k}(Y_k(\omega))|}
=
U|v_{\chi,s,k}(Y_k(\omega))|,$
and hence the preceding expression is exactly
$H_{U,\Theta_{\chi,j}}
\left(
\Phi_{\chi,s,k}(Y_k(\omega))
\right).$
Thus
\[
Q_R\left(
\left(
L_{\omega,\xi,\rho,k}
\times
\widetilde L_{\eta,k}
\right)_\#
(\mu\times\nu),
[e_0]
\right)
=
H_{U,\Theta_{\chi,j}}
\left(
\Phi_{\chi,s,k}(Y_k(\omega))
\right).
\]

Finally, decomposing according to
$\rho_{\omega,k}\in\mathcal R$ and
$s=\varepsilon_{\tau_k(\omega)}(\omega)\in\{\pm1\}$, we obtain
\[
\begin{aligned}
&\int
H_{U,\Theta_{\chi,(\rho_{\omega,k},
\varepsilon_{\tau_k(\omega)}(\omega))}}
\left(
\Phi_{\chi,\varepsilon_{\tau_k(\omega)}(\omega),k}
(Y_k(\omega))
\right)
\,dP_{k,\xi}(\omega)
\\
&\quad =
\sum_{(\rho,s)\in\mathcal J}
\int
\mathbf 1_{\{
\rho_{\omega,k}=\rho,\,
\varepsilon_{\tau_k(\omega)}(\omega)=s
\}}
H_{U,\Theta_{\chi,(\rho,s)}}
\left(
\Phi_{\chi,s,k}(Y_k(\omega))
\right)
\,dP_{k,\xi}(\omega)
\\
&\quad \leq
\sum_{(\rho,s)\in\mathcal J}
\int
H_{U,\Theta_{\chi,(\rho,s)}}
\left(
\Phi_{\chi,s,k}(Y_k(\omega))
\right)
\,dP_{k,\xi}(\omega)
\\
&\quad =
\sum_{j\in\mathcal J}
\int_{\Pone\times J}
H_{U,\Theta_{\chi,j}}(\vartheta,z)\,
dP_{\chi,j}(\vartheta,z),
\end{aligned}
\]
where in the last equality we used
$P_{\chi,j}
=
(\Phi_{\chi,s,k})_\#
\big((Y_k)_\#P_{k,\xi}\big),$ for $j=(\rho,s).$ 
Integrating this inequality with respect to
$dm(\chi)
=
d\mathbb P(\xi)\,d\mathbb Q(\eta)$ 
gives \eqref{eq:convolution-condition-B-selfconf}. Since
$\mathcal J$ is finite, the total linearization error is still
$O(Re^{-3k})$.
\end{proof}

\subsubsection{Ahlfors--David  second factor}
\label{subsec:convolution-AD}

Assume now that $\nu$ is $\alpha$-Ahlfors--David regular. The
construction in the first coordinate is unchanged.

We localize $\nu$ at scale $e^{-2k}$. Let $K=\supp\nu$, and choose a
maximal $e^{-2k}$-separated set
$\{x_{k,1},\ldots,x_{k,N_k}\}\subset K.$
Thus
$|x_{k,i}-x_{k,j}|\geq e^{-2k}$ for $i\neq j,$
while maximality implies that every point of $K$ lies within distance
$e^{-2k}$ of at least one of the points $x_{k,i}$.

Partition $K$ into measurable sets
$K=\bigsqcup_{i=1}^{N_k}V_{k,i}$
by setting
\[
V_{k,i}
=
\left\{
x\in K:
\begin{array}{l}
|x-x_{k,i}|<|x-x_{k,j}| \quad \text{for every } j<i,\\[2mm]
|x-x_{k,i}|\leq |x-x_{k,j}| \quad \text{for every } j>i
\end{array}
\right\}.
\]
Thus $V_{k,i}$ consists  of those points of $K$ for which
$x_{k,i}$ is a nearest point among
$x_{k,1},\ldots,x_{k,N_k}$, with ties resolved by choosing the
smallest index.
Since the points $x_{k,i}$ are $e^{-2k}$-separated,
$\left\{x\in K:|x-x_{k,i}|<\frac12e^{-2k}\right\}
\subset V_{k,i}.$
Moreover, maximality of the $e^{-2k}$-separated set implies that every
$x\in K$ satisfies
$$\min_{1\leq j\leq N_k}|x-x_{k,j}|<e^{-2k}, \text{ and hence }
V_{k,i}
\subset
\left\{x\in K:|x-x_{k,i}|<e^{-2k}\right\}.$$
Therefore
\begin{equation}\label{eq:convolution-AD-cells}
B\left(x_{k,i},\frac12e^{-2k}\right)\cap K
\subset
V_{k,i}
\subset
B\left(x_{k,i},e^{-2k}\right)\cap K.
\end{equation}
For $1\leq i\leq N_k$, set
$T_{k,i}(u)=x_{k,i}+e^{-2k}u$
and define the probability measure
\[
\vartheta_{k,i}
=
(T_{k,i}^{-1})_\#
\left(
\frac{\nu|_{V_{k,i}}}{\nu(V_{k,i})}
\right).
\]
Then,
$\supp\vartheta_{k,i}\subset[-1,1],$ and
\begin{equation}\label{eq:convolution-AD-decomposition}
\nu
=
\sum_{i=1}^{N_k}
\nu(V_{k,i})(T_{k,i})_\#\vartheta_{k,i}.
\end{equation}

We next specify the probability space and kernels from
Theorem~\ref{thm:stopped-intro}. For $\mathbb P$-almost every
$\xi\in\Omega$, let $c_{\xi,k}$ be the prefix supplied by
Proposition~\ref{prop:convolution-finite-remainder}. Thus
$|c'_{\xi,k}(x)|\asymp e^{-k},
\, x\in I.$
Let
\begin{equation}\label{eq:convolution-AD-conditioning-space}
(\Xi,m)
=
\left(
\Omega\times\{1,\ldots,N_k\},
\,
\mathbb P\times
\sum_{i=1}^{N_k}\nu(V_{k,i})\delta_i
\right),
\qquad
\chi=(\xi,i),
\end{equation}
and let
$\mathcal J=\mathcal R\times\{\pm1\},$ where $\mathcal{R}$ is from Proposition~\ref{prop:convolution-finite-remainder}.
Thus $\mathcal J$ is a fixed finite set, whereas the index
$i\in\{1,\ldots,N_k\}$ is part of the conditioning space.

For $s\in\{\pm1\}$ and $t\in[0,D']$, put
\begin{equation}\label{eq:convolution-v-AD}
v_{\chi,s,k}(t)
=
\left(
s e^k c'_{\xi,k}(x_0)e^{-t},
\,1
\right)
\in\mathbb R^2,
\end{equation}
and define
\[
\Phi_{\chi,s,k}(t)
=
\left(
[v_{\chi,s,k}(t)]_{\Pone},
\log|v_{\chi,s,k}(t)|
\right).
\]
By Proposition~\ref{prop:convolution-finite-remainder}, both
coordinates of $v_{\chi,s,k}(t)$ have absolute value bounded above
and below by positive constants, uniformly in $\chi$, $s$, $k$, and
$t\in[0,D']$. Hence the images of all the maps
$\Phi_{\chi,s,k}$ are contained in $\Pone\times J$ for one fixed
compact interval $J$.

Recall from \eqref{eq:convolution-normalized-word} that, for
$w\in\mathcal A^*$,
$N_w(x)=\frac{f_w(x)-f_w(x_0)}{f_w'(x_0)},$
and from \eqref{eq:convolution-first-stop} that
$Y_k(\omega)
=
S_{\tau_k(\omega)}(\omega)-k
\in[0,D'].$
Finally, for $\chi=(\xi,i)\in\Xi$ and
$j=(\rho,s)\in\mathcal J$, we define
\begin{equation}\label{eq:convolution-AD-kernel}
\Theta_{\chi,j}
=
(N_{c_{\xi,k}}\circ f_\rho)_\#\mu
\times
\vartheta_{k,i},
\qquad
P_{\chi,j}
=
(\Phi_{\chi,s,k})_\#
\big((Y_k)_\#P_{k,\xi}\big).
\end{equation}
The measures $\Theta_{\chi,j}$ are supported in one fixed compact
subset of $\mathbb R^2$, uniformly in all parameters.

Let us briefly explain the analogy with the self-conformal case.
For $P_{k,\xi}$-almost every $\omega$, one has
$\sigma^{\tau_k(\omega)}\omega
=
\sigma^{\tau_k(\xi)}\xi,$
so all such $\omega$ have the same  tail and hence the same
map $c_{\xi,k}$. As before, after fixing the sign $s$, the dependence
of the first  coefficient of \eqref{eq:convolution-v-AD} on $\omega$ is entirely through
$Y_k(\omega)$; after factoring out the common scale $e^{-2k}$, this
coefficient is
$s e^k c'_{\xi,k}(x_0)e^{-Y_k(\omega)}.$
The only difference from the self-conformal case is in the second
coordinate of \eqref{eq:convolution-v-AD}. Here the localization map $T_{k,i}$ is affine with
derivative exactly $e^{-2k}$, so after the same rescaling its
coefficient is simply $1$.

\begin{proposition}\label{prop:convolution-renormalization-AD}
There exist $C<\infty$ and $k_0\in\mathbb N$ such that, for every
$k\geq k_0$ and every $R\geq1$, setting
$U=Re^{-2k},$
we have
\begin{equation}\label{eq:convolution-condition-B-AD}
Q_R(\mu\times\nu,[e_0])
\leq
C\int_\Xi
\sum_{j\in\mathcal J}
\int_{\Pone\times J}
H_{U,\Theta_{\chi,j}}(\vartheta,z)
\,dP_{\chi,j}(\vartheta,z)\,dm(\chi)
+
CRe^{-3k}.
\end{equation}
\end{proposition}

\begin{proof}
We follow the proof of
Proposition~\ref{prop:convolution-renormalization-selfconf}, with the
 simplification that the second coordinate is already affine.

Since $\beta_k$ is a bona-fide stopping time,
\cite[Lemma~4.3]{algom2020decay} gives
$\mu
=
\int
(f_{\omega|\beta_k(\omega)})_\#\mu
\,d\mathbb P(\omega).$
Combining this with the exact decomposition
\eqref{eq:convolution-AD-decomposition} of $\nu$, we obtain
\[
\mu\times\nu
=
\int
\sum_{i=1}^{N_k}
\nu(V_{k,i})
\left(
f_{\omega|\beta_k(\omega)}
\times T_{k,i}
\right)_\#
(\mu\times\vartheta_{k,i})
\,d\mathbb P(\omega).
\]
Since projection and $\Delta_R$ are linear, Minkowski's inequality
therefore gives
\begin{equation}\label{eq:convolution-AD-stopping-bound}
\begin{aligned}
Q_R(\mu\times\nu,[e_0])
\leq
\int
\sum_{i=1}^{N_k}
\nu(V_{k,i})\,
Q_R\left(
\left(
f_{\omega|\beta_k(\omega)}
\times T_{k,i}
\right)_\#
(\mu\times\vartheta_{k,i}),
[e_0]
\right)
\,d\mathbb P(\omega).
\end{aligned}
\end{equation}

As in the proof of
Proposition~\ref{prop:convolution-renormalization-selfconf}, we now
disintegrate the $\omega$-integral with respect to $\mathcal P_k$.
Fix $\xi\in\Omega$ and write
$\zeta=\sigma^{\tau_k(\xi)}\xi.$
Then
$\sigma^{\tau_k(\omega)}\omega=\zeta$
for $P_{k,\xi}$-almost every $\omega$. By
Proposition~\ref{prop:convolution-finite-remainder}, for such $\omega$,
\begin{equation}\label{eq:convolution-AD-stopped-factorization}
f_{\omega|\beta_k(\omega)}
=
f_{\omega|\tau_k(\omega)}
\circ c_{\xi,k}
\circ f_{\rho_{\omega,k}},
\qquad
\rho_{\omega,k}\in\mathcal R.
\end{equation}

Fix also $i$ and $\rho\in\mathcal R$, and restrict temporarily to
those $\omega$ for which $\rho_{\omega,k}=\rho$. The first coordinate
is exactly the one treated in the proof of
Proposition~\ref{prop:convolution-renormalization-selfconf}. Namely,
writing
$F_\omega=f_{\omega|\tau_k(\omega)},$
and
\[
L_{\omega,\xi,\rho,k}(x)
=
F_\omega(x_\zeta)
+
F_\omega'(x_\zeta)
\left(
c_{\xi,k}(f_\rho(x))-x_\zeta
\right),
\]
equation~\eqref{eq:convolution-first-linearization} gives
\begin{equation}\label{eq:convolution-AD-first-linearization}
\sup_x
\left|
f_{\omega|\beta_k(\omega)}(x)
-
L_{\omega,\xi,\rho,k}(x)
\right|
\lesssim e^{-3k}.
\end{equation}
There is no corresponding approximation in the second coordinate:
by definition,
$T_{k,i}(y)=x_{k,i}+e^{-2k}y$
is already affine. Hence
\[
\sup_{x,y}
\left|
\pi_{e_0}\left(
f_{\omega|\beta_k(\omega)}(x),T_{k,i}(y)
\right)
-
\pi_{e_0}\left(
L_{\omega,\xi,\rho,k}(x),T_{k,i}(y)
\right)
\right|
\lesssim e^{-3k}.
\]
Lemma~\ref{lem:uniform-displacement} therefore yields
\begin{equation}\label{eq:convolution-AD-linearization-error}
\begin{aligned}
&
\left|
Q_R\left(
\left(
f_{\omega|\beta_k(\omega)}
\times T_{k,i}
\right)_\#
(\mu\times\vartheta_{k,i}),
[e_0]
\right)
\right.
\\
&\hspace{35mm}\left.
-
Q_R\left(
\left(
L_{\omega,\xi,\rho,k}
\times T_{k,i}
\right)_\#
(\mu\times\vartheta_{k,i}),
[e_0]
\right)
\right|
\lesssim
Re^{-3k}.
\end{aligned}
\end{equation}

We next identify the second term in
\eqref{eq:convolution-AD-linearization-error} with the quantity
appearing in \eqref{eq:convolution-condition-B-AD}. This is the same
calculation as in the proof of
Proposition~\ref{prop:convolution-renormalization-selfconf}, except
that the normalized second coordinate is simply
$\vartheta_{k,i}$.

Indeed, by \eqref{eq:convolution-normalized-word},
$c_{\xi,k}(f_\rho(x))
=
c_{\xi,k}(x_0)
+
c'_{\xi,k}(x_0)
(N_{c_{\xi,k}}\circ f_\rho)(x).$
Also, by \eqref{eq:convolution-first-stop},
$F_\omega'(x_\zeta)
=
\varepsilon_{\tau_k(\omega)}(\omega)
e^{-k-Y_k(\omega)}.$
Substituting these identities into
$L_{\omega,\xi,\rho,k}$, and using
$T_{k,i}(y)=x_{k,i}+e^{-2k}y,$
shows that, up to an additive constant, the projected map
\[
(x,y)
\longmapsto
\pi_{e_0}
\left(
L_{\omega,\xi,\rho,k}(x),T_{k,i}(y)
\right)
\]
is
\[
\begin{aligned}
(x,y)\longmapsto\;&
\varepsilon_{\tau_k(\omega)}(\omega)
e^{-k-Y_k(\omega)}
c'_{\xi,k}(x_0)
(N_{c_{\xi,k}}\circ f_\rho)(x)
+
e^{-2k}y
\\
=\;&
e^{-2k}
\left\langle
v_{\chi,s,k}(Y_k(\omega)),
\left(
(N_{c_{\xi,k}}\circ f_\rho)(x),y
\right)
\right\rangle,
\end{aligned}
\]
where
$\chi=(\xi,i),$
$s=\varepsilon_{\tau_k(\omega)}(\omega).$
Recall that
\[
\Theta_{\chi,(\rho,s)}
=
(N_{c_{\xi,k}}\circ f_\rho)_\#\mu
\times
\vartheta_{k,i}.
\]
So,
\[
(\pi_{e_0})_\#
\left(
(L_{\omega,\xi,\rho,k}\times T_{k,i})_\#
(\mu\times\vartheta_{k,i})
\right)
\]
is, up to translation, the pushforward of
$\Theta_{\chi,(\rho,s)}$ under the linear map
\[
z
\mapsto
e^{-2k}
\left\langle
v_{\chi,s,k}(Y_k(\omega)),z
\right\rangle.
\]
That is, one first projects
$\Theta_{\chi,(\rho,s)}$ in the direction
$[v_{\chi,s,k}(Y_k(\omega))]$
and then dilates the resulting one-dimensional measure by
$e^{-2k}|v_{\chi,s,k}(Y_k(\omega))|.$

Set $U=Re^{-2k}$. By the $L^1$ scaling identity
\eqref{eq:L1-scaling}, and since translations do not affect $Q_R$,
we therefore have
\[
\begin{aligned}
&
Q_R\left(
\left(
L_{\omega,\xi,\rho,k}
\times T_{k,i}
\right)_\#
(\mu\times\vartheta_{k,i}),
[e_0]
\right)
\\
&\qquad=
Q_{U|v_{\chi,s,k}(Y_k(\omega))|}
\left(
\Theta_{\chi,(\rho,s)},
[v_{\chi,s,k}(Y_k(\omega))]
\right)
\\
&\qquad=
H_{U,\Theta_{\chi,(\rho,s)}}
\left(
\Phi_{\chi,s,k}(Y_k(\omega))
\right).
\end{aligned}
\]
This is precisely the analogue of the corresponding identity in the
proof of Proposition~\ref{prop:convolution-renormalization-selfconf}.

It remains only to average. Decompose according to
$\rho_{\omega,k}\in\mathcal R,$
$s=\varepsilon_{\tau_k(\omega)}(\omega)\in\{\pm1\}.$
Exactly as at the end of the proof of
Proposition~\ref{prop:convolution-renormalization-selfconf}, positivity
allows us to discard the corresponding indicators, and hence
\[
\begin{aligned}
&
\int
H_{U,\Theta_{\chi,
(\rho_{\omega,k},\varepsilon_{\tau_k(\omega)}(\omega))}}
\left(
\Phi_{\chi,\varepsilon_{\tau_k(\omega)}(\omega),k}
(Y_k(\omega))
\right)
\,dP_{k,\xi}(\omega)
\\
&\qquad\leq
\sum_{j\in\mathcal J}
\int_{\Pone\times J}
H_{U,\Theta_{\chi,j}}(\vartheta,z)
\,dP_{\chi,j}(\vartheta,z).
\end{aligned}
\]
Finally, integrate this inequality with respect to
$dm(\chi)
=
d\mathbb P(\xi)
\sum_{i=1}^{N_k}\nu(V_{k,i})\,d\delta_i.$
Combining this with
\eqref{eq:convolution-AD-stopping-bound} and
\eqref{eq:convolution-AD-linearization-error} gives
\eqref{eq:convolution-condition-B-AD}. Since
$\mathcal J$ is fixed and finite, the accumulated linearization error
remains $O(Re^{-3k})$.
\end{proof}

\subsection{Uniform \(q\)-mass: condition \textup{(A)}}
\label{subsec:convolution-qmass}

We now verify condition~\textup{(A)} of
Theorem~\ref{thm:stopped-intro} for the families
$\{\Theta_{\chi,j}\}$ associated to the probability kernels constructed
in the preceding subsection. As there, we treat separately the cases in
which the second factor $\nu$ is self-conformal or Ahlfors--David regular.
\subsubsection{The self-conformal case}
Assume first that $\nu$ is self-conformal. Recall that in this case
\[
(\Xi,m)
=
(\Omega\times\widetilde\Omega,\mathbb P\times\mathbb Q),
\qquad
\chi=(\xi,\eta),
\]
and that, for $j=(\rho,s)\in\mathcal J$,
$\Theta_{\chi,j}
=
(N_{c_{\xi,k}}\circ f_\rho)_\#\mu
\times
(N_{\widetilde c_{\eta,k}})_\#\nu$,
recalling \eqref{eq:convolution-selfconf-kernel} and
\eqref{eq:convolution-selfconf-remainder}.

\begin{proposition}\label{prop:convolution-condition-A-selfconf}
There exist $1<q<2$, $S>1$, and $B_*<\infty$ such that, for every
$k\geq1$, every $\chi\in\Xi$, every $j\in\mathcal J$, and every
$m\geq1$,
\begin{equation}\label{eq:convolution-condition-A-selfconf}
S_{m,q}(\Theta_{\chi,j})
\leq
B_*\,2^{-m(q-1)S}.
\end{equation}
\end{proposition}
In particular, condition~\textup{(A)} of
Theorem~\ref{thm:stopped-intro} holds.

\begin{proof}
Since $\dim\mu+\dim\nu>1$, Theorem~\ref{thm:conformal-uniform-qmass}
allows us to choose $q\in(1,2)$ and numbers
\begin{equation}\label{eq:convolution-Snu-Smu}
S_\mu<D_q(\mu),
\qquad
S_\nu<D_q(\nu),
\qquad
S:=S_\mu+S_\nu>1,
\end{equation}
such that, for some $B<\infty$,
\begin{equation}\label{eq:convolution-qmass-factors}
S_{m,q}(\mu)
\leq
B\,2^{-m(q-1)S_\mu},
\qquad
S_{m,q}(\nu)
\leq
B\,2^{-m(q-1)S_\nu}
\end{equation}
for every $m\geq1$.

By Proposition~\ref{prop:convolution-finite-remainder} and the
discussion following \eqref{eq:convolution-normalized-word}, the maps
$N_{c_{\xi,k}}\circ f_\rho$ and
$N_{\widetilde c_{\eta,k}}$
form uniformly bi-Lipschitz families. Hence
Lemma~\ref{cor:bilip-qmass} and
\eqref{eq:convolution-qmass-factors} give
\[
S_{m,q}\big((N_{c_{\xi,k}}\circ f_\rho)_\#\mu\big)
\lesssim
2^{-m(q-1)S_\mu}, \text{ and } \,
S_{m,q}\big((N_{\widetilde c_{\eta,k}})_\#\nu\big)
\lesssim
2^{-m(q-1)S_\nu},
\]
uniformly in all parameters.
Finally, since the dyadic squares in $\mathbb R^2$ are products of
dyadic intervals,
\[
S_{m,q}(\theta_1\times\theta_2)
=
S_{m,q}(\theta_1)S_{m,q}(\theta_2)
\]
for probability measures $\theta_1,\theta_2$ on $\mathbb R$.
Applying this to $\Theta_{\chi,j}$ proves
\eqref{eq:convolution-condition-A-selfconf}.
\end{proof}

\subsubsection{The Ahlfors--David regular case}

Assume now that $\nu$ is $\alpha$-Ahlfors--David regular. Recall the
sets $V_{k,i}$ and the normalized probability measures
\[
\vartheta_{k,i}
=
(T_{k,i}^{-1})_\#
\left(
\frac{\nu|_{V_{k,i}}}{\nu(V_{k,i})}
\right)
\]
constructed in Subsection~\ref{subsec:convolution-AD}, where
$T_{k,i}(u)=x_{k,i}+e^{-2k}u.$

\begin{lemma}\label{lem:convolution-AD-localization}
There exists $C<\infty$ such that, for every sufficiently large $k$,
every $1\leq i\leq N_k$, and every $0<r\leq1$,
\begin{equation}\label{eq:convolution-AD-local-frostman}
\vartheta_{k,i}(B(u,r))
\leq
Cr^\alpha
\quad \text{ for }
u\in\mathbb R.
\end{equation}
Consequently, for every $1<q\leq2$, uniformly in $k$ and $i$,
\begin{equation}\label{eq:convolution-AD-local-qmass}
S_{m,q}(\vartheta_{k,i})
\leq
C\,2^{-m\alpha(q-1)}
\quad \text{ for }
m\geq1.
\end{equation}
\end{lemma}

\begin{proof}
By \eqref{eq:convolution-AD-cells} and the lower Ahlfors--David bound,
\[
\nu(V_{k,i})
\geq
\nu\left(
B\left(x_{k,i},\frac12e^{-2k}\right)
\right)
\gtrsim
e^{-2k\alpha}.
\]
Now fix $u\in\mathbb R$ and $0<r\leq1$. If
$V_{k,i}\cap T_{k,i}(B(u,r))=\varnothing,$
then \eqref{eq:convolution-AD-local-frostman} is immediate. Otherwise,
choose
$y\in V_{k,i}\cap T_{k,i}(B(u,r))\subset K.$
Since $T_{k,i}(B(u,r))$ is an interval of radius $e^{-2k}r$,
$T_{k,i}(B(u,r))
\subset
B(y,2e^{-2k}r).$
The upper Ahlfors--David bound therefore gives
\[
\begin{aligned}
\vartheta_{k,i}(B(u,r))
&=
\frac{
\nu\big(V_{k,i}\cap T_{k,i}(B(u,r))\big)
}{
\nu(V_{k,i})
}
\\
&\lesssim
\frac{(e^{-2k}r)^\alpha}{e^{-2k\alpha}}
\lesssim
r^\alpha,
\end{aligned}
\]
which proves \eqref{eq:convolution-AD-local-frostman}.

If $I\in\mathcal D_m$ and $\vartheta_{k,i}(I)>0$, then
\eqref{eq:convolution-AD-local-frostman} implies
$\vartheta_{k,i}(I)\lesssim 2^{-m\alpha}.$
Hence
\[
\begin{aligned}
S_{m,q}(\vartheta_{k,i})
&=
\sum_{I\in\mathcal D_m}\vartheta_{k,i}(I)^q
\\
&\leq
\left(
\max_{I\in\mathcal D_m}\vartheta_{k,i}(I)
\right)^{q-1}
\sum_{I\in\mathcal D_m}\vartheta_{k,i}(I)
\\
&\lesssim
2^{-m\alpha(q-1)},
\end{aligned}
\]
as required.
\end{proof}

Recall that in the present case
\[
\chi=(\xi,i)\in\Xi,
\qquad
j=(\rho,s)\in\mathcal J,
\]
and
\[
\Theta_{\chi,j}
=
(N_{c_{\xi,k}}\circ f_\rho)_\#\mu
\times
\vartheta_{k,i}
\]
by \eqref{eq:convolution-AD-kernel}.

\begin{proposition}\label{prop:convolution-condition-A-AD}
There exist $1<q<2$, $S>1$, and $B_*<\infty$ such that, for every
sufficiently large $k$, every $\chi\in\Xi$, every
$j\in\mathcal J$, and every $m\geq1$,
\begin{equation}\label{eq:convolution-condition-A-AD}
S_{m,q}(\Theta_{\chi,j})
\leq
B_*\,2^{-m(q-1)S}.
\end{equation}
\end{proposition}
Thus condition~\textup{(A)} of
Theorem~\ref{thm:stopped-intro} holds also in the
Ahlfors--David regular case.

\begin{proof}
Choose
$1-\alpha<d_0<\dim\mu.$
By Theorem~\ref{thm:conformal-uniform-qmass}, there exist
$q\in(1,2)$, a number
$d_0<S_\mu<D_q(\mu),$
and $B<\infty$ such that
\begin{equation}\label{eq:convolution-AD-qmass-mu}
S_{m,q}(\mu)
\leq
B\,2^{-m(q-1)S_\mu}
\qquad
(m\geq1).
\end{equation}
In particular,
$S:=S_\mu+\alpha>1.$

As in the self-conformal case, the maps
$N_{c_{\xi,k}}\circ f_\rho$ form a uniformly bi-Lipschitz family.
Lemma~\ref{cor:bilip-qmass} and
\eqref{eq:convolution-AD-qmass-mu} therefore give
\[
S_{m,q}\big(
(N_{c_{\xi,k}}\circ f_\rho)_\#\mu
\big)
\lesssim
2^{-m(q-1)S_\mu}
\]
uniformly in $\xi$, $\rho$, and $k$. On the other hand,
Lemma~\ref{lem:convolution-AD-localization} gives
\[
S_{m,q}(\vartheta_{k,i})
\lesssim
2^{-m\alpha(q-1)}
\]
uniformly in $k$ and $i$. Using again the exact product factorization
of dyadic $q$-mass, we obtain
\[
S_{m,q}(\Theta_{\chi,j})
\lesssim
2^{-m(q-1)(S_\mu+\alpha)}
=
2^{-m(q-1)S},
\]
which is \eqref{eq:convolution-condition-A-AD}.
\end{proof}
\subsection{The angular law: condition \textup{(C)}}
\label{subsec:convolution-angular}

It remains to verify condition~\textup{(C)} of
Theorem~\ref{thm:stopped-intro}. This is the only point in the proof
where we use the assumption that the IFS generating $\mu$ is not
$C^2$-conjugate to linear.

The  input we need in this case, stated below, was developed in our previous work
on Fourier decay of self-conformal measures. A non-effective version first
appeared in \cite{algom2020decay}, an effective version with
logarithmic rates was obtained in \cite{algom2021decay}, and the
exponential error term needed here was proved in
\cite{algom2023polynomial}. These arguments were inspired in part by
the renewal-theoretic methods introduced by Jialun Li in
\cite{Li2018decay,li2018fourier}. We record the precise consequence
that we shall use.

Recall that $\Omega=\mathcal A^{\mathbb N}$ is the coding space of the
IFS generating $\mu=\mu_{\mathbf{p}}$, endowed with the Bernoulli measure $\mathbb P=\mathbf{p}^\mathbb{N}$.
For $k\geq1$, the measurable partition $\mathcal P_k$ of $\Omega$ is
defined by declaring two sequences equivalent when they have the same
 tail after the stopping rule $\tau_k$. For
$\mathbb P$-almost every $\xi\in\Omega$, we write
$P_{k,\xi}:=\mathbb P_{\mathcal P_k(\xi)}$
for the conditional measure of $\mathbb P$ on the cell
$\mathcal P_k(\xi)$. Recall also from
\eqref{eq:convolution-first-stop} that
$Y_k(\omega)
=
S_{\tau_k(\omega)}(\omega)-k
\in[0,D'].$
\begin{proposition}
\label{prop:convolution-effective-renewal}
There exist $\epsilon_{\rm ren}>0$, $C<\infty$, and a probability
density $h\in L^\infty(\mathbb R)$ supported on $[0,D']$ such that,
for every sufficiently large $k$, for $\mathbb P$-almost every
$\xi\in\Omega$, and every $G\in C^3(\mathbb R)$,
\begin{equation}\label{eq:convolution-effective-renewal}
\left|
\int G(Y_k(\omega))\,dP_{k,\xi}(\omega)
-
\int_0^{D'}G(t)h(t)\,dt
\right|
\leq
Ce^{-\epsilon_{\rm ren}k}\|G\|_{C^3}.
\end{equation}
The constants are uniform in $k$ and $\xi$.
\end{proposition}

\begin{remark}\label{rem:convolution-effective-renewal}
Proposition~\ref{prop:convolution-effective-renewal} is a
specialization of \cite[Theorem~4.1]{algom2023polynomial} to the
derivative cocycle and the conditional measures $P_{k,\xi}$ used
here. We have only
repackaged its main term as the probability measure $h(t)\,dt$.
Indeed, the renewal main term in that theorem can be written, after
Fubini, as integration against a density of the form
$h(t)
=
\frac{1}{\chi}
\kappa_0([\max\{D,t\},D'])
\mathbf 1_{[0,D']}(t),$
where $\chi>0$ is the Lyapunov exponent and  $\kappa_0$ is the probability measure appearing in the 
main term of \cite[Theorem~4.1]{algom2023polynomial}. In particular, $h$ is bounded, and the
renewal identity gives $\int h=1$.
\end{remark}

We now use Proposition~\ref{prop:convolution-effective-renewal} to
construct the limiting laws $\Lambda_{\chi,j}$ in the two cases
considered above. In both cases their angular marginals will have
uniformly bounded densities on $\Pone$. Thus condition~\textup{(C)}
will hold with the optimal Frostman exponent
$\kappa=1.$
Together with the verification of conditions~\textup{(A)} and
\textup{(B)} in the preceding subsections, this will contain all the
substantive ingredients required to prove
Theorem~\ref{thm:convolution-intro}. In the following subsection we
will nevertheless collect the choices of scales and parameters and
check explicitly that they satisfy all the quantitative hypotheses of
Theorem~\ref{thm:stopped-intro}.

\subsubsection{The self-conformal second factor}

Assume first that $\nu$ is self-conformal. Recall from
Subsection~\ref{subsec:convolution-renormalization} that
\[
\chi=(\xi,\eta)\in\Xi,
\qquad
j=(\rho,s)\in\mathcal J,
\]
and that
$P_{\chi,j}
=
(\Phi_{\chi,s,k})_\#
\big((Y_k)_\#P_{k,\xi}\big).$
We define the corresponding limiting law by
\begin{equation}\label{eq:convolution-lambda-selfconf}
\Lambda_{\chi,j}
=
(\Phi_{\chi,s,k})_\#
\big(h(t)\,dt\big).
\end{equation}

\begin{proposition}\label{prop:convolution-condition-C-selfconf}
There exists $C<\infty$ such that, uniformly in $k$, $\chi\in\Xi$,
and $j\in\mathcal J$, the following hold.

For every nonnegative $G\in C^3(\Pone\times J)$,
\begin{equation}\label{eq:convolution-condition-C-renewal-selfconf}
\int G\,dP_{\chi,j}
\leq
C\int G\,d\Lambda_{\chi,j}
+
Ce^{-\epsilon_{\rm ren}k}\|G\|_{C^3}.
\end{equation}
Moreover,
\begin{equation}\label{eq:convolution-condition-C-frostman-selfconf}
(\operatorname{proj}_{\Pone})_\#
\Lambda_{\chi,j}(B(\theta,r))
\leq
Cr
\qquad
(\theta\in\Pone,\ 0<r\leq1).
\end{equation}
\end{proposition}
Therefore, condition~\textup{(C)} of
Theorem~\ref{thm:stopped-intro} holds in the self-conformal case with
\[
r=3,
\qquad
\kappa=1,
\qquad
\epsilon=\epsilon_{\rm ren}.
\]

\begin{proof}
By the definition of $P_{\chi,j}$ and
\eqref{eq:convolution-lambda-selfconf},
\[
\int G\,dP_{\chi,j}
=
\int
G\big(\Phi_{\chi,s,k}(Y_k(\omega))\big)
\,dP_{k,\xi}(\omega), \text{ whereas } \,\,\,
\int G\,d\Lambda_{\chi,j}
=
\int_0^{D'}
G\big(\Phi_{\chi,s,k}(t)\big)h(t)\,dt.
\]
The coefficients in \eqref{eq:convolution-v-selfconf} are bounded
above and below in absolute value by positive constants, uniformly in
all parameters. Hence
$\|\Phi_{\chi,s,k}\|_{C^3}\lesssim1.$
Applying Proposition~\ref{prop:convolution-effective-renewal} to
$G\circ\Phi_{\chi,s,k}$ and using the chain rule gives
\eqref{eq:convolution-condition-C-renewal-selfconf}.

It remains to estimate the angular marginal. Write
$v_{\chi,s,k}(t)
=
\big(ae^{-t},b\big),$
where $|a|\asymp1$ and $|b|\asymp1$, uniformly in all parameters.
If
$\vartheta(t)=[v_{\chi,s,k}(t)]_{\Pone},$
then, in a fixed projective chart,
\begin{equation}\label{eq:convolution-projective-speed-selfconf}
|\vartheta'(t)|
=
\frac{|ab|e^{-t}}
{a^2e^{-2t}+b^2}
\asymp1.
\end{equation}
Thus $t\mapsto\vartheta(t)$ is uniformly bi-Lipschitz onto its image.
Since $h\in L^\infty$, the pushforward of $h(t)\,dt$ under this map
has a uniformly bounded density with respect to Lebesgue measure on
$\Pone$. This proves
\eqref{eq:convolution-condition-C-frostman-selfconf}.
\end{proof}

\subsubsection{The Ahlfors--David regular second factor}

Assume now that $\nu$ is $\alpha$-Ahlfors--David regular. Recall that
in this case
\[
\chi=(\xi,i)\in\Xi,
\qquad
j=(\rho,s)\in\mathcal J,
\]
and
$P_{\chi,j}
=
(\Phi_{\chi,s,k})_\#
\big((Y_k)_\#P_{k,\xi}\big).$
Define
\begin{equation}\label{eq:convolution-lambda-AD}
\Lambda_{\chi,j}
=
(\Phi_{\chi,s,k})_\#
\big(h(t)\,dt\big).
\end{equation}

\begin{proposition}\label{prop:convolution-condition-C-AD}
There exists $C<\infty$ such that, uniformly in $k$, $\chi\in\Xi$,
and $j\in\mathcal J$, the following hold.

For every nonnegative $G\in C^3(\Pone\times J)$,
\begin{equation}\label{eq:convolution-condition-C-renewal-AD}
\int G\,dP_{\chi,j}
\leq
C\int G\,d\Lambda_{\chi,j}
+
Ce^{-\epsilon_{\rm ren}k}\|G\|_{C^3}.
\end{equation}
Moreover,
\begin{equation}\label{eq:convolution-condition-C-frostman-AD}
(\operatorname{proj}_{\Pone})_\#
\Lambda_{\chi,j}(B(\theta,r))
\leq
Cr
\qquad
(\theta\in\Pone,\ 0<r\leq1).
\end{equation}
\end{proposition}
Thus, condition~\textup{(C)} of
Theorem~\ref{thm:stopped-intro} holds in the Ahlfors--David regular
case with
\[
r=3,
\qquad
\kappa=1,
\qquad
\epsilon=\epsilon_{\rm ren}.
\]

\begin{proof}
The proof of the renewal estimate is identical to the
self-conformal case. The maps $\Phi_{\chi,s,k}$ have uniformly bounded
$C^3$ norms, so Proposition~\ref{prop:convolution-effective-renewal},
applied to $G\circ\Phi_{\chi,s,k}$, gives
\eqref{eq:convolution-condition-C-renewal-AD}.

For the angular marginal, recall from
\eqref{eq:convolution-v-AD} that
$v_{\chi,s,k}(t)
=
\left(
s e^k c'_{\xi,k}(x_0)e^{-t},
1
\right).$
Set
$a_{\xi,k}=e^k c'_{\xi,k}(x_0).$
By Proposition~\ref{prop:convolution-finite-remainder},
$|a_{\xi,k}|\asymp1$. Hence, writing
$\vartheta(t)=[v_{\chi,s,k}(t)]_{\Pone},$
we have
\begin{equation}\label{eq:convolution-projective-speed-AD}
|\vartheta'(t)|
=
\frac{|a_{\xi,k}|e^{-t}}
{1+|a_{\xi,k}|^2e^{-2t}}
\asymp1.
\end{equation}
Thus $t\mapsto\vartheta(t)$ is again uniformly bi-Lipschitz onto its
image. Since $h\in L^\infty$, the angular marginal of
$\Lambda_{\chi,j}$ has uniformly bounded density, which proves
\eqref{eq:convolution-condition-C-frostman-AD}.
\end{proof}
\subsection{Choice of parameters and conclusion of proof}
\label{subsec:convolution-conclusion}

The preceding subsections contain all the substantive ingredients of
the proof. For completeness, we now collect the parameters and check
explicitly that they fit the hypotheses of
Theorem~\ref{thm:stopped-intro}.

We apply that theorem to the planar measure
$\mu\times\nu$
and to the singleton set of directions
$E=\{[e_0]\},$ where
$e_0=2^{-1/2}(1,1).$
In the notation of Theorem~\ref{thm:stopped-intro}, the normalized
measures $\eta_{\xi,j}$ are the measures $\Theta_{\chi,j}$ constructed
above.

We first record the parameters which are common to the two cases.
Via Propositions~\ref{prop:convolution-condition-C-selfconf} and
\ref{prop:convolution-condition-C-AD}, we take
\begin{equation}\label{eq:convolution-common-C-parameters}
r=3,
\qquad
\kappa=1,
\qquad
\epsilon=\epsilon_{\rm ren}.
\end{equation}
Choose $\delta>0$ so small that
\begin{equation}\label{eq:convolution-parameter-choice}
0<\delta<1,
\qquad
5\delta<\epsilon_{\rm ren}.
\end{equation}
Since $5=2\cdot3-1$, this is precisely the condition
$(2r-1)\delta<\epsilon$,
required in \eqref{eq:smoothing-renewal-balance-intro}.

For sufficiently large $R$, set
\begin{equation}\label{eq:convolution-k-choice}
k
=
\left\lfloor
\frac{\log R}{2+\delta}
\right\rfloor
\end{equation}
and, as in \eqref{eq:convolution-U}, put
$U=Re^{-2k}.$
Then
$k\asymp\log R$
and
\begin{equation}\label{eq:convolution-U-choice}
U\asymp e^{\delta k}.
\end{equation}

The error term in both renormalization propositions satisfies
\begin{equation}\label{eq:convolution-error-choice}
Re^{-3k}
=
Ue^{-k}
\lesssim
e^{-(1-\delta)k}.
\end{equation}
Thus condition~\textup{(B)} has the form required in
Theorem~\ref{thm:stopped-intro}, with, for example,
$c=1-\delta>0.$
Moreover, $k\to\infty$ with $R$, so for all sufficiently large $R$
the lower bounds on $k$ required in the renormalization and
Ahlfors--David localization arguments are automatically satisfied.

The auxiliary set in both cases is
$\mathcal J=\mathcal R\times\{\pm1\}.$
After fixing an enumeration of $\mathcal J$, we therefore take
\[
N=|\mathcal J|=2|\mathcal R|,
\qquad
N_0=N.
\]
The measures $\Theta_{\chi,j}$ are supported in a fixed compact subset
of $\mathbb R^2$ in both constructions, so we choose
$M\in\mathbb N$ large enough that
$\supp\Theta_{\chi,j}\subset B(0,M)$
uniformly in all parameters. Likewise, the maps $\Phi_{\chi,s,k}$
constructed in the two cases take values in
$\Pone\times J$ for a fixed compact interval $J$; enlarging $J$ if
necessary, we use the same notation for such an interval below.
Finally, the constant $C$ in Theorem~\ref{thm:stopped-intro} is chosen
larger than the constants occurring in the relevant renormalization
and angular-law propositions.

We now check the two possible second factors separately.

\subsubsection*{The self-conformal second factor}

Suppose first that $\nu$ is self-conformal. We use the probability
space
\[
(\Xi,m)
=
(\Omega\times\widetilde\Omega,\mathbb P\times\mathbb Q)
\]
from \eqref{eq:convolution-selfconf-conditioning-space}, and the
families $\Theta_{\chi,j}$ and $P_{\chi,j}$ defined in
\eqref{eq:convolution-selfconf-kernel}, together with the limiting
laws $\Lambda_{\chi,j}$ from
\eqref{eq:convolution-lambda-selfconf}.

Proposition~\ref{prop:convolution-condition-A-selfconf} supplies
numbers
\[
1<q<2,
\qquad
S=S_\mu+S_\nu>1,
\]
and a constant $B_*<\infty$ such that
$S_{m,q}(\Theta_{\chi,j})
\leq
B_*2^{-m(q-1)S}$
uniformly in all parameters. Thus we take $B=B_*$, and
condition~\textup{(A)} of Theorem~\ref{thm:stopped-intro} holds.

Proposition~\ref{prop:convolution-renormalization-selfconf}, together
with \eqref{eq:convolution-U-choice} and
\eqref{eq:convolution-error-choice}, gives condition~\textup{(B)} with
the choices of $U$ and $c$ above.

Finally, Proposition~\ref{prop:convolution-condition-C-selfconf}
gives condition~\textup{(C)} with
\[
r=3,
\qquad
\epsilon=\epsilon_{\rm ren},
\qquad
\kappa=1.
\]
Since $S>1$,
$S+\kappa=S+1>2.$
All the hypotheses of Theorem~\ref{thm:stopped-intro} are therefore
satisfied.

\subsubsection*{The Ahlfors--David regular second factor}

Suppose instead that $\nu$ is $\alpha$-Ahlfors--David regular. We use
the probability space
\[
(\Xi,m)
=
\left(
\Omega\times\{1,\ldots,N_k\},
\,
\mathbb P\times
\sum_{i=1}^{N_k}\nu(V_{k,i})\delta_i
\right)
\]
from \eqref{eq:convolution-AD-conditioning-space}, together with the
families $\Theta_{\chi,j}$ and $P_{\chi,j}$ from
\eqref{eq:convolution-AD-kernel} and the limiting laws
$\Lambda_{\chi,j}$ from \eqref{eq:convolution-lambda-AD}.

Proposition~\ref{prop:convolution-condition-A-AD} supplies
\[
1<q<2,
\qquad
S=S_\mu+\alpha>1,
\]
and $B_*<\infty$ such that
$S_{m,q}(\Theta_{\chi,j})
\leq
B_*2^{-m(q-1)S}$
uniformly in the stopped data. Taking again $B=B_*$ gives
condition~\textup{(A)}.

Proposition~\ref{prop:convolution-renormalization-AD}, together with
\eqref{eq:convolution-U-choice} and
\eqref{eq:convolution-error-choice}, gives condition~\textup{(B)}
with exactly the same choices of $\delta$, $k$, $U$, and $c$ as
above.

Finally, Proposition~\ref{prop:convolution-condition-C-AD} gives
condition~\textup{(C)} with
\[
r=3,
\qquad
\epsilon=\epsilon_{\rm ren},
\qquad
\kappa=1.
\]
Again,
\[
S+\kappa=S+1>2,
\]
so all the hypotheses of Theorem~\ref{thm:stopped-intro} are
satisfied.

Thus, in either case, Theorem~\ref{thm:stopped-intro} yields constants
$C_*,\gamma>0$ such that
\[
Q_R(\mu\times\nu,[e_0])
\leq
C_*R^{-\gamma}
\qquad
(R\geq2).
\]
Recalling that
$\|\Delta_R(\mu*\nu)\|_1
=
Q_{\sqrt2 R}(\mu\times\nu,[e_0]),$
and absorbing the fixed factor $\sqrt2$, we obtain the required
Littlewood--Paley decay and hence
\[
\mu*\nu\ll\Leb^1.
\]
This proves Theorem~\ref{thm:convolution-intro}.

\section{Projections of self-conformal measures}
\label{sec:planar-conformal-proof}
We now prove Theorem~\ref{thm:planar-conformal-intro}. Recall that here
$D=B(0,1)\subset\C$, and that
$\Phi=\{f_i:i\in\mathcal A\}$
is a finite $C^\omega(\C)$ IFS consisting of injective 
contractions defined on a neighbourhood of $D$. Given a strictly
positive probability vector $p=(p_i)_{i\in\mathcal A}$, the associated
self-conformal measure is the unique probability measure $\mu$ satisfying
$\mu=\sum_{i\in\mathcal A}p_i(f_i)_\#\mu.$
Its support is the attractor $K_\Phi$.

Throughout this section we assume that $\Phi$ is not
$C^\omega$-conjugate to a self-similar IFS, that $K_\Phi$ is not
contained in a real-analytic planar curve, and that
$\dim\mu>1.$
Our goal is to prove
\[
(\pi_e)_\#\mu\ll\Leb^1
\quad
\text{ for every }e\in\Sone.
\]

As in the convolution argument in the previous Section, we verify the three hypotheses of
Theorem~\ref{thm:stopped-intro} by first constructing  normalized
pieces and  probability kernels which arise from a 
decomposition of $\mu$. The present situation here is somewhat simpler:
there is only one measure and only one derivative cocycle. On the
other hand, because the derivative is complex-valued, this cocycle
 contains an angular component. In fact, it is this angular component that will
supply the angular equidistribution required in condition~\textup{(C)}.

\subsection{Decomposition of $\mu$ and  normalized kernels}
\label{subsec:planar-stopping-inputs}
We retain the notations and conventions
$\Omega=\mathcal A^{\N},
\,
\mathbb P=p^{\N},$
 $x_\omega \in K_\Phi$ for the point coded by $\omega\in\Omega$, and
$f_w=f_{w_1}\circ\cdots\circ f_{w_n}$
for a finite word $w=w_1\cdots w_n$.

Put $\mathbb T=\R/(2\pi\Z)$. For $\omega\in\Omega$, define the
norm and angle components of the derivative cocycle by
\begin{equation}\label{eq:planar-derivative-cocycle}
S_n(\omega)
=
-\log\left|
f'_{\omega|n}(x_{\sigma^n\omega})
\right|,
\quad
A_n(\omega)
=
\arg f'_{\omega|n}(x_{\sigma^n\omega})
\in\mathbb T, \text{ for } \omega \in \Omega.
\end{equation}
These are the same definitions given in \cite{algom2024plane}. 
Uniform contraction gives
\[
0<D_0
\leq
S_{n+1}(\omega)-S_n(\omega)
\leq
D_0'<\infty, \text{ for every } \omega,\,n.
\]
For $k\geq1$, set, as in the convolutions proof,
\begin{equation}\label{eq:planar-first-stop}
\tau_k(\omega)
=
\min\{n:S_n(\omega)\geq k\},
\qquad
Y_k(\omega)
=
S_{\tau_k(\omega)}(\omega)-k
\in[0,D_0'].
\end{equation}
Thus
\begin{equation}\label{eq:planar-stopped-derivative}
f'_{\omega|\tau_k(\omega)}
\bigl(x_{\sigma^{\tau_k(\omega)}\omega}\bigr)
=
e^{-k-Y_k(\omega)+iA_{\tau_k(\omega)}(\omega)}.
\end{equation}
As in Subsection~\ref{subsec:convolution-stopping}, define a measurable
partition $\mathcal P_k$ of $\Omega$ by
\[
\omega\sim_k\omega'
\quad\Longleftrightarrow\quad
\sigma^{\tau_k(\omega)}\omega
=
\sigma^{\tau_k(\omega')}\omega'.
\]
Thus the elements of $\mathcal P_k$ consist of codes having the same
 tail after the stopping rule $\tau_k$. For
$\mathbb P$-almost every $\xi\in\Omega$, let
$P_{k,\xi}
=
\mathbb P_{\mathcal P_k(\xi)}$
denote the conditional measure of $\mathbb P$ on the cell
$\mathcal P_k(\xi)$. Accordingly,
\begin{equation}\label{eq:planar-tail-disintegration}
\int F\,d\mathbb P
=
\int_\Omega
\left(
\int F\,dP_{k,\xi}
\right)
d\mathbb P(\xi),\, \text{ for every integrable } F.
\end{equation}

As in the convolution argument, $\tau_k$ need not be a
stopping time. We therefore again compare it with a genuine prefix
stopping time at a slightly deeper geometric scale. Fix a base point
$z_0\in D$ and define
\begin{equation}\label{eq:planar-deeper-stop}
\beta_k(\omega)
=
\min
\left\{
m:
|f'_{\omega|m}(z_0)|<e^{-2k}
\right\}.
\end{equation}
By bounded distortion and \eqref{eq:planar-first-stop}, for all
sufficiently large $k$,
$\beta_k(\omega)>\tau_k(\omega)$,
uniformly in $\omega$.
The portion of the word between these two stopping rules admits the
same finite-remainder decomposition as in the one-dimensional
argument.

\begin{lemma}[Finite remainder decomposition]
\label{lem:planar-finite-remainder}
There exists a finite set of words $\mathcal R$, independent of $k$,
such that, for $\mathbb P$-almost every $\xi\in\Omega$, there is a
word-map $g_{\xi,k}$ with the following property. For
$P_{k,\xi}$-almost every $\omega$,
\begin{equation}\label{eq:planar-finite-remainder}
f_{\sigma^{\tau_k(\omega)}\omega|
\,\beta_k(\omega)-\tau_k(\omega)}
=
g_{\xi,k}\circ f_{\rho_{\omega,k}}
\end{equation}
for some $\rho_{\omega,k}\in\mathcal R$. Moreover,
\begin{equation}\label{eq:planar-common-prefix-scale}
|g'_{\xi,k}(z)|
\asymp
e^{-k},
\qquad z\in D,
\end{equation}
uniformly in $\xi$ and $k$. In particular,
$\diam g_{\xi,k}(D)\lesssim e^{-k}.$
\end{lemma}
Both the statement and proof are quite similar to Proposition \ref{prop:convolution-finite-remainder}.
\begin{proof}
For $P_{k,\xi}$-almost every $\omega$, the  tail
$\sigma^{\tau_k(\omega)}\omega$ is the same. By minimality in
\eqref{eq:planar-deeper-stop} and bounded distortion \eqref{eq:planar-distortion-section4}, the norm of the derivative of the full cylinder
$f_{\omega|\beta_k(\omega)}$ is comparable to $e^{-2k}$ at all points, whereas the term in
\eqref{eq:planar-stopped-derivative} has modulus comparable to
$e^{-k}$ uniformly at all points. Bounded distortion and the chain rule therefore show that the norm of the derivative of
$f_{\sigma^{\tau_k(\omega)}\omega|
\,\beta_k(\omega)-\tau_k(\omega)}$ is comparable to $e^{-k}$ at all points.

Let $\zeta$ denote the common terminal tail $\sigma^{\tau_k(\omega)}\omega$ on the fibre
$\mathcal P_k(\xi)$, and write
$n(\omega)=\beta_k(\omega)-\tau_k(\omega).$
By the preceding paragraph, there is a constant $C\geq1$, independent
of $k$, $\xi$, and $\omega$, such that
\begin{equation}\label{eq:planar-suffix-scale}
C^{-1}e^{-k}
\leq
\left|f'_{\zeta|n(\omega)}(z)\right|
\leq
Ce^{-k}
\qquad
(z\in D).
\end{equation}
Choose a reference integer $\ell=\ell(\xi,k)$ to be the least integer
such that
$\left|f'_{\zeta|\ell}(z_0)\right|\leq e^{-k}.$
By the uniform contraction bounds and
\eqref{eq:planar-distortion-section4}, there is an integer $P\geq1$,
depending only on the IFS, such that
\begin{equation}\label{eq:planar-suffix-length-comparison}
|n(\omega)-\ell|\leq P
\end{equation}
for every $\omega$ under consideration. Indeed, adding or deleting
$j$ symbols changes the norm of the derivative by an exponential
factor bounded above and below in terms only of $j$ and the uniform
one-step contraction constants. Hence
\eqref{eq:planar-suffix-scale}, together with the corresponding
estimate for the reference prefix $\zeta|\ell$, rules out
$|n(\omega)-\ell|>P$ once $P$ is chosen sufficiently large.

Set
$m=\max\{0,\ell-P\}$ and
$g_{\xi,k}=f_{\zeta|m}.$
Since every $n(\omega)$ lies between $\ell-P$ and $\ell+P$, each
prefix $\zeta|n(\omega)$ has the exact factorization
\[
f_{\zeta|n(\omega)}
=
g_{\xi,k}\circ f_{\rho_{\omega,k}},
\]
where $\rho_{\omega,k}$ is a word of length at most $2P$. Thus we may
take
$\mathcal R=\bigcup_{j=0}^{2P}\mathcal A^j,$
which is independent of $k$ and $\xi$. Finally,
\eqref{eq:planar-suffix-scale}, the fact that $m$ differs from
$n(\omega)$ by at most $2P$, and bounded distortion imply
$|g'_{\xi,k}(z)|\asymp e^{-k}$ for all
$z\in D$,
which is \eqref{eq:planar-common-prefix-scale}. The diameter estimate
then follows from the quasi-similarity estimate
\eqref{eq:common-quasisim-section4}.
\end{proof}

We retain the notations from Lemma \ref{lem:planar-finite-remainder}. We  normalize these common prefixes  as in
\eqref{eq:convolution-normalized-word}. Namely, set
$N_{\xi,k}(z)
=
\frac{g_{\xi,k}(z)-g_{\xi,k}(z_0)}
{g'_{\xi,k}(z_0)}.$
For $\rho\in\mathcal R$, define, as in Section~\ref{sec:conformal-qmass},
\begin{equation}\label{eq:planar-remainder-measure}
\Theta_{\xi,\rho}
=
(N_{\xi,k}\circ f_\rho)_\#\mu.
\end{equation}
Since $\mathcal R$ is finite,
the maps $N_{\xi,k}\circ f_\rho$ form a uniformly bi-Lipschitz
family and their images lie in one fixed ball. In particular, there
is $M<\infty$ such that
\begin{equation}\label{eq:planar-remainder-support}
\supp\Theta_{\xi,\rho}
\subset B(0,M)
\end{equation}
uniformly in $\xi$, $\rho$, and $k$.

We next specify the probability space and the kernels appearing in
Theorem~\ref{thm:stopped-intro}. Set
\begin{equation}\label{eq:planar-conditioning-space}
(\Xi,m)=(\Omega,\mathbb P),
\qquad
\mathcal J=\mathcal R.
\end{equation}
Thus the conditioning variable is simply $\xi\in\Omega$, whereas the
finite auxiliary index is $\rho\in\mathcal R$.

Fix  a direction $e\in\Sone$, and define
\begin{equation}\label{eq:planar-actual-law-map}
\Phi^e_{\xi,k}(a,y)
=
\left(
\left[
R_{-\left(a+\arg g'_{\xi,k}(z_0)\right)}e
\right],
\,
\log\left(e^k|g'_{\xi,k}(z_0)|\right)-y
\right),
\qquad
(a,y)\in\mathbb T\times[0,D_0'].
\end{equation}
By \eqref{eq:planar-common-prefix-scale}, the second coordinate of
$\Phi^e_{\xi,k}$ ranges in one fixed compact interval, uniformly in
$e$, $\xi$, and $k$.
For $\rho\in\mathcal R$, set
\begin{equation}\label{eq:planar-kernels}
P^e_{\xi,\rho}
=
(\Phi^e_{\xi,k})_\#
\left(
(A_{\tau_k},Y_k)_\#P_{k,\xi}
\right).
\end{equation}
The right-hand side is independent of $\rho$; we retain the index
$\rho$ only to match the family of normalized measures
$\Theta_{\xi,\rho}$ in Theorem~\ref{thm:stopped-intro}. The images of
the maps \eqref{eq:planar-actual-law-map} lie in
$\Pone\times J_1$ for one fixed compact interval $J_1$, uniformly in
$e$, $\xi$, and $k$.

The structure is  analogous to the convolution
construction. The principal difference is that there is no second
factor and hence no second stopping rule: the direction variable is
already carried by the argument $A_{\tau_k}$ of the complex
derivative cocycle. The overshoot $Y_k$ controls the remaining radial
scale, while the finite index $\rho$ records the bounded discrepancy
between the cocycle stopping rule $\tau_k$ and the genuine stopping
time $\beta_k$.

\subsection{Renormalization: condition \textup{(B)}}
\label{subsec:planar-renormalization}

We first verify the renormalization condition. This is the planar
analogue of Propositions~\ref{prop:convolution-renormalization-selfconf}
and \ref{prop:convolution-renormalization-AD}.

\begin{proposition}\label{prop:planar-renormalization}
There exist $C<\infty$ and $k_0\in\N$ such that, for every
$k\geq k_0$, every $R\geq1$, and every $e\in\Sone$, setting
\begin{equation}\label{eq:planar-U}
U=Re^{-2k},
\end{equation}
one has, uniformly in $e$,
\begin{equation}\label{eq:planar-condition-B}
\begin{aligned}
Q_R(\mu,[e])
\leq\;&
C
\int_\Omega
\sum_{\rho\in\mathcal R}
\int_{\Pone\times J_1}
H_{U,\Theta_{\xi,\rho}}(\theta,z)
\,dP^e_{\xi,\rho}(\theta,z)
\,d\mathbb P(\xi)
\\
&\qquad
+
CRe^{-3k}.
\end{aligned}
\end{equation}\end{proposition}

\begin{proof}
Since $\beta_k$ is a genuine  stopping time, iteration of the
stationary identity gives
$\mu
=
\int
(f_{\omega|\beta_k(\omega)})_\#\mu
\,d\mathbb P(\omega).$
Hence, by linearity and Minkowski's inequality,
\begin{equation}\label{eq:planar-minkowski-stop}
Q_R(\mu,[e])
\leq
\int
Q_R\left(
(f_{\omega|\beta_k(\omega)})_\#\mu,[e]
\right)
\,d\mathbb P(\omega).
\end{equation}
We disintegrate the right-hand side according to
$\mathcal P_k$ using \eqref{eq:planar-tail-disintegration}. Fix
$\xi\in\Omega$. For $P_{k,\xi}$-almost every $\omega$, the tail
$\sigma^{\tau_k(\omega)}\omega$ is the same; let $x_*$ denote its
coded point, and write
$F_\omega=f_{\omega|\tau_k(\omega)}.$
By Lemma~\ref{lem:planar-finite-remainder},
\[
f_{\omega|\beta_k(\omega)}
=
F_\omega\circ g_{\xi,k}\circ f_{\rho_{\omega,k}},
\quad
\rho_{\omega,k}\in\mathcal R.
\]

Fix $\xi\in\Omega$ and, for the moment, fix $\rho\in\mathcal R$.
We restrict to those $\omega$ for which $\rho_{\omega,k}=\rho$.
At the end of the argument we sum over $\rho\in\mathcal R$; since
$\mathcal R$ is fixed and finite, its cardinality is absorbed into
the constants.

For such $\omega$,
$f_{\omega|\beta_k(\omega)}
=
F_\omega\circ g_{\xi,k}\circ f_\rho.$
Both $x_*$ and $g_{\xi,k}\circ f_\rho(D)$ are contained in
$g_{\xi,k}(D)$, whose diameter is $O(e^{-k})$ by
Lemma~\ref{lem:planar-finite-remainder}. By the complex-analytic analogue of
Lemma~\ref{lem:convolution-relative-Taylor}, whose proof is essentially the same,
\begin{equation}\label{eq:planar-relative-Taylor}
|F_\omega(x)-F_\omega(y)-F_\omega'(y)(x-y)|
\leq
C|F_\omega'(y)|\,|x-y|^2.
\end{equation}
Taking $y=x_*$ and using
\eqref{eq:planar-stopped-derivative}, we have
$|F_\omega'(x_*)|
=
e^{-k-Y_k(\omega)}
\asymp e^{-k}.$
Hence, on $g_{\xi,k}\circ f_\rho(D)$, the map $F_\omega$ differs
from its affine approximation at $x_*$ by $O(e^{-3k})$. 
Define the affine map
\[
L_{\omega,\xi,k}(x)
=
F_\omega(x_*)
+
F_\omega'(x_*)(x-x_*).
\]
By the preceding estimate, for $x\in g_{\xi,k}\circ f_\rho(D)$,
\begin{equation*}
\left|F_\omega(x)-L_{\omega,\xi,k}(x)\right|
\lesssim e^{-3k}.
\end{equation*}
Consequently, since orthogonal projection is $1$-Lipschitz,
Lemma~\ref{lem:uniform-displacement} gives
\begin{equation}\label{eq:planar-linearization-error}
\begin{aligned}
\Big|
&
Q_R\left(
(F_\omega\circ g_{\xi,k}\circ f_\rho)_\#\mu,[e]
\right)
\\
&\quad-
Q_R\left(
(L_{\omega,\xi,k}\circ g_{\xi,k}\circ f_\rho)_\#\mu,[e]
\right)
\Big|
\lesssim
Re^{-3k}.
\end{aligned}
\end{equation}

We now compute the contribution of the map $L_{\omega,\xi,k}$.
By the definition of $N_{\xi,k}$,
$g_{\xi,k}(f_\rho(z))
=
g_{\xi,k}(z_0)
+
g'_{\xi,k}(z_0)
(N_{\xi,k}\circ f_\rho)(z).$
So,
\[
L_{\omega,\xi,k}(g_{\xi,k}(f_\rho(z)))
=
F_\omega(x_*)
+
F_\omega'(x_*)
\bigl(g_{\xi,k}(z_0)-x_*\bigr)
+
F_\omega'(x_*)g'_{\xi,k}(z_0)
(N_{\xi,k}\circ f_\rho)(z).
\]
Thus, up to translation, this map is multiplication of
$N_{\xi,k}\circ f_\rho$ by
$$F_\omega'(x_*)g'_{\xi,k}(z_0)
=
e^{-k-Y_k(\omega)+iA_{\tau_k(\omega)}(\omega)}
g'_{\xi,k}(z_0).$$

For complex multiplication by $ce^{i\varphi}$, $c>0$, we have
$\pi_e(ce^{i\varphi}z)
=
c\,\pi_{R_{-\varphi}e}(z).$
Since translations do not affect the $L^1$ norm of a
Littlewood--Paley piece, Lemma~\ref{lem:L1-scaling} gives, with
$U=Re^{-2k},$
\begin{align*}
&
Q_R\left(
(L_{\omega,\xi,k}\circ g_{\xi,k}\circ f_\rho)_\#\mu,
[e]
\right)
\\
&\qquad =
H_{U,\Theta_{\xi,\rho}}
\left(
\left[
R_{-\left(
A_{\tau_k(\omega)}(\omega)
+
\arg g'_{\xi,k}(z_0)
\right)}e
\right],
\,
\log\left(e^k|g'_{\xi,k}(z_0)|\right)-Y_k(\omega)
\right).
\end{align*}
By the definition of $\Phi^e_{\xi,k}$, the right-hand side is
$H_{U,\Theta_{\xi,\rho}}
\left(
\Phi^e_{\xi,k}
\bigl(A_{\tau_k(\omega)}(\omega),Y_k(\omega)\bigr)
\right).$

Recall that we are presently considering only those $\omega$ for
which $\rho_{\omega,k}=\rho$. Combining the preceding identity with
\eqref{eq:planar-linearization-error}, we obtain
\begin{align*}
&
\int_{\{\rho_{\omega,k}=\rho\}}
Q_R\left(
(f_{\omega|\beta_k(\omega)})_\#\mu,[e]
\right)
\,dP_{k,\xi}(\omega)
\\
&\qquad\leq
\int_{\{\rho_{\omega,k}=\rho\}}
H_{U,\Theta_{\xi,\rho}}
\left(
\Phi^e_{\xi,k}
\bigl(A_{\tau_k(\omega)}(\omega),Y_k(\omega)\bigr)
\right)
\,dP_{k,\xi}(\omega)
\\
&\qquad\qquad
+
CRe^{-3k}
P_{k,\xi}\{\rho_{\omega,k}=\rho\}.
\end{align*}
Since the integrand is nonnegative, we may enlarge the first
integral to the whole fibre. By the definition of
$P^e_{\xi,\rho}$,
\[
\int
H_{U,\Theta_{\xi,\rho}}
\left(
\Phi^e_{\xi,k}
\bigl(A_{\tau_k(\omega)}(\omega),Y_k(\omega)\bigr)
\right)
\,dP_{k,\xi}(\omega)
=
\int_{\Pone\times J_1}
H_{U,\Theta_{\xi,\rho}}(\theta,z)
\,dP^e_{\xi,\rho}(\theta,z).
\]

We now sum over $\rho\in\mathcal R$. Since
$\rho_{\omega,k}\in\mathcal R$ and the sets
$\{\rho_{\omega,k}=\rho\}$ partition the fibre,
$\sum_{\rho\in\mathcal R}
P_{k,\xi}\{\rho_{\omega,k}=\rho\}
=
1.$
Therefore,
\[
\int
Q_R\left(
(f_{\omega|\beta_k(\omega)})_\#\mu,[e]
\right)
\,dP_{k,\xi}(\omega)
\leq
\sum_{\rho\in\mathcal R}
\int_{\Pone\times J_1}
H_{U,\Theta_{\xi,\rho}}(\theta,z)
\,dP^e_{\xi,\rho}(\theta,z)
+
CRe^{-3k}.
\]
Finally, integrating in $\xi$ and using
\eqref{eq:planar-minkowski-stop}, we obtain
\[
Q_R(\mu,[e])
\leq
\int_\Omega
\sum_{\rho\in\mathcal R}
\int_{\Pone\times J_1}
H_{U,\Theta_{\xi,\rho}}(\theta,z)
\,dP^e_{\xi,\rho}(\theta,z)
\,d\mathbb P(\xi)
+
CRe^{-3k},
\]
which is \eqref{eq:planar-condition-B}.
\end{proof}

\subsection{Uniform \(q\)-mass: condition \textup{(A)}}
\label{subsec:planar-qmass}

We next verify condition~\textup{(A)} for the normalized measures
$\Theta_{\xi,\rho}$ constructed above. It is here that the hypothesis
$\dim\mu>1$ enters the proof.

\begin{proposition}\label{prop:planar-condition-A}
There exist $1<q<2$, $S>1$, and $B_*<\infty$ such that, for every
sufficiently large $k$, for $\mathbb P$-almost every $\xi\in\Omega$,
every $\rho\in\mathcal R$, and every $m\geq1$,
\begin{equation}\label{eq:planar-condition-A}
S_{m,q}(\Theta_{\xi,\rho})
\leq
B_*\,2^{-m(q-1)S}.
\end{equation}
\end{proposition}
Thus, condition~\textup{(A)} of
Theorem~\ref{thm:stopped-intro} holds.

\begin{proof}
Since $\dim\mu>1$, Theorem~\ref{thm:conformal-uniform-qmass} gives
$1<q<2$, a number
$1<S<D_q(\mu),$
and $B<\infty$ such that
\begin{equation}\label{eq:planar-base-qmass}
S_{m,q}(\mu)
\leq
B\,2^{-m(q-1)S}
\qquad
(m\geq1).
\end{equation}

By \eqref{eq:planar-common-prefix-scale}, the normalized-cylinder
geometry of Section~\ref{sec:conformal-qmass}, and the finiteness of
$\mathcal R$, the maps
$N_{\xi,k}\circ f_\rho$
form a uniformly bi-Lipschitz family. Hence
Lemma~\ref{cor:bilip-qmass}, applied to
\eqref{eq:planar-base-qmass}, gives
\[
S_{m,q}
\left(
(N_{\xi,k}\circ f_\rho)_\#\mu
\right)
\lesssim
2^{-m(q-1)S}
\]
uniformly in $\xi$, $\rho$, and $k$. By
\eqref{eq:planar-remainder-measure}, this is exactly
\eqref{eq:planar-condition-A}.
\end{proof}

\subsection{The angular law: condition \textup{(C)}}
\label{subsec:planar-angular}

It remains to verify condition~\textup{(C)}, that is, to control the
angular law. There is a conceptual difference here from the
convolution setting. In the previous section the angular variable
arose after representing convolution as a projection of a product,
and the renewal law of one of the factors was used to produce the
required angular distribution. Here the angular variable is already
part of the complex derivative cocycle: the angle component
$A_{\tau_k}$ directly rotates the projection direction. Thus the
required smooth angular law follows from the joint norm--angle
renewal theorem for this cocycle.

The input we use was proved in our previous work \cite[Theorem~3.1]{algom2024plane}. It is the
planar extension of the one-dimensional renewal theorem, which is Proposition \ref{prop:convolution-effective-renewal} here. 

\begin{thm}
\label{thm:planar-angular-renewal}
There exist $\epsilon_{\rm ang}>0$, $C_{\rm ang}<\infty$, a fixed
compact interval $J_0\subset\R$, and the probability measure
$\lambda_0
=
m_{\mathbb T}
\times
\frac{1}{|J_0|}\left.\Leb^1\right|_{J_0}$
on $\mathbb T\times J_0$ such that:

For every sufficiently large $k$,
for $\mathbb P$-almost every $\xi\in\Omega$, and every nonnegative
$G\in C^8(\mathbb T\times J_0)$,
\begin{equation}\label{eq:planar-angular-renewal}
\int
G\bigl(A_{\tau_k(\omega)}(\omega),Y_k(\omega)\bigr)
\,dP_{k,\xi}(\omega)
\leq
C_{\rm ang}\int G\,d\lambda_0
+
C_{\rm ang}e^{-\epsilon_{\rm ang}k}
\|G\|_{C^8}.
\end{equation}
The constants are uniform in $k$ and $\xi$.
\end{thm}

Theorem~\ref{thm:planar-angular-renewal} is a direct consequence of
\cite[Theorem~3.1]{algom2024plane}, applied to the norm and angle
components of the derivative cocycle. The renewal main term in that
theorem is Haar measure in the angular variable, while the overshoot
variable is supported in a fixed compact interval. Since we only use
nonnegative test functions, we may enlarge this interval to $J_0$ and
dominate the overshoot part by Lebesgue measure there, absorbing the
resulting constant into $C_{\rm ang}$. Haar invariance also absorbs
the bounded angular translation appearing in the renewal main term.
Finally, the estimate in \cite[Theorem~3.1]{algom2024plane} is uniform
in the prescribed  tail, and therefore applies uniformly to
the conditional measures $P_{k,\xi}$.

Let $J$ be a fixed compact interval containing the interval $J_1$
defined after \eqref{eq:planar-kernels} and all numbers
$\log\left(e^k|g'_{\xi,k}(z_0)|\right)-y,
$ for $y\in J_0.$
Such a $J$ exists by \eqref{eq:planar-common-prefix-scale}.

For $e\in\Sone$, $\xi\in\Omega$, and $\rho\in\mathcal R$, define
\begin{equation}\label{eq:planar-limiting-law}
\Lambda^e_{\xi,\rho}
=
(\Phi^e_{\xi,k})_\#\lambda_0,
\end{equation}
where
$\Phi^e_{\xi,k}(a,y)
=
\left(
\left[
R_{-\left(a+\arg g'_{\xi,k}(z_0)\right)}e
\right],
\,
\log\left(e^k|g'_{\xi,k}(z_0)|\right)-y
\right).$
Thus $\Lambda^e_{\xi,\rho}$ is supported on
$\Pone\times J$. As with $P^e_{\xi,\rho}$, it does not actually
depend on $\rho$; we retain the index only to keep the notation
parallel to the family $\Theta_{\xi,\rho}$.

\begin{proposition}\label{prop:planar-condition-C}
There exists $C<\infty$ such that, uniformly in
$e\in\Sone$, $k$, for $\mathbb P$-almost every $\xi\in\Omega$, and
$\rho\in\mathcal R$, the following hold.

For every nonnegative $G\in C^8(\Pone\times J)$,
\begin{equation}\label{eq:planar-condition-C-renewal}
\int G\,dP^e_{\xi,\rho}
\leq
C\int G\,d\Lambda^e_{\xi,\rho}
+
Ce^{-\epsilon_{\rm ang}k}
\|G\|_{C^8}.
\end{equation}
Moreover,
\begin{equation}\label{eq:planar-condition-C-frostman}
(\operatorname{proj}_{\Pone})_\#
\Lambda^e_{\xi,\rho}(B(\theta,r))
\leq
Cr
\quad
\theta\in\Pone,\ 0<r\leq1.
\end{equation}
\end{proposition}
Therefore, condition~\textup{(C)} of
Theorem~\ref{thm:stopped-intro} holds with
\[
r=8,
\quad
\kappa=1,
\quad
\epsilon=\epsilon_{\rm ang}.
\]
\begin{proof}
By definition,
\[
P^e_{\xi,\rho}
=
(\Phi^e_{\xi,k})_\#
\left(
(A_{\tau_k},Y_k)_\#P_{k,\xi}
\right),
\qquad
\Lambda^e_{\xi,\rho}
=
(\Phi^e_{\xi,k})_\#\lambda_0,
\]
where
$\Phi^e_{\xi,k}(a,y)
=
\left(
\left[
R_{-\left(a+\arg g'_{\xi,k}(z_0)\right)}e
\right],
\,
\log\left(e^k|g'_{\xi,k}(z_0)|\right)-y
\right).$ 
The maps $\Phi^e_{\xi,k}$ have uniformly bounded derivatives through
order $8$ in smooth coordinates, uniformly in $e$, $\xi$, and $k$.
Therefore, for every nonnegative
$G\in C^8(\Pone\times J)$, applying
Theorem~\ref{thm:planar-angular-renewal} to
$G\circ\Phi^e_{\xi,k}$ and using the chain rule gives
\begin{equation*}
\int G\,dP^e_{\xi,\rho}
\leq
C\int G\,d\Lambda^e_{\xi,\rho}
+
Ce^{-\epsilon_{\rm ang}k}\|G\|_{C^8}.
\end{equation*}

It remains to check the angular Frostman estimate. The first
coordinate of $\Phi^e_{\xi,k}$ is
$a
\mapsto
\left[
R_{-\left(a+\arg g'_{\xi,k}(z_0)\right)}e
\right].$
The pushforward of normalized Haar measure on $\mathbb T$ under this
map is normalized Haar measure on $\Pone$. Consequently, the angular
marginal of $\Lambda^e_{\xi,\rho}$ is normalized Haar measure on
$\Pone$, and hence
\begin{equation*}
(\operatorname{proj}_{\Pone})_\#\Lambda^e_{\xi,\rho}\bigl(B(\theta,r)\bigr)
\leq Cr
\quad
\theta\in\Pone,\ 0<r\leq1,
\end{equation*}
uniformly in $e$, $\xi$, $\rho$, and $k$. Thus condition~\textup{(C)}
holds with $\kappa=1$ and $r=8$.
\end{proof}

At this point the three substantive conditions of
Theorem~\ref{thm:stopped-intro} have been verified. As in the
convolution proof, we finish by recording explicitly how the
parameters above match those of the abstract theorem.

\subsection{Choice of parameters and conclusion of proof}
\label{subsec:planar-conclusion}

We apply Theorem~\ref{thm:stopped-intro} with
$E=\Pone.$
The conditioning space and finite auxiliary set are
$(\Xi,m)=(\Omega,\mathbb P),$
$\mathcal J=\mathcal R.$
Thus we take
\[
N=|\mathcal R|,
\qquad
N_0=N.
\]
The normalized measures $\eta_{\xi,j}$ in the notation of the
abstract theorem are the measures $\Theta_{\xi,\rho}$ from
\eqref{eq:planar-remainder-measure}. By
\eqref{eq:planar-remainder-support}, one may fix
$M\in\N$ such that
$\supp\Theta_{\xi,\rho}\subset B(0,M)$
uniformly in all parameters. The interval $J$ was fixed in the
preceding subsection.

Proposition~\ref{prop:planar-condition-A} supplies
\[
1<q<2,
\quad
S>1,
\quad
B<\infty,
\]
for condition~\textup{(A)}. Proposition~\ref{prop:planar-condition-C}
supplies the common choices
\begin{equation}\label{eq:planar-C-parameters}
r=8,
\quad
\kappa=1,
\quad
\epsilon=\epsilon_{\rm ang}.
\end{equation}
Choose $\delta>0$ so small that
\begin{equation}\label{eq:planar-parameter-choice}
0<\delta<1,
\qquad
15\delta<\epsilon_{\rm ang}.
\end{equation}
This is precisely the requirement
$(2r-1)\delta<\epsilon$
from \eqref{eq:smoothing-renewal-balance-intro}.

For sufficiently large $R$, set
\begin{equation}\label{eq:planar-k-choice}
k
=
\left\lfloor
\frac{\log R}{2+\delta}
\right\rfloor.
\end{equation}
Then
$k\asymp\log R,$
and, by \eqref{eq:planar-U},
\begin{equation}\label{eq:planar-U-choice}
U
=
Re^{-2k}
\asymp
e^{\delta k}.
\end{equation}

The geometric error in
Proposition~\ref{prop:planar-renormalization} satisfies
\begin{equation}\label{eq:planar-error-choice}
Re^{-3k}
=
Ue^{-k}
\lesssim
e^{-(1-\delta)k}.
\end{equation}
Thus condition~\textup{(B)} holds with, for example,
$c=1-\delta>0.$
For sufficiently large $R$, the corresponding value of $k$ also
exceeds all fixed lower bounds on $k$ occurring in the stopping,
renormalization, and renewal propositions.

We have therefore verified the hypotheses of
Theorem~\ref{thm:stopped-intro} as follows:
Proposition~\ref{prop:planar-condition-A} gives
condition~\textup{(A)};
Proposition~\ref{prop:planar-renormalization}, together with
\eqref{eq:planar-U-choice}--\eqref{eq:planar-error-choice}, gives
condition~\textup{(B)};
and Proposition~\ref{prop:planar-condition-C} gives
condition~\textup{(C)} with
\[
r=8,
\qquad
\kappa=1,
\qquad
\epsilon=\epsilon_{\rm ang}.
\]
Moreover,
$S+\kappa=S+1>2.$
All constants, compactness bounds, and implicit comparison constants
are uniform in $e\in\Sone$.

Theorem~\ref{thm:stopped-intro} therefore gives
$C_*,\gamma>0$, independent of $e$, such that
\[
Q_R(\mu,[e])
\leq
C_*R^{-\gamma}
\quad
e\in\Sone,\ R\geq2.
\]
So,
$\|\Delta_R((\pi_e)_\#\mu)\|_{L^1}
\leq
C_*R^{-\gamma}$ for
$e\in\Sone,\ R\geq2.$
The criterion from Theorem~\ref{thm:stopped-intro} now yields
\[
(\pi_e)_\#\mu\ll\Leb^1
\quad
\text{for every }e\in\Sone.
\]
This proves Theorem~\ref{thm:planar-conformal-intro}.

\section{Projections of self-affine measures}\label{sec:self-affine-proof}
We now prove Theorem~\ref{thm:self-affine-intro}. Recall from
Subsection~\ref{subsec:self-affine-intro} that
$\Phi=\{f_i:i\in\mathcal A\},
\,
f_i(x)=A_i x+b_i,$
where
$A_i\in\mathrm{GL}(2,\R),$
$\|A_i\|<1,$
$b_i\in\R^2,$
and that $\mu=\mu_{\mathbf{p}}$ is the self-affine measure associated to a strictly
positive probability vector $p=(p_i)_{i\in\mathcal A}$, so that
\eqref{eq:self-conformal} holds. We assume throughout that the
 semigroup
$\Gamma_+^T=\langle A_i^T:i\in\mathcal A\rangle_+$
is strongly irreducible and proximal, in the sense defined in
Subsection~\ref{subsec:self-affine-intro}. We denote by $\nu_F$ the
corresponding Furstenberg measure on $\Pone$, characterized by
\eqref{eq:furstenberg-stationary-intro}, and use the Frostman
dimension $\dimFr\nu_F$ defined in \eqref{eq:frostman-dim-intro}.

Fix $0<s<2$ as in \eqref{eq:self-affine-threshold-intro}; thus,
using the notation from \eqref{eq:def-correlation-dim},
\[
I_s(\mu)<\infty,
\qquad
s+\dimFr\nu_F>2.
\]
Our goal is to prove that, under these assumptions
\[
(\pi_e)_\#\mu\ll\Leb^1
\qquad
\text{for every }e\in\Sone.
\]

As in the preceding section, we apply Theorem~\ref{thm:stopped-intro}. The
present application is simpler: the  stopping
rule is already a genuine prefix stopping time; the cylinder maps
are affine, so the renormalization in condition~\textup{(B)} is
exact. The main new input is the Li--Sahlsten renewal theorem \cite{Li2020Sahl} for the
transpose matrix random walk. In particular, whereas the limiting angular
law in previous applications was the Haar measure, here it is
governed by the Furstenberg measure $\nu_F$.

\subsection{The stopping decomposition and normalized kernels}
\label{subsec:self-affine-stopping-inputs}

As before, let
$\Omega=\mathcal A^{\N},$
$\mathbb P=p^{\N}.$
For a finite word $w=i_1\cdots i_n$, write
\[
A_w=A_{i_1}\cdots A_{i_n},
\qquad
f_w=f_{i_1}\circ\cdots\circ f_{i_n},
\qquad
p_w=p_{i_1}\cdots p_{i_n}.
\]

We follow the notation of Li--Sahlsten
\cite{Li2020Sahl} for the matrix random walk and its norm cocycle,
with the minor modifications that we work with the transposed linear
parts and pass to projective space when convenient.
Put
\begin{equation}\label{eq:self-affine-lambda}
\lambda
=
\sum_{i\in\mathcal A}p_i\,\delta_{A_i^T}
\in\mathcal P(\mathrm{GL}(2,\R)).
\end{equation}
For $g\in\mathrm{GL}(2,\R)$ and $e\in\Sone$, set
\[
\sigma(g,e):=\log|ge|.
\]
Since the alphabet is finite and every $A_i$ is an invertible strict
contraction, there are constants $0<a_-\leq a_+<\infty$ such that
\begin{equation}\label{eq:self-affine-one-step}
-a_+
\leq
\sigma(A_i^T,e)
\leq
-a_-<0
\quad
i\in\mathcal A,\ e\in\Sone.
\end{equation}

For $e\in\Sone$, $k>0$, and $\omega\in\Omega$, define
\begin{equation}\label{eq:self-affine-first-passage}
\tau_k^e(\omega)
=
\min\left\{
n\geq1:
-\log|A_{\omega|n}^Te|>k
\right\}.
\end{equation}
Since
$A_{\omega|n}^T
=
A_{\omega_n}^T\cdots A_{\omega_1}^T,$
this is  the first-passage time for the left random walk with
step distribution \eqref{eq:self-affine-lambda}. It is also a genuine
prefix stopping time on $\Omega$.

Let $\mathcal W_k(e)$ be the cut-set
\begin{equation}\label{eq:self-affine-stopping-line}
\mathcal W_k(e)
=
\left\{
w\in\mathcal A^*:
-\log|A_w^Te|>k
\ \text{and}\
-\log|A_{w^-}^Te|\leq k
\right\},
\end{equation}
where $w^-$ denotes the word obtained from $w$ by deleting its last
symbol. For
$w\in\mathcal W_k(e)$ set, recalling \eqref{eq:self-affine-one-step},
\begin{equation*}
\theta_w(e)=[A_w^Te]\in\Pone,
\quad
z_w(e)=\log|A_w^Te|+k\in[-a_+,0].
\end{equation*}
Thus
$|A_w^Te|
=
e^{-k+z_w(e)}.$ We write, 
\begin{equation}\label{eq:self-affine-Pek}
P_{e,k}
=
\sum_{w\in\mathcal W_k(e)}
p_w\,\delta_{(\theta_w(e),z_w(e))}
\in
\mathcal P(\Pone\times[-a_+,0]).
\end{equation}
Since $\mathcal W_k(e)$ is a cut-set, iterating
\eqref{eq:self-conformal} gives 
\begin{equation}\label{eq:self-affine-stopping-identity}
\mu
=
\sum_{w\in\mathcal W_k(e)}
p_w(f_w)_\#\mu.
\end{equation}

We can now specify the objects from Theorem~\ref{thm:stopped-intro}
that arise from this decomposition. We take
\[
(\Xi,m)=(\{\ast\},\delta_\ast),
\quad
N=N_0=1,
\quad
J=[-a_+,0],
\]
and set
\[
\eta_{\ast,1}=\mu.
\]
The corresponding probability kernel is
$P_{\ast,1}=P_{e,k},$ defined in \eqref{eq:self-affine-Pek}. 
Thus there is no  conditioning  or auxiliary finite
index in the self-affine application. The limiting law
$\Lambda_{\ast,1}$ required in condition~\textup{(C)} will be defined
below from the Li--Sahlsten renewal theorem, Proposition \ref{prop:self-affine-condition-C} below.

\subsection{Renormalization: condition \textup{(B)}}
\label{subsec:self-affine-renormalization}

We next verify condition~\textup{(B)} of
Theorem~\ref{thm:stopped-intro}. Recall from \eqref{eq:Q-def-intro}
that, for $\eta\in\mathcal P(\R^2)$ and $\theta=[v]\in\Pone$,
$Q_R(\eta,\theta)
=
\|\Delta_R((\pi_v)_\#\eta)\|_{L^1(\R)},$
and from \eqref{eq:H-def} that
$H_{U,\eta}(\theta,z)
=
Q_{Ue^z}(\eta,\theta).$

We use the probability space and kernels specified in the preceding
subsection:
\[
(\Xi,m)=(\{\ast\},\delta_\ast),
\quad
N=1,
\quad
\eta_{\ast,1}=\mu,
\quad
P_{\ast,1}=P_{e,k},
\]
with $J=[-a_+,0]$.

\begin{lemma}\label{lem:self-affine-renormalization}
For every $e\in\Sone$, $k\geq1$, and $R\geq1$, set
$U=Re^{-k}.$
Then
\begin{equation}\label{eq:self-affine-condition-B}
Q_R(\mu,[e])
\leq
\int_\Xi
\int_{\Pone\times J}
H_{U,\eta_{\xi,1}}(\theta,z)\,
dP_{\xi,1}(\theta,z)\,dm(\xi).
\end{equation}
\end{lemma}
In particular, condition~\textup{(B)} of
Theorem~\ref{thm:stopped-intro} holds, with no error term.
\begin{proof}
For a finite word $w$, write
$f_w(x)=A_wx+b_w.$
Then
\[
\pi_e(f_w(x))
=
\langle b_w,e\rangle
+
|A_w^Te|\,
\pi_{A_w^Te/|A_w^Te|}(x).
\]
Hence, by the affine covariance \eqref{eq:L1-scaling},
$Q_R((f_w)_\#\mu,[e])
=
Q_{R|A_w^Te|}(\mu,\theta_w(e)).$
Applying Minkowski's inequality to
\eqref{eq:self-affine-stopping-identity}, and recalling that for
$w\in\mathcal W_k(e)$,
$|A_w^Te|
=
e^{-k+z_w(e)},$
we obtain
\[
\begin{aligned}
Q_R(\mu,[e])
&\leq
\sum_{w\in\mathcal W_k(e)}
p_w
Q_{Re^{-k+z_w(e)}}(\mu,\theta_w(e))\\
&=
\int_{\Pone\times J}
H_{U,\mu}(\theta,z)\,dP_{e,k}(\theta,z).
\end{aligned}
\]
By the definitions
$\eta_{\ast,1}=\mu$, $P_{\ast,1}=P_{e,k}$, and
$(\Xi,m)=(\{\ast\},\delta_\ast)$, the last expression is exactly the
right-hand side of \eqref{eq:self-affine-condition-B}.
\end{proof}

\subsection{Uniform \(q\)-mass: condition \textup{(A)}}
\label{subsec:self-affine-qmass}
We next verify condition~\textup{(A)} of
Theorem~\ref{thm:stopped-intro}. Recall that, in the notation fixed
above,
$\eta_{\ast,1}=\mu.$

\begin{lemma}\label{lem:self-affine-condition-A}
Condition~\textup{(A)} of Theorem~\ref{thm:stopped-intro} holds with
$q=2$ and
$S=s.$
That is,
\begin{equation}\label{eq:self-affine-condition-A}
S_{m,2}(\eta_{\ast,1})
=
S_{m,2}(\mu)
\leq
C_s I_s(\mu)\,2^{-ms}
\quad
m\geq1.
\end{equation}
\end{lemma}

\begin{proof}
By \eqref{eq:self-affine-threshold-intro},
$I_s(\mu)<\infty$. Hence
\eqref{eq:self-affine-condition-A} follows directly from
Lemma~\ref{lem:energy-qmass}.
\end{proof}

\begin{remark}\label{rem:self-affine-energy-role}
This is the only place where the finite-energy assumption in
\eqref{eq:self-affine-threshold-intro} is used. Recall from the
discussion following Theorem~\ref{thm:self-affine-intro} that if
$\lim_{q\downarrow1}D_q(\mu)=\dim\mu,$
then the same argument as in the preceding sections would yield the same
conclusion under the assumption
$\dim\mu+\dimFr\nu_F>2.$
\end{remark}

\subsection{The angular law: condition \textup{(C)}}
\label{subsec:self-affine-renewal}
We now turn to the only substantial dynamical input. In the
self-conformal applications, the limiting angular law is Haar measure.
For the transpose matrix random walk \eqref{eq:self-affine-lambda}, it is
 the Furstenberg measure $\nu_F$ from
\eqref{eq:furstenberg-stationary-intro}.

Recall the norm cocycle
$
\sigma(g,e)=\log|ge|$
defined above. Since it is invariant under replacing $e$ by $-e$, we
also write $\sigma(g,\theta)$ for $\theta=[e]\in\Pone$, and write
$g\theta=[ge]$ for the induced projective action. Set
\[
\sigma_\lambda
=
\int_{\Pone}
\int_{\mathrm{GL}(2,\R)}
\sigma(g,\theta)\,
d\lambda(g)\,d\nu_F(\theta)
<0.
\]

Define $\Lambda\in\mathcal P(\Pone\times[-a_+,0])$ by
\[
\int G\,d\Lambda
=
\frac{1}{|\sigma_\lambda|}
\int_{\Pone}
\int_{\mathrm{GL}(2,\R)}
\int_0^{-\sigma(g,\theta)}
G\bigl(g\theta,\sigma(g,\theta)+u\bigr)
\,du\,d\lambda(g)\,d\nu_F(\theta).
\]
The normalization follows immediately by taking $G\equiv1$. In the
notation of Theorem~\ref{thm:stopped-intro}, we set
$\Lambda_{\ast,1}=\Lambda.$

\begin{proposition}[Li--Sahlsten renewal theorem]
\label{prop:self-affine-condition-C}
There exist $C<\infty$ and $\epsilon_{\rm LS}>0$ such that, uniformly
in $e\in\Sone$ and $k\geq1$,
\[
\left|
\int G\,dP_{e,k}
-
\int G\,d\Lambda
\right|
\leq
Ce^{-\epsilon_{\rm LS}k}\|G\|_{C^3}
\]
for every $G\in C^3(\Pone\times[-a_+,0])$.

Moreover, for every $0<\kappa<\dimFr\nu_F$, the angular marginal
$\bar\sigma
=
(\operatorname{proj}_{\Pone})_\#\Lambda$
satisfies
\[
\bar\sigma(B(\theta,\rho))
\leq
C_\kappa\rho^\kappa
\quad
\theta\in\Pone,\ 0<\rho\leq1.
\]
\end{proposition}

The effective renewal estimate is the projective form of
\cite[Proposition~3.7]{Li2020Sahl}; its proof is given in
\cite[Section~5.4]{Li2020Sahl}. To check that it applies here, note
that in dimension two the SIP hypothesis implies the algebraic
assumptions of that proposition by
\cite[Remark~1.2(3)]{Li2020Sahl}. Their result is formulated on
$\Sone$. After projectivizing, the limiting spherical stationary law
projects to the unique Furstenberg measure $\nu_F$. Since both the
norm cocycle and the test functions considered here are antipodally
invariant, this gives precisely the measure $\Lambda$ above. The
regularity appearing in \cite[Proposition~3.7]{Li2020Sahl} is
controlled by the $C^3$ norm used here.

It remains only to record the Frostman estimate. If $B\subset\Pone$
is a ball, then by the definition of $\Lambda$ and
\eqref{eq:self-affine-one-step},
\[
\begin{aligned}
\bar\sigma(B)
&\leq
\frac{a_+}{|\sigma_\lambda|}
\int_{\Pone}
\int_{\mathrm{GL}(2,\R)}
\mathbf 1_B(g\theta)\,
d\lambda(g)\,d\nu_F(\theta)\\
&=
\frac{a_+}{|\sigma_\lambda|}\nu_F(B),
\end{aligned}
\]
where the last equality follows from
\eqref{eq:furstenberg-stationary-intro}. The required estimate now
follows directly from \eqref{eq:frostman-dim-intro}.

Finally, choose
\[
0<\kappa<\dimFr\nu_F
\qquad\text{such that}\qquad
s+\kappa>2,
\]
which is possible by \eqref{eq:self-affine-threshold-intro}.
Proposition~\ref{prop:self-affine-condition-C} therefore gives
condition~\textup{(C)} of Theorem~\ref{thm:stopped-intro}, with
\[
r=3,
\qquad
\epsilon=\epsilon_{\rm LS},
\qquad
\kappa<\dimFr\nu_F.
\]
Together with Lemmas~\ref{lem:self-affine-renormalization} and
\ref{lem:self-affine-condition-A}, all three conditions of
Theorem~\ref{thm:stopped-intro} have now been verified.

\subsection{Choice of parameters and conclusion of proof}
\label{subsec:self-affine-conclusion}
We now apply Theorem~\ref{thm:stopped-intro} with
$E=\Pone,$ $J=[-a_+,0].$
Recall that
\[
(\Xi,m)=(\{\ast\},\delta_\ast),
\quad
N=N_0=1,
\quad
\eta_{\ast,1}=\mu,
\]
and that
$P_{\ast,1}=P_{e,k},$
$\Lambda_{\ast,1}=\Lambda.$
Since $\mu$ is compactly supported, the support parameter $M$ in
Theorem~\ref{thm:stopped-intro} may be fixed independently of all
parameters.

Lemma~\ref{lem:self-affine-condition-A} gives
condition~\textup{(A)} with
$q=2,
S=s.$
Lemma~\ref{lem:self-affine-renormalization} gives
condition~\textup{(B)}, with no error term. Finally,
Proposition~\ref{prop:self-affine-condition-C} gives
condition~\textup{(C)} with
$r=3,
\,
\epsilon=\epsilon_{\rm LS},$
and with any $0<\kappa<\dimFr\nu_F$. By
\eqref{eq:self-affine-threshold-intro}, we may choose such a $\kappa$
so that
$S+\kappa=s+\kappa>2.$

Choose $\delta>0$ sufficiently small that
$5\delta<\epsilon_{\rm LS}.$
Since $5=2r-1$, this is precisely the requirement
\eqref{eq:smoothing-renewal-balance-intro}.

For sufficiently large $R$, set
$k
=
\left\lfloor
\frac{\log R}{1+\delta}
\right\rfloor.$
Then
$k\asymp\log R.$
Moreover, with the choice
\[
U=Re^{-k}
\]
from Lemma~\ref{lem:self-affine-renormalization}, we have
$U\asymp e^{\delta k}.$
Thus all the hypotheses of Theorem~\ref{thm:stopped-intro} are
satisfied uniformly in $e\in\Sone$. Therefore, there exist
$C_*,\gamma>0$ such that
\[
Q_R(\mu,[e])
\leq
C_*R^{-\gamma}
\quad
e\in\Sone,\ R\geq2.
\]
In particular,
\[
(\pi_e)_\#\mu\ll\Leb^1
\quad
\text{for every }e\in\Sone.
\]
This proves Theorem~\ref{thm:self-affine-intro}.
\section{Explicit examples}\label{sec:explicit-examples}

We conclude  with some explicit examples to which the three main theorems apply.  

\subsection{Explicit examples of absolutely continuous convolutions}
\label{subsec:explicit-convolution}

Recall from Subsection~\ref{subsec:convolutions-intro} that a $C^2$ IFS
$\Phi=\{f_i:i\in\mathcal A\}$ satisfies uniform contraction
\eqref{eq:1d-uniform-contraction}, that its attractor $K_\Phi$ is
characterized by \eqref{eq:attractor}, and that a self-conformal measure
for $\Phi$ is a measure satisfying \eqref{eq:self-conformal}. We also use
the notions of a linear IFS and of a $C^2$ IFS which is $C^2$-conjugate
to linear exactly as defined in Subsection~\ref{subsec:convolutions-intro}.

Throughout this subsection, let $\nu$ denote the middle-third Cantor
measure, that is, the equal-weight self-similar measure for the IFS
$\{ S_0(x)=\frac{x}{3},
S_1(x)=\frac23+\frac{x}{3} \}.$
Thus, $\nu$ is the unique measure such that
$\nu=\frac12(S_0)_\#\nu+\frac12(S_1)_\#\nu.$
The measure $\nu$ is also Ahlfors--David regular and it is well known that
$\dim \nu=\frac{\log2}{\log3}>0.63.$

\begin{cor}\label{cor:explicit-convolutions}
For $|\varepsilon|\leq1/24$, let
$\Phi_\varepsilon
=
\left\{
x\mapsto\frac{x}{3},
\quad
x\mapsto\frac23+\frac{x}{3}+\varepsilon x(1-x)
\right\},$
and let $\mu_\varepsilon \in \mathcal{P}(\mathbb{R})$ be its equal-weight self-conformal measure.
Then, for every $\varepsilon\neq0$, the $C^\omega(\mathbb R)$ IFS
$\Phi_\varepsilon$ satisfies the SSC, is not $C^2$-conjugate to linear,
and
$\mu_\varepsilon*\nu\ll\Leb^1.$
\end{cor}
Notice that $\mu_\varepsilon\rightarrow\nu$ as $\varepsilon\to0$ in the weak-$*$ sense, whereas
$\mu_0=\nu$ and $\nu*\nu$ is singular.  Thus, even within this very rigid real-analytic family of IFS, absolute continuity of the convolution is not closed under weak-$*$ limits.

We  record two elementary criteria for checking the non-linearity
hypothesis in Theorem~\ref{thm:convolution-intro}.

\begin{lemma}\label{lem:check-nonlinearity}
Let $\Phi=\{f_i:i\in\mathcal A\}$ be a $C^2$ IFS with attractor
$K=K_\Phi$.

\begin{enumerate}
\item Suppose that $\Phi$ is real analytic. For $i\in\mathcal A$, let
$y_i$ be the fixed point of $f_i$, and, for a finite word
$w=w_1\cdots w_n$, let $y_w$ be the fixed point of $f_w$. If
\begin{equation}\label{eq:multiplier-obstruction}
f_w'(y_w)
\neq
\prod_{j=1}^n f_{w_j}'(y_{w_j}) \text{ for some } w,
\end{equation}
then $\Phi$ is not $C^2$-conjugate to linear.

\item For $\omega\in\mathcal A^{\mathbb N}$ and $x\in K$, set
$D_{\omega,n}(x)
=
\frac{d}{dx}\log\left|f_{\omega|n}'(x)\right|.$
If $\Phi$ is $C^2$-conjugate to linear, then, for every $x\in K$,
$\lim_{n\to\infty}D_{\omega,n}(x)$
exists and is independent of $\omega\in\mathcal A^{\mathbb N}$.
\end{enumerate}
\end{lemma}

\begin{proof}
For part~\textup{(1)}, suppose towards a contradiction that $\Phi$ is
$C^2$-conjugate to linear. Since $\Phi$ is real analytic,  \cite[Claim~6.1]{algom2023polynomial} implies that there
is a real-analytic diffeomorphism $h$ such that
$\widetilde f_i=h\circ f_i\circ h^{-1},$ for every $i\in\mathcal A,$
is  affine.  In particular, each $\widetilde f_i'$ is
constant.
Put
$z_i=h(y_i),$
and
$z_w=h(y_w).$
Since $y_i$ and $y_w$ are the fixed points of $f_i$ and $f_w$,
respectively, $z_i$ and $z_w$ are the corresponding fixed points of
$\widetilde f_i$ and $\widetilde f_w$.  By the chain rule,
\begin{align*}
\widetilde f_i'(z_i)
&=
\frac{h'(f_i(y_i))}{h'(y_i)}f_i'(y_i)
=
f_i'(y_i),\\
\widetilde f_w'(z_w)
&=
\frac{h'(f_w(y_w))}{h'(y_w)}f_w'(y_w)
=
f_w'(y_w).
\end{align*}
On the other hand, since the maps $\widetilde f_i$ are similarities,
their derivatives are constant, and therefore
\begin{align*}
f_w'(y_w)
&=
\widetilde f_w'(z_w)\\
&=
\prod_{j=1}^n \widetilde f_{w_j}'(z_{w_j})\\
&=
\prod_{j=1}^n f_{w_j}'(y_{w_j}).
\end{align*}
This contradicts \eqref{eq:multiplier-obstruction}, and proves
part~\textup{(1)}.

For part~\textup{(2)}, suppose again towards a contradiction that a
$C^2$ diffeomorphism $h$ conjugates $\Phi$ to a linear IFS
$\widetilde\Phi
=
\{\widetilde f_i:i\in\mathcal A\},
\,
\widetilde f_i=h\circ f_i\circ h^{-1} \}.$
Recall that, with our definition of linearity from
Subsection~\ref{subsec:convolutions-intro},
$\widetilde f_i''(z)=0$ for all $i\in\mathcal A,\ z\in h(K)).$ 
Since each $\widetilde f_i$ maps $h(K)$ into itself, the chain rule
implies that
$\widetilde f_w''(z)=0$
for every $z\in h(K)$
and
for every finite word $w$.
Arguing similarly to
\cite[Lemma~6.3]{algom2023polynomial},  from
$\widetilde f_w\circ h=h\circ f_w$
we obtain, for $x\in K$,
$\widetilde f_w'(h(x))h'(x)
=
h'(f_w(x))f_w'(x).$
Differentiating this identity once more gives
\begin{align*}
&\widetilde f_w''(h(x))(h'(x))^2
+\widetilde f_w'(h(x))h''(x)\\
&\hspace{35mm}
=
h''(f_w(x))(f_w'(x))^2
+h'(f_w(x))f_w''(x).
\end{align*}
The first term on the left vanishes, since $h(x)\in h(K)$ and
$\widetilde f_w$ is linear on $h(K)$.  Using also
\[
\widetilde f_w'(h(x))
=
\frac{h'(f_w(x))f_w'(x)}{h'(x)},
\]
we obtain
\begin{equation}\label{eq:distortion-under-conjugacy}
\frac{f_w''(x)}{f_w'(x)}
=
\frac{h''(x)}{h'(x)}
-
\frac{h''(f_w(x))}{h'(f_w(x))}f_w'(x).
\end{equation}
By \eqref{eq:1d-uniform-contraction},
$|f_w'(x)|\longrightarrow0$
uniformly as  $|w|\to\infty,$
while $h''/h'$ is bounded on $K$.  Hence
\eqref{eq:distortion-under-conjugacy} implies that, for every
$x\in K$ and every $\omega\in\mathcal A^{\mathbb N}$,
\[
D_{\omega,n}(x)
=
\frac{f_{\omega|n}''(x)}{f_{\omega|n}'(x)}
\longrightarrow
\frac{h''(x)}{h'(x)}.
\]
In particular, the limit exists and is independent of $\omega$, which
proves part~\textup{(2)}.
\end{proof}

\begin{remark}\label{rem:finite-regularity-examples}
Using Lemma~\ref{lem:check-nonlinearity}\textup{(2)}, one can also
produce examples of  finite regularity which are not real
analytic.  Fix $2<r<\infty$.  By
\cite[Theorem~2; see also Claim~22]{algom2024linear}, there exists a
strongly separated $C^r$ IFS
$\Psi_r=\{g_0,g_1\},$
 not of class $C^s$ for any $s>r$, such that, writing $K_r$ for its
attractor,
\[
g_0(0)=0,\qquad g_1(1)=1,
\quad
g_i'(x)=\frac14
\quad x\in K_r,\ i=0,1.
\]
For $\delta>0$ define
$f_0=g_0,$
$f_1(x)=g_1(x)+\delta(1-x)^2,$
$\Phi_{r,\delta}=\{f_0,f_1\}.$
For $\delta>0$ sufficiently small, $\Phi_{r,\delta}$ is still strongly
separated and has exactly $C^r$ regularity.  Moreover,
$D_{0^\infty,n}(1)=0,$
whereas
\[
D_{1^\infty,n}(1)
=
\frac{2\delta}{1/4}
\sum_{k=0}^{n-1}\left(\frac14\right)^k
\longrightarrow
\frac{2\delta}{(1/4)(3/4)}
\neq0.
\]
Hence Lemma~\ref{lem:check-nonlinearity}\textup{(2)} shows that
$\Phi_{r,\delta}$ is not $C^2$-conjugate to linear.

If $\mu_{r,\delta}$ is the uniform self-conformal measure for
$\Phi_{r,\delta}$, then
$\dim \mu_{r,\delta}\rightarrow\frac12
\quad \delta\to0.$
Thus, for $\delta>0$ sufficiently small,
$\dim \mu_{r,\delta}+\dim_H\nu>1,$
and Theorem~\ref{thm:convolution-intro} gives
$\mu_{r,\delta}*\nu\ll\Leb^1.$
The same argument, starting from the $C^\infty$ non-real-analytic
examples in \cite[Theorem~2]{algom2024linear}, gives corresponding
$C^\infty$ examples.
\end{remark}

\begin{proof}
For $|\varepsilon|\leq1/24$,
$\frac7{24}
\leq
f_{1,\varepsilon}'(x)
=
\frac13+\varepsilon(1-2x)
\leq
\frac38,$
while
$f_{0,\varepsilon}([0,1])=[0,1/3],$ and
$f_{1,\varepsilon}([0,1])=[2/3,1].$
Thus $\Phi_\varepsilon$ is a strongly separated real-analytic IFS.  Suppose
$\varepsilon\neq0$, and let $x_\varepsilon$ be the fixed point of
$f_{0,\varepsilon}\circ f_{1,\varepsilon}$.  The fixed points of
$f_{0,\varepsilon}$ and $f_{1,\varepsilon}$ are $0$ and $1$.  If
$\Phi_\varepsilon$ were $C^2$-conjugate to linear, then
Lemma~\ref{lem:check-nonlinearity}\textup{(1)} would give
$\frac13 f_{1,\varepsilon}'(x_\varepsilon)
=
\frac13 f_{1,\varepsilon}'(1).$
Since $f_{1,\varepsilon}'$ is injective for $\varepsilon\neq0$, this
would imply $x_\varepsilon=1$, contradicting
$(f_{0,\varepsilon}\circ f_{1,\varepsilon})(1)=1/3$.  Hence
$\Phi_\varepsilon$ is not $C^2$-conjugate to linear.  Finally, the
dimension formula  and the above derivative bound give
\[
\dim \mu_\varepsilon
\geq
\frac{\log2}
{\frac12\log3+\frac12\log(24/7)}
>0.59.
\]
Since $\dim \nu=\log2/\log3>0.63$, we have
$\dim \mu_\varepsilon+\dim_H\nu>1$, and
Theorem~\ref{thm:convolution-intro} yields
$\mu_\varepsilon*\nu\ll\Leb^1$.
\end{proof}

\subsection{Explicit nonlinear self-conformal measures with all projections absolutely continuous}
\label{subsec:explicit-planar}

We now give a concrete application of
Theorem~\ref{thm:planar-conformal-intro}.  The example is a small
holomorphic perturbation of a strongly separated planar self-similar
system.

Let
$\omega=e^{2\pi i/3},
r=\frac38,
p=\frac45\,\omega^2.$
For $|\varepsilon|\leq 1/200$, define the $C^\omega (D)$ IFS
\begin{equation}\label{eq:explicit-planar-family}
\Phi_\varepsilon
=
\left \lbrace  F_0(z)=rz+\frac12,\, F_1(z)=rz+\frac{\omega}{2}
,\, G_\varepsilon(z)
=
p+r(z-p)+\varepsilon(z-p)^2.\right \rbrace.
\end{equation}
Notice that
$(1-r)p=\frac{\omega^2}{2},$
so, when $\varepsilon=0$, the three maps in
\eqref{eq:explicit-planar-family} are the similarities
$z\mapsto rz+\frac{\omega^j}{2},
\,
j=0,1,2$; in particular, $\Phi_0$ is a self-similar IFS.
Let $\mu_\varepsilon$ denote the equal-weight self-conformal measure
associated to $\Phi_\varepsilon$.

\begin{cor}\label{cor:explicit-planar}
For every
$0<|\varepsilon|\leq\frac1{200}$
 $\Phi_\varepsilon$ is a strongly separated
$C^\omega(\C)$ IFS which is not $C^\omega$-conjugate to a
self-similar IFS, and
$\dim\mu_\varepsilon>1.$
Therefore,
\[
(\pi_e)_\#\mu_\varepsilon\ll\Leb^1
\qquad
\text{for every }e\in\Sone.
\]
\end{cor}

\begin{proof}
We verify the hypotheses of
Theorem~\ref{thm:planar-conformal-intro}.

First, it is direct to check that $G_\varepsilon$ is an injective contraction with nonvanishing
derivative on a neighbourhood of $D$, with
\begin{equation}\label{eq:explicit-planar-derivative-bounds}
\frac{357}{1000}
\leq
|G_\varepsilon'(z)|
\leq
\frac{393}{1000}.
\end{equation}

The three first-level cylinders are   disjoint.  The maps
$F_0$ and $F_1$ send $D$ into discs of radius $3/8$ centred at
$1/2$ and $\omega/2$, respectively, while
\[
G_\varepsilon(D)
\subset
B\left(
\frac{\omega^2}{2},
\frac38+\frac1{200}\left(\frac95\right)^2
\right)
=
B\left(
\frac{\omega^2}{2},
\frac{489}{1250}
\right).
\]
All three discs are compactly contained in $D$.  Their centres are
the vertices of an equilateral triangle of side length
$\sqrt3/2$, and
$\frac34<\frac{\sqrt3}{2},$
$\frac38+\frac{489}{1250}<\frac{\sqrt3}{2}.$ 
Hence $\Phi_\varepsilon$ satisfies the strong separation condition.

We next check the nonlinearity hypothesis.  We use 
Lemma~\ref{lem:check-nonlinearity}\textup{(1)}.  Although that lemma was
stated for real one-dimensional IFSs, its proof applies verbatim in the
present holomorphic setting: if an analytic IFS is analytically conjugate
to a self-similar IFS, then the multiplier at the fixed point of every
composition is the product of the multipliers at the fixed points of its
letters.
Now $G_\varepsilon$ fixes $p$ and
$G_\varepsilon'(p)=r.$
Let $z_\varepsilon$ be the fixed point of
$F_0\circ G_\varepsilon$.  If $\Phi_\varepsilon$ were
$C^\omega$-conjugate to a self-similar IFS, the preceding multiplier
identity would give
\[
(F_0\circ G_\varepsilon)'(z_\varepsilon)
=
F_0'\,G_\varepsilon'(p)
=
r^2.
\]
On the other hand,
$(F_0\circ G_\varepsilon)'(z_\varepsilon)
=
rG_\varepsilon'(z_\varepsilon),$
and hence
$G_\varepsilon'(z_\varepsilon)=r.$
Since $\varepsilon\neq0$ and
$G_\varepsilon'(z)
=
r+2\varepsilon(z-p),$
this forces $z_\varepsilon=p$.  But
\[
(F_0\circ G_\varepsilon)(p)=F_0(p)\neq p,
\]
since the fixed point of $F_0$ is $4/5$, whereas
$p=(4/5)\omega^2$.  This contradiction shows that
$\Phi_\varepsilon$ is not $C^\omega$-conjugate to a self-similar IFS.

It remains to check the dimension hypothesis.  In this equal weights case, the
Lyapunov exponent of $\mu_\varepsilon$ is
$\chi_\varepsilon
:=
-\frac13\sum_{f\in\Phi_\varepsilon}
\int \log |f'(z)|\,d\mu_\varepsilon(z).$ 
Since $\Phi_\varepsilon$ satisfies the strong separation condition,
the standard  dimension formula
\cite[Theorem~2.8]{feng2009dimension} gives
$dim\mu_\varepsilon
=
\frac{\log3}{\chi_\varepsilon}.$
The maps $F_0,F_1$ have derivative $3/8$, while
\eqref{eq:explicit-planar-derivative-bounds} gives
$|G_\varepsilon'|\geq357/1000$.  Therefore
$\chi_\varepsilon
\leq
\frac23\log\frac83
+
\frac13\log\frac{1000}{357},$
and so
\[
\dim\mu_\varepsilon
\geq
\frac{\log3}
{\frac23\log(8/3)+\frac13\log(1000/357)}
>1.10.
\]
In particular, $K_{\Phi_\varepsilon}=\supp\mu_\varepsilon$ cannot be
contained in a real-analytic planar curve, since every such curve has
Hausdorff dimension one.

Thus all the hypotheses of
Theorem~\ref{thm:planar-conformal-intro} are satisfied, and hence
\[
(\pi_e)_\#\mu_\varepsilon\ll\Leb^1
\qquad
\text{for every }e\in\Sone.
\]
\end{proof}

\subsection{A singular self-affine measure with every projection absolutely continuous}
\label{subsec:explicit-self-affine}
We conclude with a concrete application of
Theorem~\ref{thm:self-affine-intro}.  Recall from
Subsection~\ref{subsec:self-affine-intro} that a planar self-affine IFS
has the form
$\Phi=\{x\mapsto A_i x+b_i:i\in\mathcal A\},$ where
$A_i\in\mathrm{GL}(2,\R)$, $\|A_i\|<1$, and $b_i\in \mathbb{R}^2$.
We say that $\Phi$ is \emph{SIP} if the semigroup generated by the
transposed linear parts is strongly irreducible and proximal.  In this
case we write $\nu_F$ for the corresponding stationary Furstenberg
measure \eqref{eq:furstenberg-stationary-intro} on $\Pone$, and recall that
\[
\dimFr\nu_F
=
\sup\left\{
\kappa:
\nu_F(B(\theta,r))\leq C_\kappa r^\kappa
\text{ for all }\theta\in\Pone,\ 0<r\leq1
\right\}.
\]
We also recall the notation
$I_s(\mu)
=
\iint |x-y|^{-s}\,d\mu(x)\,d\mu(y)$
from \eqref{eq:def-correlation-dim}.

For $\sigma\in\{-1,1\}$, set
\begin{equation}\label{eq:explicit-Bpm}
B_\sigma
=
\sqrt{\frac{80}{39}}
\begin{pmatrix}
1&\frac{\sigma}{200}\\[2mm]
\frac{\sigma}{2}&\frac{49}{100}
\end{pmatrix}
\in\mathrm{SL}_2(\R),
\qquad
A_\sigma=\frac8{25}B_\sigma^T.
\end{equation}
Let
$Q=[-1,1]\times\left[-\frac1{36},\frac1{36}\right],$
$y_j=\frac{j}{54},$
$j\in\{-1,0,1\},$
and define the six affine maps
\begin{equation}\label{eq:explicit-six-maps}
\Phi= \left \lbrace f_{\sigma,j}(x)
=
A_\sigma x+
\left(\frac{\sigma}{2},y_j\right) \right \rbrace_{\sigma\in\{-1,1\},
\,
j\in\{-1,0,1\}.
}
\end{equation}
Let $\mu_A$ be the equal-weight self-affine measure for this IFS.

\begin{cor}\label{cor:explicit-self-affine}
The IFS $\Phi$ from \eqref{eq:explicit-six-maps} satisfies the strong separation
condition and is SIP.  If $\nu_F$ denotes its Furstenberg measure, then
$\dimFr\nu_F
\geq
\kappa_0
:=
\frac{\log2}{\log(13467/6500)}
>0.95,
$ while
$I_s(\mu_A)<\infty$
for every 
$0<s<d_0$, where
$d_0
:=
\frac{\log6}{\log(125/24)}
>1.08.$ 
In particular,
$d_0+\kappa_0>2,$
and hence
\[
(\pi_e)_\#\mu_A\ll\Leb^1
\qquad
\text{for every }e\in\Sone.
\]
Moreover,
$\mu_A\perp\Leb^2.$
\end{cor}

Before proving the corollary, let us note that by
\cite[Theorem~1.2]{BaranyHochmanRapaport2019},
the strong separation and SIP properties proved below imply that
$\dim\mu_A$ equals its Lyapunov dimension.  A direct computation then shows that
\begin{equation}\label{eq:explicit-self-affine-dimension-range}
1.572<\dim\mu_A<1.705.
\end{equation}
Thus $\mu_A$ is indeed a  singular measure.

\begin{proof}[Proof of Corollary~\ref{cor:explicit-self-affine}]
We first verify SIP and estimate the Furstenberg measure.  Since
$A_\sigma^T=\frac8{25}B_\sigma,$
the corresponding projective maps are the projective actions of
$B_\sigma$.  In the slope coordinate $t=y/x$ they are
$\psi_\sigma(t)
=
\frac{
\frac{\sigma}{2}+\frac{49}{100}t
}{
1+\frac{\sigma}{200}t
}.$ A direct calculation gives
\[
\psi_-([-1,1])
=
\left[-\frac{66}{67},-\frac2{199}\right],
\qquad
\psi_+([-1,1])
=
\left[\frac2{199},\frac{66}{67}\right],
\]
so the two first-level projective cylinders are disjoint.  Moreover,
$|\psi_\sigma'(t)|
=
\frac{39/80}{(1+\sigma t/200)^2},$
and therefore, for $|t|\leq1$,
\begin{equation}\label{eq:explicit-projective-derivative}
\delta_0
:=
\frac{6500}{13467}
\leq
|\psi_\sigma'(t)|
\leq
\frac{19500}{39601}
<1.
\end{equation}
The two matrices $B_\sigma$ are hyperbolic, since
$\operatorname{tr}(B_\sigma)
=
\frac{149}{100}\sqrt{\frac{80}{39}}
>2.$
Their respective pairs of fixed slopes are
$\{-51\pm\sqrt{2701}\},$ and
$\{51\pm\sqrt{2701}\},$
which are disjoint.  Thus the semigroup is proximal.  If a finite
nonempty set of projective lines were invariant under both generators,
each generator would permute it.  Every point of this set would
therefore be periodic for both hyperbolic projective transformations,
and hence would have to be an eigenline of both matrices, contradicting
the disjointness of the two  fixed-point sets.  Thus the
semigroup is strongly irreducible.

Since each of the three maps with linear part $A_-$ has weight $1/6$,
and similarly for $A_+$, the transpose-projective random walk is the
equal-weight walk generated by $\psi_-$ and $\psi_+$.  Let $\nu_F$ be
its stationary measure.  The first-level projective intervals are
separated by a gap $4/199$, and
\eqref{eq:explicit-projective-derivative} implies that two level-$n$
cylinders which first branch at level $j$ are separated by at least a
fixed multiple of $\delta_0^j$.  Hence a projective ball of radius
comparable to $\delta_0^n$ meets only $O(1)$ level-$n$ cylinders.
Since every such cylinder has $\nu_F$-mass $2^{-n}$,
\[
\nu_F(B(\theta,r))
\leq
Cr^{\kappa_0},
\qquad
\kappa_0
=
\frac{\log2}{\log(1/\delta_0)}
=
\frac{\log2}{\log(13467/6500)}.
\]
Therefore
\begin{equation}\label{eq:explicit-kappa}
\dimFr\nu_F\geq\kappa_0>0.95.
\end{equation}

We next verify strong separation and the energy estimate.  We use the
elementary singular-value bounds
\begin{equation}\label{eq:explicit-B-singular-values}
\|B_\sigma\|<\frac53,
\qquad
\alpha_2(B_\sigma)>\frac35.
\end{equation}
Indeed,
$\operatorname{tr}(B_\sigma^TB_\sigma)
=
\frac{917}{300}.$
Since $\det B_\sigma=1$, its singular values are
$\rho,\rho^{-1}$ for some $\rho\geq1$, and
$\rho^2+\rho^{-2}
=
\frac{917}{300}
<
\left(\frac53\right)^2
+
\left(\frac35\right)^2,$
which proves \eqref{eq:explicit-B-singular-values}.  So,
\begin{equation}\label{eq:explicit-A-singular-values}
\|A_\sigma\|<\frac8{15}<1,
\qquad
\alpha_2(A_\sigma)>\frac{24}{125}.
\end{equation}
Put
$c_0=\frac8{25}\sqrt{\frac{80}{39}}.$
For $x\in Q$, a direct estimate gives
\[
|(A_\sigma x)_1|
\leq
c_0\left(1+\frac1{72}\right)
<0.465<\frac12,
\text{ and    } 
|(A_\sigma x)_2|
\leq
c_0\left(\frac1{200}+\frac{49}{3600}\right)
<0.00854<\frac1{108}.
\]
Thus the two values of $\sigma$ occupy disjoint left and right columns
of $Q$, while, for each fixed $\sigma$, the three values of $j$ occupy
disjoint horizontal strips.  Hence the six sets
$f_{\sigma,j}(Q)$ are pairwise disjoint compact subsets of the interior
of $Q$, so the IFS satisfies the strong separation condition.

Set
$a=\frac{24}{125}.$
Let $\Delta>0$ be the minimum distance between two distinct first-level
cylinders.  If two level-$n$ cylinders first separate after a common
prefix $w$ of length $j<n$, then
\[
\dist(f_w(K_1),f_w(K_2))
\geq
\alpha_2(A_w)\Delta
\geq
a^j\Delta
\geq
a^{n-1}\Delta.
\]
It follows by a standard packing argument that a ball of radius
comparable to $a^n$ meets only $O(1)$ level-$n$ cylinders.  Every such
cylinder has $\mu_A$-mass $6^{-n}$, and hence

$\mu_A(B(x,r))
\leq
Cr^{d_0},
\qquad
d_0
=
\frac{\log6}{\log(125/24)}
>1.08.$
It follows from  standard   estimates
\cite[Chapter~8]{mattila1999geometry} that\begin{equation}\label{eq:explicit-self-affine-energy}
I_s(\mu_A)<\infty
\qquad
\text{for every }0<s<d_0.
\end{equation}

Now
$d_0+\kappa_0>2.$
Thus we may choose
$2-\kappa_0<s<d_0;$
for instance $s=1.07$ works.  By
\eqref{eq:explicit-kappa} and
\eqref{eq:explicit-self-affine-energy},
$I_s(\mu_A)<\infty,
\,
s+\dimFr\nu_F>2.$
All the hypotheses of Theorem~\ref{thm:self-affine-intro} are therefore
satisfied, and we obtain
\[
(\pi_e)_\#\mu_A\ll\Leb^1
\qquad
\text{for every }e\in\Sone.
\]

Finally, $\mu_A$ is singular with respect to the Lebesgue measure on the plane; this follows directly from \eqref{eq:explicit-self-affine-dimension-range}.
\end{proof}

\section{AI disclosure}
ChatGPT (GPT-5.6 Sol) was used to assist with language
polishing, typographical corrections, and final consistency checks.
All mathematical results and proofs are the work of the authors,
who take full responsibility for the contents of this paper.
\bibliographystyle{plainurl}
\bibliography{bib}
\end{document}